\documentclass[11pt]{article}
\usepackage[a4paper,margin=0.9in]{geometry}
\usepackage[T1]{fontenc}
\usepackage{lmodern}
\usepackage{amsmath,amssymb,amsthm,mathtools}
\usepackage{booktabs,array,xcolor,mathtools}
\usepackage{graphicx,microtype,needspace}
\usepackage[colorlinks=true,allcolors=black]{hyperref}
\hypersetup{pdftitle={Isolated singularities of solutions to the Yamabe equation in higher dimensions II},pdfauthor={Zheng-Chao Han, Qinfeng Jiang, Hua-Yang Wang, Jingang Xiong, Lei Zhang}}
\allowdisplaybreaks[2]
\numberwithin{equation}{section}
\newtheorem{theorem}{Theorem}[section]
\newtheorem{proposition}[theorem]{Proposition}
\newtheorem{lemma}[theorem]{Lemma}
\theoremstyle{remark}
\newtheorem{remark}[theorem]{Remark}
\newcommand{\R}{\mathbb R}
\newcommand{\RP}{\mathbb{RP}}
\newcommand{\cJ}{\mathcal J}
\newcommand{\cR}{\mathcal R}
\newcommand{\cC}{\mathcal C}
\newcommand{\cK}{\mathcal K}
\newcommand{\cL}{\mathcal L}
\newcommand{\cH}{\mathcal H}
\newcommand{\cN}{\mathcal N}
\newcommand{\cX}{\mathcal X}
\newcommand{\dd}{\,\mathrm d}
\newcommand{\ip}[2]{\langle #1,#2\rangle}
\newcommand{\Scal}{\operatorname{Scal}}

\newcommand{\arxiv}[1]{\href{https://arxiv.org/abs/#1}{arXiv:\nolinkurl{#1}}}
\newcommand{\step}[1]{\par\medskip\noindent\textbf{#1}\enspace}
\newcommand{\Fop}{\mathcal F}

\usepackage[markup=underlined]{changes} 
\definechangesauthor[name={Z. C.}, color=orange]{Z.C}

\title{\bfseries Isolated singularities of solutions to the Yamabe equation
in higher dimensions II}
\author{Zheng-Chao Han,\quad Qinfeng Jiang,\quad
Hua-Yang Wang\thanks{H.-Y. Wang was partially supported by NSFC grant 12601183.}\\[2mm]
Jingang Xiong\thanks{J. Xiong was partially supported by NSFC grant 12325104.},
\quad Lei Zhang\thanks{Lei Zhang was partially supported by Simons Foundation
Grant SFI-MPS-TSM-00013752.}}
\date{}

\begin{document}
\maketitle

\begin{abstract}
This is the second of two papers on isolated singularities of positive
solutions to the Yamabe equation on manifolds with nonflat background
metrics. In the first paper, we proved that, in dimensions $3\le n\le29$,
every isolated singularity is either removable or the solution is
asymptotic to a radial model singular solution in the flat metric on the Euclidean space with origin as its singular point---a Fowler solution. In the present paper, we show that this
dimension range is sharp. More precisely, for each $n\ge30$, we construct
a smooth metric on $\mathbb{RP}^n$ and a positive solution of the
corresponding Yamabe equation having an isolated singularity, near which
the solution blows up along a sequence at a rate much faster than that
of Fowler solutions. Moreover, the singular solution defines a complete
conformal metric of constant positive scalar curvature and the metric agrees with the round metric to infinite order at the singular point.
\end{abstract}

\section{Introduction}\label{sec:introduction}

Let \(g_0\) be the standard quotient metric on \(\mathbb{RP}^n\), induced by the round metric on \(\mathbb{S}^n(1/2)\) under the antipodal identification. For a Riemannian metric \(g\), let \(\operatorname{Scal}_g\) denote its scalar curvature, and define the conformal Laplacian by
$$
L_g:=\Delta_g-c(n)\operatorname{Scal}_g,
\qquad
c(n):=\frac{n-2}{4(n-1)}.
$$
We write \(d_g\) for the distance induced by \(g\).

The main result of this paper is the following.

\begin{theorem}\label{thm:main}
For every integer \(n\ge 30\) and every point \(o\in\mathbb{RP}^n\), there exist a smooth Riemannian metric \(g_\infty\) on \(\mathbb{RP}^n\) agreeing with the round metric to infinite order at \(o\)  and a positive smooth function
$u_\infty $ on $\mathbb{RP}^n\setminus\{o\}$
satisfying the Yamabe equation
\begin{equation}\label{eq:yamabe}
-L_{g_\infty}u_\infty=n(n-2)u_\infty^{\frac{n+2}{n-2}}
\qquad\text{in }\mathbb{RP}^n\setminus\{o\},
\end{equation}
such that
\begin{equation}\label{eq:main-lower}
u_\infty(x)\ge
c d_{g_\infty}(x,o)^{-\frac{n-2}{2}}
\end{equation}
for some constant \(c>0\), and
\begin{equation}\label{eq:main-limsup}
\limsup_{x\to o}
d_{g_\infty}(x,o)^{n-2}u_\infty(x)=\infty.
\end{equation}
\end{theorem}

We make the following remarks.

\begin{itemize}

\item[-]  In Part~I~\cite{YamabePartI}, we completed the classification of isolated singularities of positive solutions to the Yamabe equation in punctured balls, for arbitrary smooth Riemannian metrics in dimensions \(3\le n\le 29\): every such singularity is either removable or the solution is asymptotic to a Fowler solution---a radial model singular solution in the flat metric on the Euclidean space with origin as its singular point. See also Marques~\cite{Marques2008}  in dimensions $3\le n\le5$ and
Xiong--Zhang~\cite{XiongZhang2022} in dimension $6$.  By \eqref{eq:main-limsup}, the above theorem shows that this classification cannot be extended to dimensions \(30\) and higher.

\item[-]  We will see in the proof of the main theorem that, for any \(\tau>0\),
\[
  \lim_{x\to 0}|x|^{-\tau}|g_\infty(x)-g_0(x)|=0
\]
in local coordinates centered at \(o\). This shows that the upper bound assumption in Han--Xiong--Zhang~\cite[Theorems~1.2 and 1.3]{HanXiongZhang2023} is necessary when \(n\ge 30\).

\item[-] By \eqref{eq:main-lower} and \eqref{eq:yamabe}, the conformal metric
\(u_\infty^{4/(n-2)}g_\infty\)
is complete and has scalar curvature \(4n(n-1)\). To the best of our knowledge, this is the first example of a
complete conformal metric of constant positive scalar curvature
whose conformal factor, relative to a smooth background metric,
grows faster than the Fowler rate along a sequence approaching
an isolated singularity. See
Schoen~\cite{Schoen1988}, Mazzeo--Smale~\cite{MazzeoSmale1991}, Mazzeo--Pollack--Uhlenbeck~\cite{MazzeoPollackUhlenbeck1996}, Mazzeo--Pacard~\cite{MazzeoPacard1999}, Byde~\cite{Byde2003}, Santos~\cite{Santos2010}, and others for the singular Yamabe problem. In a forthcoming paper, we  prove that if $(M,g)$ is a smooth,
connected, closed Riemannian manifold of dimension $n\ge3$,
has positive Yamabe constant, and is not conformally diffeomorphic
to the round sphere, then, for every nonempty finite set
$\{P_1,\ldots,P_m\}\subset M$, there exists a positive smooth
function $u$ on $M\setminus\{P_1,\ldots,P_m\}$ such that
$u^{4/(n-2)}g$ is complete and has constant positive scalar
curvature, and $u$ has Fowler asymptotics at each $P_i$.

\item[-]  For a fixed positive Yamabe conformal class on a closed manifold not
conformally diffeomorphic to the round sphere,
Khuri--Marques--Schoen~\cite{KMS2009} proved compactness through
dimension $24$, using the positive mass theorem in the final global
argument. In dimensions $n\ge52$, Brendle~\cite{Brendle2008} constructed smooth metrics on the sphere with a noncompact sequence of \emph{smooth} solutions, developing exactly one bubble;
Brendle--Marques~\cite{BrendleMarques2009} extended the construction to
$25\le n\le51$. For every $n\ge25$, Gong--Li~\cite{GongLi2025} constructed
sequences of positive solutions on the sphere for a fixed
smooth background metric, with both the number of bubbles
and the volumes of the associated conformal metrics
tending to infinity.
In their construction, the bubble centers approach a single
point, and the conformal factors converge to zero locally
uniformly away from that point.
\end{itemize}

\subsection*{Overview of the proofs}

Our construction proceeds by an iterative Lyapunov–Schmidt reduction, starting from a Fowler solution on the punctured round projective space. At each stage, we perturb the metric locally and construct a new exact solution with an additional bubble, while retaining Fowler asymptotics sufficiently close to the puncture. The insertion points approach the puncture. The metrics converge smoothly on the whole manifold, and the solutions converge smoothly on compact subsets away from the puncture. The limiting solution retains the peaks created at successive stages and violates the Fowler upper bound. Its behavior therefore differs from the Fowler asymptotics of the isolated ends in the constructions of Schoen~\cite{Schoen1988} and Mazzeo–Pacard~\cite{MazzeoPacard1999}. To accommodate the infinite energy of the Fowler background, we introduce a renormalized relative energy for the finite-dimensional reduction. We adapt the localized metric perturbation constructions of Brendle~\cite{Brendle2008} and Brendle–Marques~\cite{BrendleMarques2009} to successive backgrounds with Fowler asymptotics, preserving the nondegeneracy needed to continue the iteration. For a related infinite-dimensional gluing construction in the fractional Yamabe setting with the flat metric, we refer to Ao–DelaTorre–González–Wei~\cite{Ao+22}.

We now sketch our construction. 

Let $\gamma_0=|dx|^2$ on $B_1(0)\setminus\{0\}$, with the antipodal
boundary matching on $\partial B_1(0)$, and take a fixed Fowler solution $V_0(x)=|x|^{-(n-2)/2}v_*(-\log |x|)$ as the
initial solution for this flat background, where $v_*$ is a fixed non-constant periodic even positive Fowler
profile chosen in Section~\ref{sec:inductive}  with minimum  value $m_*$ (called the neck size) and is normalized so that $v_*(0)=m_*$. The initial conformal
metric is
\[
V_0^{4/(n-2)}\gamma_0.
\]
At stage $j$, after fixing a new
insertion ball, choose $\mu_j>0$ sufficiently small. Proposition~\ref{prop:one-bubble} perturbs $\gamma_{j-1}$ in that ball to a new metric $\gamma_j$ and adds a
bubble to the current solution $V_{j-1}$; the Lyapunov--Schmidt
correction makes the resulting $V_j$ an exact positive solution for
the perturbed metric $\gamma_j$. Thus the conformal metric after stage
$j$ is
\[
V_j^{4/(n-2)}\gamma_j,
\]
and $(\gamma_j,V_j)$ is the input for the next stage.

Each $V_j$ has the same Fowler profile and neck parameter at the
puncture as the initial solution; its periodic factor may be translated
by a phase. In the Euclidean coordinates $x$, the end condition
has the form
\[
V_j(x)=|x|^{-(n-2)/2}
\bigl(v_*(-\log |x|+T_j)+O(|x|^\delta)\bigr),
\qquad |x|\downarrow0,
\]
where  the remainder has the
weighted derivative bounds specified in \eqref{eq:background}. Proposition~\ref{prop:one-bubble}
gives small metric and solution changes outside the insertion
ball and a high peak inside it.

Proposition~\ref{prop:one-bubble} also proves that the augmented
linearization in \eqref{eq:extended-linearization} is an
isomorphism. Its unknown consists of a scalar phase increment and a
weighted function correction; the scalar changes the Fowler phase
near the puncture. This is the inverse estimate needed for the next
insertion.

The Fowler end has infinite energy, so we use a finite relative energy
to measure the change produced by inserting a bubble. With the insertion scale fixed, the Lyapunov--Schmidt reduction
first solves for the phase increment and the function correction,
and then determines the bubble parameters by finding a critical
point of the reduced energy. The main difficulty is to show that neither the interaction between the bubble and the Fowler background nor the change in energy caused by adjusting the Fowler phase and the function through the projected equation destroys the nondegenerate minimum of the leading energy. The dimension assumption $n\ge 30$ provides the
necessary order separation. At the resulting critical point, the
remaining finite-dimensional error vanishes, giving an exact
solution. Sections \ref{sec:starting}--\ref{sec:projected-energy} give the construction and estimates.

We choose each insertion scale after fixing the current background.
The metric increments are summable in every smooth norm, and the
solution increments are summable outside the new ball. A local
comparison transfers the boundary lower bound into that ball.
Consequently, the limiting solution satisfies the lower bound
\eqref{eq:main-lower}, which gives completeness. The peak condition
\eqref{eq:inductive-peak}, together with the summability of later
solution increments, yields \eqref{eq:main-limsup}.

Section~\ref{sec:inductive} defines the class of backgrounds, states the
insertion proposition, and deduces Theorem~\ref{thm:main}.
Section~\ref{sec:starting} proves the initial inverse estimate and develops
the nonlinear phase correction and the relative energy on punctured balls.
Section~\ref{sec:metric-insertion} constructs the metric perturbation and
solves the projected equation after adding a bubble.
Section~\ref{sec:projected-energy} chooses the bubble parameters by finding
a nondegenerate critical point of the reduced energy.
Section~\ref{sec:insertion-proof} proves the insertion proposition,
including invertibility of the full linearization at the new solution.
Appendix~\ref{app:coefficient-tables} proves the nondegeneracy of the leading bubble energy. Subsection~\ref{app:local-cylindrical} establishes a local solution theorem on arbitrary smooth backgrounds, and Subsection~\ref{app:quadratic-curvature} gives its second-order asymptotics; these local results hold in every dimension \(n\ge3\). Subsection~\ref{app:regularity} proves the regularity and weighted-convergence results.

\noindent\textbf{Constants and uniformity}
We fix $n\ge30$ in Sections~\ref{sec:inductive}--\ref{sec:insertion-proof}.
The constant $C>0$ in the estimates may change from line to line. During one insertion,
it may depend on the fixed dimension, Fowler profile, background,
weight, insertion point, radius, and cutoffs. It is independent of
$\mu$ and uniform on the stated compact parameter set. The inverse
norm at the new solution may depend on the chosen insertion scale.

\section{Bubble insertion and the proof of the main theorem}\label{sec:inductive}

In this section, we use the insertion result and a local lower bound
to prove Theorem~\ref{thm:main}. Lemma~\ref{lem:base-inverse} supplies
the initial background, Proposition~\ref{prop:one-bubble} preserves
its properties at each insertion, and Lemma~\ref{lem:local-lower}
provides the lower bound needed for completeness.

We first describe the punctured ball and its boundary identification.
Put $X=B_1(0)\setminus\{0\}\subset\R^n$.
We use fixed Euclidean coordinates $x$, with $r=|x|$, throughout
the construction. In these coordinates, the Euclidean metric $\gamma_0$
has components $(\gamma_0)_{ij}=\delta_{ij}$. The closed ball with antipodal boundary
points identified represents $\RP^n$; its center represents $o$.
The round metric of radius $1/2$ and its Euclidean conformal factor are
\begin{equation}\label{eq:round-chart}
g_0=\omega^{4/(n-2)}\gamma_0,
\qquad \omega(x)=(1+|x|^2)^{-(n-2)/2}.
\end{equation}
Our analytic unknowns are the Euclidean representatives
\begin{equation}\label{eq:compactification}
\gamma=\omega^{-4/(n-2)}g,\qquad V=\omega u.
\end{equation}
Thus $g=\omega^{4/(n-2)}\gamma$ and $u=\omega^{-1}V$.
In particular,
$u^{4/(n-2)}g=V^{4/(n-2)}\gamma$.
The conformal covariance formula is
\begin{equation}\label{eq:conformal-covariance}
L_{\Omega^{4/(n-2)}g}(\Omega^{-1}w)
=\Omega^{-(n+2)/(n-2)}L_gw.
\end{equation}
Under the change \eqref{eq:compactification}, the Yamabe equation becomes
$-L_\gamma V=n(n-2)V^p$ on $X$, where \(p:=\frac{n+2}{n-2}\).
All metrics $\gamma$ used at a finite stage are smooth on
$\overline{B_1}$, equal to $\gamma_0$ near $0$ and on a fixed collar of
$\partial B_1$, and satisfy $\dd V_\gamma=\dd x$.

The antipodal identification is expressed by the Kelvin extension
\[
V_{\rm ext}(r,\theta)=r^{2-n}V(r^{-1},-\theta),\qquad r\ge1.
\]
At $r=1$,
\[
\partial_rV_{\rm ext}(1,\theta)
=-(n-2)V(1,-\theta)-\partial_rV(1,-\theta).
\]
Matching values and radial derivatives therefore gives
\begin{equation}\label{eq:boundary-matching}
\begin{aligned}
V(1,\theta)&=V(1,-\theta),\\
(\partial_r+\frac{n-2}{2})V(1,\theta)
&=-(\partial_r+\frac{n-2}{2})V(1,-\theta).
\end{aligned}
\end{equation}
Equivalently, spherical harmonics of even degree satisfy the Robin
condition $V_\ell'(1)+\frac{n-2}{2}V_\ell(1)=0$, and harmonics of odd degree satisfy
$V_\ell(1)=0$.
Because $\gamma=\gamma_0$ in the collar, the Kelvin extension of a weak
solution satisfying \eqref{eq:boundary-matching} solves the same
equation across $r=1$. Elliptic regularity then gives smoothness of the
extension and of the corresponding solution on $\RP^n\setminus\{o\}$.
Smooth functions on the quotient are represented by functions whose
Kelvin extensions are smooth across $r=1$.

We next fix the Fowler profile and define the weighted spaces and phase variation.
The Fowler equation is
\begin{equation}\label{eq:fowler}
v''-\frac{(n-2)^2}{4}v+n(n-2)v^p=0.
\end{equation}
Let $A_n=((n-2)/(4n))^{(n-2)/4}$, fix $m_*\in(0,A_n)$,
and let $v_*=v_{m_*}$ be the solution with $v_*(0)=m_*$ and $v_*'(0)=0$.
The negative energy orbit lies in $v>0$ and is closed; hence
$v_*$ is positive, even, nonconstant, and periodic. Its energy is
\begin{equation}\label{eq:fixed-neck}
\frac12(v_*')^2-\frac{(n-2)^2}{8}v_*^2
+\frac{(n-2)^2}{2}v_*^{p+1}=-E_*<0,
\qquad \min v_*=m_*>0.
\end{equation}
We fix the neck parameter $m_*$. The phase $T$, taken modulo the
period of $v_*$, determines the radial profile
\[
V_T(x):=\Theta(x)v_*(-\log r+T),
\qquad \Theta(x)=|x|^{-(n-2)/2}.
\]
We also use the radial notation $V_T(r)$ and $\Theta (r)$. The initial solution for
$\gamma_0$ is $V_0(r)=\Theta (r) v_*(-\log r)$.
It satisfies \eqref{eq:boundary-matching}, since $v_*'(0)=0$.
Near $x=0$, $V_0(x)^{4/(n-2)}\gamma_0\ge m_*^{4/(n-2)}|x|^{-2}\gamma_0$;
after the boundary identification its conformal metric is complete.

Fix a scalar weight $0<\delta<1$ throughout the construction.
For $s\in\{-\delta,\delta\}$,
let $H^1_s(X)$ be the completion, in the following norm, of smooth
functions on $\overline{B_1}$ vanishing near $x=0$ whose boundary trace is
antipodally even:
\begin{equation}\label{eq:weighted-spaces}
\|w\|_{H^1_s}^2
:=\int_{B_1}|x|^{-2s}
       \bigl(|Dw|^2+|x|^{-2}|w|^2\bigr)\dd x,
\qquad H^{-1}_\delta:=(H^1_{-\delta})^*.
\end{equation}
At weight zero, the space is
\begin{equation}\label{eq:unweighted-even-space}
\begin{aligned}
H^1_{\mathrm{ev}}(B_1)
&:=\{w\in H^1(B_1):\operatorname{Tr}w(\theta)
                         =\operatorname{Tr}w(-\theta)\},\\
\|w\|_{H^1_{\mathrm{ev}}}^2
&:=\int_{B_1}(|Dw|^2+|x|^{-2}|w|^2)\dd x.
\end{aligned}
\end{equation}
The subscript $\mathrm{ev}$ refers only to the boundary trace.
Hardy's inequality with a boundary term and the trace inequality give
\[
\|w\|_{H^1(B_1)}^2
\le\|w\|_{H^1_{\mathrm{ev}}}^2
\le C\|w\|_{H^1(B_1)}^2.
\]
Thus the completion at weight zero is $H^1_{\mathrm{ev}}(B_1)$,
with an equivalent norm to the ordinary Sobolev norm on this subspace.
For $n\ge3$, the point $0$ has zero $H^1$-capacity. Consequently
restriction identifies ordinary $H^1(B_1)$ with ordinary Euclidean
$H^1(X)$, with equal ordinary norms. We use $H^1_0(D)$ for the space with zero boundary
trace on a domain $D$. Duality in \eqref{eq:weighted-spaces} uses $\dd x$.
For $1\le q\le\infty$ and $s\in\{-\delta,\delta\}$, set
$\|f\|_{L^q_s}:=\||x|^{-s}f\|_{L^q(B_1)}$.

Set $\mathcal A=B_2\setminus\overline{B_{1/2}}$. The weighted H\"older
norm, for $s\in\{-\delta,0,\delta\}$, is
\begin{equation}\label{eq:weighted-holder}
\|w\|_{\cC^{k,\alpha}_{\Theta;s}}
:=\sup_{0<\rho\le1/2}\rho^{\frac{n-2}{2}-s}
        \|w(\rho\,\cdot)\|_{C^{k,\alpha}(\overline{\mathcal A})}
 +\|w\|_{C^{k,\alpha}(\overline{B_1}\setminus B_{1/4})}.
\end{equation}
Restrictions near $0$ use only annuli contained in the indicated region.
For estimates on arbitrary subsets $D\subset X$, we use
\[
\|w\|_{\cC^k_{\Theta;s}(D)}
:=\sum_{|\beta|\le k}\sup_{x\in D}
 |x|^{|\beta|-s}\Theta(x)^{-1}|D^\beta w(x)|.
\]
Omitting the H\"older seminorms in \eqref{eq:weighted-holder} gives
an equivalent norm to $\cC^k_{\Theta;s}(X)$.
In particular, $\|w\|_{\cC^0_{\Theta;0}}=\sup_X|w|/\Theta$. Every
$w\in\cC^{k,\alpha}_{\Theta;\delta}$ satisfies
\[
|D^\beta w(x)|\le C\|w\|_{\cC^{k,\alpha}_{\Theta;\delta}}
                      |x|^{-\frac{n-2}{2}+\delta-|\beta|},\qquad |\beta|\le k.
\]
Thus $w/\Theta=O(|x|^\delta)$; the norm controls the function and its
derivatives relative to the Fowler scale $\Theta$.

The weak realization of $-L_\gamma$ uses the symmetric form
\begin{equation}\label{eq:weak-robin-form}
Q_\gamma(w,\varphi):=
\int_{B_1}\bigl(\ip{Dw}{D\varphi}_\gamma
                 +c(n)\Scal_\gamma w\varphi\bigr)\dd x
+\frac{n-2}{2}\int_{\partial B_1}w\varphi\dd S.
\end{equation}
The trace of each test function $\varphi$ is antipodally even. Hence
\[
\begin{aligned}
\int_{\partial B_1}(\partial_r+\frac{n-2}{2})w(\theta)\varphi(\theta)\dd S
=\frac12\int_{\partial B_1}
\left[(\partial_r+\frac{n-2}{2})w(\theta)+(\partial_r+\frac{n-2}{2})w(-\theta)\right]
\varphi(\theta)\dd S.
\end{aligned}
\]
Thus \eqref{eq:boundary-matching} is the natural boundary condition
of this weak form. The form is continuous, with
\begin{equation*}
|Q_\gamma(w,\varphi)|\le C_\gamma
       \|w\|_{H^1_\delta}\|\varphi\|_{H^1_{-\delta}}.
\end{equation*}
Indeed, uniform ellipticity controls the gradient term by
$C\||x|^{-\delta}Dw\|_2\||x|^\delta D\varphi\|_2$.
Bounded scalar curvature and $|x|\le 1$ control the term of order zero by
$C\||x|^{-1-\delta}w\|_2\||x|^{-1+\delta}\varphi\|_2$.
The trace inequality controls the boundary term on the fixed collar,
where the weights and their reciprocals are bounded.
Smooth functions satisfying both matching conditions and vanishing
near $0$ are dense in $H^1_\delta$: use the Kelvin identification in
a boundary collar, smooth there, and cut off near $0$ in the weighted norm.

For a real valued function $f$, write
$f_+:=\max\{f,0\}$ and $f_-:=\max\{-f,0\}$. We write the nonlinear residual as
\[
\Fop_\gamma(Y):=-L_\gamma Y-n(n-2)(Y_+)^p,
\]
For smooth test functions on $\overline{B_1}$ that vanish near $0$
and have antipodally even boundary traces, the weak formulation
uses $Q_\gamma$ for the linear term.
To vary the Fowler phase near the puncture, choose a smooth radial
cutoff $0\le\chi(r)\le1$ equal to one near
$0$ and zero near $r=1$. For the initial inverse, any such cutoff
is allowed. When a ball is chosen for a particular insertion, we will
choose the cutoff to vanish on a neighborhood of that ball as well.
The cutoff is held fixed throughout each insertion. Define
\begin{equation}\label{eq:phase-family}
\begin{aligned}
\Phi_{T,\tau}(r)&:=\chi(r)\bigl(V_{T+\tau}(r)-V_T(r)\bigr),\\
\psi_T(r)&:=\left.\partial_\tau\Phi_{T,\tau}(r)\right|_{\tau=0}
       =\chi(r)\Theta (r) v_*'(-\log r+T).
\end{aligned}
\end{equation}
Here $r=|x|$ is the radial variable, $T$ is the background phase,
and $\tau$ is its increment. The cutoff preserves the boundary matching
conditions, and $\partial_\tau\Phi_{T,\tau}=\psi_{T+\tau}\in H^1_{-\delta}\setminus H^1_\delta$.
Set
\[
\cX_\delta=\R\oplus H^1_\delta,
\qquad\|(\tau,w)\|_{\cX_\delta}=|\tau|+\|w\|_{H^1_\delta}.
\]
The spaces $H^1_{\pm\delta}$, $H^{-1}_\delta$, and $\cX_\delta$,
with their displayed norms, are fixed independently of $\chi$.
Changing $\chi$ changes the representation of a function as
$\tau\psi_T+w$. The exact coordinate change and the dependence of
the inverse bound are given in \eqref{eq:cutoff-linear-change}--
\eqref{eq:cutoff-inverse-bound} below.
Near $0$, the identity $V_T+\Phi_{T,\tau}=V_{T+\tau}$ preserves
the exact Fowler profile. The variable $\tau$ describes the phase
variation, while $w\in H^1_\delta$ describes the remaining correction.

We now specify the backgrounds used in the induction.
Let $\gamma$ satisfy the metric conditions stated after
\eqref{eq:conformal-covariance}, and let $V>0$ be smooth on
$\overline{B_1}\setminus\{0\}$, satisfy \eqref{eq:boundary-matching},
and obey
\begin{equation}\label{eq:background}
-L_\gamma V=n(n-2)V^p\quad\text{in }X,\qquad
V=V_T+w_0\quad\text{near }0,\qquad
w_0\in H^1_\delta\cap\cC^{k,\alpha}_{\Theta;\delta}
\end{equation}
for every integer $k\ge0$ and every $0<\alpha<1$.
The membership \(w_0\in H^1_\delta\cap\mathcal C^{k,\alpha}_{\Theta;\delta}\) is understood after extending \(w_0\) smoothly to \(X\), keeping it unchanged near \(0\) and making it vanish near \(\partial B_1\), for example by using a suitable smooth cutoff.

Set $\cL_{\gamma,V}:=-L_\gamma-n(n+2)V^{p-1}$ and define
\begin{equation}\label{eq:extended-linearization}
\widehat\cL_{\gamma,V}:\cX_\delta\to H^{-1}_\delta,\qquad
\widehat\cL_{\gamma,V}(\tau,w)
:=\cL_{\gamma,V}(\tau\psi_T+w),
\end{equation}
where $\psi_T$ is from \eqref{eq:phase-family}.
For $(\tau,w)\in\cX_\delta$ and $\varphi\in H^1_{-\delta}$,
the weak formulation is
\[
\begin{aligned}
\langle\widehat{\mathcal L}_{\gamma,V}(\tau,w),\varphi\rangle
={}&Q_{\gamma}(w,\varphi)
-n(n+2)\int_XV^{p-1}w\varphi\,dx
+\tau\int_X
\bigl(\mathcal L_{\gamma,V}\psi_T\bigr)\varphi\,dx.
\end{aligned}
\]

Since \(\psi_T=\chi\partial_TV_T\), \(\chi\equiv1\) and \(\gamma\) is Euclidean near \(0\), the Fowler equation gives
\[
\cL_{\gamma,V}\psi_T
=n(n+2)(V_T^{p-1}-V^{p-1})\partial_TV_T\quad \text{near \(0\)}.
\]
Hence $|\cL_{\gamma,V}\psi_T|\le C|x|^{-2}|w_0|$ there.
Together with $w_0\in H^1_\delta$ and smoothness away from $0$,
this implies $\cL_{\gamma,V}\psi_T\in H^{-1}_\delta$.

We call a pair \((\gamma,V)\) admissible if \(\gamma\) satisfies the metric conditions stated after \eqref{eq:conformal-covariance}, \(V>0\) is smooth on \(\overline{B_1}\setminus\{0\}\) and satisfies \eqref{eq:boundary-matching} and \eqref{eq:background}, and the map \(\widehat{\mathcal L}_{\gamma,V}\) defined in \eqref{eq:extended-linearization} is an isomorphism. The last requirement is the nondegeneracy condition for the pair.

The following inverse provides the initial background for the construction.
\begin{lemma}\label{lem:base-inverse}
At $V_0(r)=\Theta (r) v_*(-\log r)$, the map $\widehat\cL_{\gamma_0,V_0}: \cX_\delta\to H^{-1}_\delta$ 
is an isomorphism. Its inverse bound depends on $n,v_*,\delta$
and the cutoff.
\end{lemma}
The proof is given in Section~\ref{sec:starting}.

By \eqref{eq:background} and the smoothness and positivity of $V$
away from $0$, the quotient $V/\Theta$ has positive upper and lower
bounds on $X$, and
$\|V\|_{\cC^k_{\Theta;0}}<\infty$ for every fixed $k\ge0$.
By Lemma~\ref{lem:base-inverse}, the initial pair
$(\gamma_0,V_0)$ is admissible.
Proposition~\ref{prop:one-bubble} preserves admissibility
at each insertion.

\noindent\textbf{Change of phase cutoff.}
Let $(\gamma,V)$ be admissible, with phase $T$ in
\eqref{eq:background}.  For two smooth radial cutoffs $\chi$ and
$\widetilde\chi$, each equal to one near $0$ and zero near $r=1$, write
\[
\psi_T^\chi=\chi\partial_TV_T,\qquad
\Phi_{T,\tau}^\chi(r)=\chi(r)\bigl(V_{T+\tau}(r)-V_T(r)\bigr),\qquad
\widehat\cL^\chi_{\gamma,V}(\tau,w)
=\cL_{\gamma,V}(\tau\psi_T^\chi+w).
\]
Set $h_T=(\chi-\widetilde\chi)\partial_TV_T\in H^1_\delta(X)$.  Then
\begin{equation}\label{eq:cutoff-linear-change}
S_{\chi,\widetilde\chi}:\cX_\delta\to\cX_\delta,
\qquad
S_{\chi,\widetilde\chi}(\tau,w)=(\tau,w+\tau h_T)
\end{equation}
is a bounded linear isomorphism.  Since
\[
\widehat\cL^{\widetilde\chi}_{\gamma,V}
S_{\chi,\widetilde\chi}=\widehat\cL^\chi_{\gamma,V},
\qquad
\|S_{\chi,\widetilde\chi}\|
\le 1+\|h_T\|_{H^1_\delta},
\]
invertibility for one cutoff is equivalent to invertibility for the other,
and
\begin{equation}\label{eq:cutoff-inverse-bound}
\| (\widehat\cL^{\widetilde\chi}_{\gamma,V})^{-1}\|_{H^{-1}_\delta\to\cX_\delta}
\le (1+\|h_T\|_{H^1_\delta})
             \|(\widehat\cL^\chi_{\gamma,V})^{-1}\|_{H^{-1}_\delta\to\cX_\delta}.
\end{equation}
At the nonlinear level, the change
\begin{equation}\label{eq:cutoff-nonlinear-change}
\widetilde w=w+(\chi-\widetilde\chi)(V_{T+\tau}-V_T),
\qquad
V+\Phi_{T,\tau}^\chi+w
=V+\Phi_{T,\tau}^{\widetilde\chi}+\widetilde w
\end{equation}
preserves the represented function. The comparison constants depend on
the two cutoffs. We suppress the cutoff superscript after this comparison.

Choose real phase representatives near $T$. All thresholds and constants refer to the
fixed background and insertion ball.

\begin{proposition}\label{prop:one-bubble}
Let $(\gamma,V)$ be admissible, with phase $T$ in
\eqref{eq:background}, and let $B_{4\rho}(a)$ have compact closure
in $B_{1/2}\setminus\{0\}$, with $\gamma=\gamma_0$ on this ball.
For every sufficiently small $\mu>0$, there is an admissible pair
$(\gamma_\mu,Y_\mu)$ with Fowler phase
$\widehat T_\mu$ and the same neck parameter $m_*$.
The metric change from $\gamma$ to $\gamma_{\mu}$ is supported in $B_\rho(a)$,
$\dd V_{\gamma_\mu}=\dd x$, and, for every fixed $k\ge0$,
\begin{equation}\label{eq:insertion-data}
\|\gamma_\mu-\gamma\|_{C^k(\overline{B_1})}\to0,\qquad
\widehat T_\mu\to T
\quad\text{as }\mu\downarrow0.
\end{equation}
The full inverse satisfies
\begin{equation}\label{eq:insertion-inverse}
\|\widehat\cL_{\gamma_\mu,Y_\mu}^{-1}f\|_{\cX_\delta}
\le C_\mu\|f\|_{H^{-1}_\delta},\qquad C_\mu<\infty,
\end{equation}
for every \(f\in H^{-1}_\delta\). For every fixed $k\ge0$,
\begin{equation}\label{eq:outside-convergence}
\|Y_\mu-V\|_{\cC^k_{\Theta;0}(X\setminus B_\rho(a))}\to0
\quad\text{as }\mu\downarrow0.
\end{equation}
There are $c,C>0$, independent of small $\mu$, and $x_\mu\in B_{C\mu}(a)$
such that
\begin{equation}\label{eq:peak}
Y_\mu(x_\mu)\ge c\mu^{-(n-2)/2}.
\end{equation}
\end{proposition}

To prove Proposition~\ref{prop:one-bubble}, we first solve a projected equation for the phase and function correction, keeping the bubble’s scale and center fixed. At this intermediate stage, \(n+1\) residual coefficients may still be nonzero. We then choose the scale and center at a nondegenerate critical point of the reduced energy, which makes these coefficients vanish and yields the exact solution \(Y_\mu\). Nondegeneracy of this critical point, together with the projected inverse, also gives invertibility of the augmented linearization at \((\gamma_\mu,Y_\mu)\). Section~\ref{sec:insertion-proof} completes the proof.

\begin{remark}\label{rem:not-supported}
The metric change is supported in $B_\rho(a)$, and the solution
change outside $B_\rho(a)$ satisfies \eqref{eq:outside-convergence}.
Lemma~\ref{lem:local-lower} supplies the lower bound inside
$B_{2\rho}(a)$ used in the induction.
\end{remark}

We now transfer boundary lower bounds into the insertion ball and carry out the induction.
For balls $D$ whose radii are at most one sixteenth of the distance of
their centers to $0$, the function $\Theta=|x|^{-(n-2)/2}$ satisfies
\begin{equation}\label{eq:theta-ratio}
\sup_D\Theta\le C_n\inf_D\Theta
\end{equation}
with a constant $C_n$ depending only on $n$.

\begin{lemma}\label{lem:local-lower}
There exist a $C^2(\overline{B_1})$ neighborhood $\mathcal U_0$ of
$\gamma_0$ and $r_0>0$ such that, if $\gamma\in\mathcal U_0$,
$D=B_R(a)$ has compact closure in $B_1$, $R\le r_0$, and $V>0$ solves
$-L_\gamma V=n(n-2)V^p$ near $\overline D$, then
\begin{equation}\label{eq:local-minimum}
\inf_D V\ge\frac12\inf_{\partial D}V.
\end{equation}
If, in addition, \(R\le |a|/16\)  and $V/\Theta\ge m_*/2$ on
$\partial D$, then $V/\Theta\ge m_*/(4C_n)$ on $D$.
\end{lemma}

\begin{proof}
In a fixed small neighborhood of $\gamma_0$, the metrics are uniformly
elliptic, their coefficients have uniform $C^1$ bounds, and
$|\Scal_\gamma|\le C$.
After fixing a sufficiently small neighborhood \(\mathcal U_0\), there exist \(r_0,c_0>0\) such that \(\Delta_\gamma|x-a|^2\ge c_0\) throughout every ball \(B_R(a)\Subset B_1\) with \(R\le r_0\), for every \(\gamma\in\mathcal U_0\). Choose $C_0$ so that
$C_0c_0\ge c(n)\sup|\Scal_\gamma|$, and require $C_0r_0^2\le1/2$.
Then
\[
b(x):=1-C_0(R^2-|x-a|^2),\qquad
\tfrac12\le b\le1,\qquad -L_\gamma b\le0,
\qquad b|_{\partial D}=1.
\]
After decreasing \(r_0\), the Dirichlet form of \(-L_\gamma\) on \(D\) 
\[
\int_D
\left(
|D\varphi|_\gamma^2+
c(n)\operatorname{Scal}_\gamma\varphi^2
\right)dV_\gamma,
\]
is uniformly coercive for every \(\gamma\in\mathcal U_0\) and \(R\le r_0\): 
\[
c(n)\int_D(\operatorname{Scal}_\gamma)_-\varphi^2\,dV_\gamma
\le CR^2\int_D|D\varphi|_\gamma^2\,dV_\gamma.
\quad \varphi\in H^1_0(D).
\] 
Uniform ellipticity and Poincaré’s inequality justify this estimate. Choose $r_0$ so that $Cr_0^2\le1/2$.
Then the Dirichlet form is uniformly coercive for every
$\gamma\in\mathcal U_0$ and $R\le r_0$. Set $m=\inf_{\partial D}V$. Since $V-mb\ge0$ on $\partial D$
and $-L_\gamma(V-mb)\ge0$ in $D$, testing with $(V-mb)_-$
and using coercivity gives $V\ge mb\ge m/2$. Dividing by $\Theta$ proves the final assertion.
\end{proof}

\begin{proof}[Proof of Theorem~\ref{thm:main}]
\step{Step 1. The balls and the initial pair.}
Fix a unit vector $e$ and a sufficiently small $r_1>0$. Set
\[
r_j=4^{1-j}r_1,\qquad a_j=r_je,\qquad
\rho_j=r_j/64,\qquad D_j=B_{2\rho_j}(a_j),
\qquad c_*:=\frac{m_*}{16C_n},\quad d_j=c_*2^{-j-4}.
\]
The balls $B_{4\rho_j}(a_j)$ are pairwise disjoint in $B_{1/2}\setminus\{0\}$.
Fix a smooth function $\chi_*:[0,\infty)\to[0,1]$ equal to one on
$[0,1]$ and zero on $[2,\infty)$, and set
\begin{equation}\label{eq:stage-cutoff}
\chi_j(r):=\chi_*(4r/r_j).
\end{equation}
Thus $\chi_j=1$ for $r\le r_j/4$ and $\chi_j=0$ for $r\ge r_j/2$.
Since $|a_j|-4\rho_j=15r_j/16$, $\chi_j$ vanishes on a neighborhood of
$\overline{B_{4\rho_j}(a_j)}$.
Start with the admissible pair $(\gamma_0,V_0)$ from
Lemma~\ref{lem:base-inverse}, using $\chi_1$.

\step{Step 2. Choosing each scale after its background and cutoff.}
Suppose $(\gamma_{j-1},V_{j-1})$ has been constructed. The current
metric is still $\gamma_0$ on $B_{4\rho_j}(a_j)$. For $j\ge2$,
first replace the preceding cutoff $\chi_{j-1}$ by $\chi_j$ in the
representation of variations at this fixed background. Equation
\eqref{eq:cutoff-inverse-bound}, with $T=T_{j-1}$, gives a finite
inverse bound for the new cutoff.
The change of cutoff is controlled by
\[
\| (\chi_{j-1}-\chi_j)\partial_TV_{T_{j-1}}\|_{H^1_\delta}^2
\le C\int_{r_j/4}^{2r_j}r^{-1-2\delta}\dd r
\le C r_j^{-2\delta}.
\]
Indeed $\chi_{j-1}(r)=\chi_*(r/r_j)$,
$|\partial_TV_T|\le Cr^{-(n-2)/2}$, and
$|D\partial_TV_T|\le Cr^{-n/2}$, uniformly in $T$.
On the cutoff transition annuli, $|D\chi_i|\le C/r$.
The constant in this display is independent of $j$ and the phase.
The resulting full inverse bound may grow with $j$.

Fix $(\gamma_{j-1},V_{j-1})$, the ball $B_{4\rho_j}(a_j)$,
the phase cutoff $\chi_j$,  the cutoff
used in \eqref{eq:metric-profile} to localize the metric perturbation,
and the bubble cutoff used in \eqref{eq:cutoff-bubble}, with
$(a,\rho)=(a_j,\rho_j)$. Also fix a finite bound for
\[
\bigl(\widehat\cL^{\chi_j}_{\gamma_{j-1},V_{j-1}}\bigr)^{-1}
:H^{-1}_\delta\longrightarrow\cX_\delta,
\]
provided by \eqref{eq:cutoff-inverse-bound} for $j\ge2$ and by Lemma~\ref{lem:base-inverse} for $j=1$.
All these choices are made before choosing $\mu_j$. The constants and smallness thresholds
in Lemmas~\ref{lem:background-continuation} and \ref{lem:projected-inverse}
and Proposition~\ref{prop:projected-solution} are then fixed before
choosing $\mu_j$. In Proposition~\ref{prop:reduced-minimum}, shrink the
parameter neighborhood to absorb the $C|b|^2$ term in the normalized
Hessian estimate, retaining the constants obtained on the initial
compact parameter set. The radius of this neighborhood may depend on
the stage. After these choices, apply
Proposition~\ref{prop:one-bubble} to
$(\gamma,V)=(\gamma_{j-1},V_{j-1})$.
Choose $\mu_j$ sufficiently small and denote the final output at
$\mu=\mu_j$ by $(\gamma_j,V_j)$, its Fowler phase by $T_j$,
and its selected point by $x_j$, such that
\begin{align}\label{eq:inductive-metric}
\|\gamma_j-\gamma_{j-1}\|_{C^m(\overline{B_1})}
&\le2^{-j}r_j^j,\qquad 0\le m\le j+2,\\
\|V_j-V_{j-1}\|_{\cC^m_{\Theta;0}(X\setminus D_j)}
&\le d_j,\qquad 0\le m\le j,\label{eq:inductive-solution}\\
|T_j-T_{j-1}|&\le2^{-j},\label{eq:inductive-phase}\\
|x_j-a_j|&\le\rho_j/4,
\qquad r_j^{n-2}V_j(x_j)\ge2^{2j}.\label{eq:inductive-peak}
\end{align}
These finitely many conditions follow from
\eqref{eq:insertion-data}, \eqref{eq:outside-convergence}, and
\eqref{eq:peak} by decreasing $\mu_j$ after fixing the current data.
The estimates concern actual metric and solution differences in fixed
norms. The new pair is admissible, and
\eqref{eq:cutoff-inverse-bound} supplies the inverse for the next cutoff.
The small phase increments determine consistent real representatives
of $T_j$.
The restrictions imposed on \(r_1\) in Step 1, together with \eqref{eq:inductive-metric}, keep every \(\gamma_j\) in \(\mathcal U_0\) and ensure uniform comparability with \(\gamma_0\)

\step{Step 3. The uniform lower bound.}

Every point of $\partial D_j$ lies outside the other insertion balls.
By \eqref{eq:inductive-solution} with $m=0$, we have
\[
\frac{V_j}{\Theta}\ge v_*(-\log r)-\sum_{i=1}^j d_i
\ge m_*/2\quad\text{on }\partial D_j.
\]
Lemma~\ref{lem:local-lower} therefore yields
\begin{equation*}
\frac{V_j}{\Theta}\ge\frac{m_*}{4C_n}=4c_*
\quad\text{on }D_j.
\end{equation*}
At each later stage $i>j$, this quotient changes on $D_j$ by at most
$d_i$. Since $\sum_i d_i=c_*/16$ and $m_*=16C_n c_*$,
\[
\begin{aligned}
\left.\frac{V_\ell}{\Theta}\right|_{D_j}
&\ge4c_*-\sum_{i=j+1}^\ell d_i>c_*\quad(\ell\ge j),\\
\left.\frac{V_\ell}{\Theta}\right|_{X\setminus\bigcup_{i=1}^\ell D_i}
&\ge m_*-\sum_{i=1}^\ell d_i\ge c_*.
\end{aligned}
\]
This proves the same bound at every finite stage and for the limit:
\begin{equation}\label{eq:global-cyl-lower}
V_\infty(x)\ge c_*\Theta(x)=c_*|x|^{-(n-2)/2}.
\end{equation}
\step{Step 4. The smooth limiting metric.}
For each fixed $m$, \eqref{eq:inductive-metric} is summable in $C^m$.
Thus $\gamma_j\to\gamma_\infty$ smoothly on $\overline{B_1}$.
The limit is positive definite, has Euclidean volume, and equals
$\gamma_0$ on the fixed boundary collar. Every tensor increment \(\gamma_j-\gamma_{j-1}\) vanishes near $0$,
so uniform convergence of each derivative series gives
\begin{equation*}
D^\beta(\gamma_\infty-\gamma_0)(0)
=\sum_{j\ge1}D^\beta(\gamma_j-\gamma_{j-1})(0)=0
\quad\text{for every multiindex }\beta.
\end{equation*}
Consequently $g_\infty=\omega^{4/(n-2)}\gamma_\infty$ descends to a
smooth metric on $\RP^n$ and agrees with $g_0$ to infinite order at $o$.

\step{Step 5. The limiting solution.}
On any compact subset of $\overline{B_1}\setminus\{0\}$, the norms
$\cC^m_{\Theta;0}$ and ordinary $C^m$ are equivalent, and all sufficiently late
balls $D_j$ are disjoint from that set. Summing
\eqref{eq:inductive-solution} gives smooth convergence to $V_\infty$,
including on the boundary collar. The equation and both boundary
conditions pass to the limit. Therefore
\[
-L_{\gamma_\infty}V_\infty=n(n-2)V_\infty^p,
\qquad u_\infty:=\omega^{-1}V_\infty>0
\]
defines a smooth solution on $\RP^n\setminus\{o\}$ for $g_\infty$.

\step{Step 6. Completeness.}
Near $0$, $\omega$ is bounded above and below, and
$d_{g_\infty}(x,o)\asymp r$. Thus \eqref{eq:global-cyl-lower} gives
\eqref{eq:main-lower}. Moreover,
\[
u_\infty^{4/(n-2)}g_\infty
=V_\infty^{4/(n-2)}\gamma_\infty
\ge C^{-1}c_*^{4/(n-2)}r^{-2}\gamma_0.
\]
Every curve approaching $0$ has infinite length, since its length is
bounded below by a positive multiple of $\int|\dd r|/r$.
After the boundary identification, $\mathbb{RP}^n\setminus\{o\}$
has no boundary, and the complement of a small neighborhood of $o$
is compact. Hence the conformal metric is complete.

\step{Step 7. The retained peaks.}
The point $x_j$ lies outside every later insertion ball, so
\[
|V_\infty(x_j)-V_j(x_j)|
\le\Theta(x_j)\sum_{i>j}d_i
\le C r_j^{-(n-2)/2}2^{-j}.
\]
By \eqref{eq:inductive-peak}, for large $j$,
$V_\infty(x_j)\ge\tfrac12 V_j(x_j)\ge2^{2j-1}r_j^{2-n}$.
Since $u_\infty=\omega^{-1}V_\infty$ and
$d_{g_\infty}(x_j,o)\asymp r_j$,
\[
d_{g_\infty}(x_j,o)^{n-2}u_\infty(x_j)\ge c2^{2j}\to\infty,
\qquad
d_{g_\infty}(x_j,o)^{(n-2)/2}u_\infty(x_j)
\ge c2^{2j}r_j^{-(n-2)/2}\to\infty.
\]
The first estimate proves \eqref{eq:main-limsup}.
The second excludes Fowler asymptotics.
\end{proof}

\begin{remark}\label{rem:limit-space}
Every finite stage has a Fowler asymptote and the phases converge.
The limit satisfies $V_\infty(x_j)/\Theta(x_j)\to\infty$ along the
selected peaks, which excludes asymptotics to a fixed Fowler profile.
\end{remark}

\section{Linear and nonlinear estimates for Fowler backgrounds}\label{sec:starting}

In this section, we prove the initial inverse in Lemma~\ref{lem:base-inverse}
and establish the nonlinear estimates and relative energy used in the
insertion. Lemmas~\ref{lem:nonlinear-map} and \ref{lem:relative-energy}
provide the differentiability needed for the reduction.

We first prove the inverse by separating spherical harmonics, using
$t=-\log r$ in this proof.

\begin{proof}[Proof of Lemma~\ref{lem:base-inverse}]
Let $\sigma$ be the unit round metric on $S^{n-1}$. In this proof,
put $x=e^{-t}\theta$, $t\ge0$, and normalize
the unknown and a source function by
\[
\bar w(t,\theta)=r^{(n-2)/2}w(r\theta),\qquad
\bar f(t,\theta)=r^{(n+2)/2}f(r\theta),\qquad r=e^{-t}.
\]

For distributions the source transformation is defined by duality.
Direct differentiation and polar integration give
\[
\begin{aligned}
r^{(n+2)/2}\cL_{\gamma_0,V_0}(r^{-(n-2)/2}\bar w)
&=\left[-\partial_t^2-\Delta_\sigma+\frac{(n-2)^2}{4}
             -n(n+2)v_*(t)^{p-1}\right]\bar w,\\
\|w\|_{H^1_{\pm\delta}(X)}
&\asymp\|e^{\pm\delta t}\bar w\|_{H^1([0,\infty)\times S^{n-1})},\\
\ip{f}{w}&=\ip{\bar f}{\bar w}_{\dd t\dd\theta}.
\end{aligned}
\]
The norm equivalence follows by expanding
$\partial_t(r^{(n-2)/2}w)=-r^{(n-2)/2}(r\partial_rw+\frac{n-2}{2}w)$;
the terms without derivatives control the lower order contributions
in both directions.
For smooth functions vanishing near the puncture, the radial cross
term cancels the Robin term:
\[
Q_{\gamma_0}(w,w)
=\int_0^\infty\!\int_{S^{n-1}}
\bigl(|\partial_t\bar w|^2+|\nabla_\sigma\bar w|^2
      +\tfrac{(n-2)^2}{4}|\bar w|^2\bigr)\dd\theta\dd t.
\]

Thus the weak form becomes the usual form on the half cylinder.
The transformed endpoint condition is Neumann for harmonics of even
degree and Dirichlet for harmonics of odd degree.
Reflect coefficients to the full line with their prescribed parity.
In the ODE calculations below, \(w_\ell\) and \(f\) denote spherical-harmonic coefficients of the normalized unknown and source. On the half-line, we use the weighted Sobolev norms
\[
\|u\|_{H_s^1(0,\infty)}
:=\|e^{st}u\|_{H^1(0,\infty)},
\qquad s\in\{-\delta,\delta\}.
\]
The source norms are dual to the corresponding \(H^1_{-\delta}\) test spaces. For \(t\ge0\), set
\[
\widetilde\chi(t):=\chi(e^{-t}),\qquad
\bar\psi_0(t)
:=e^{-(n-2)t/2}\psi_0(e^{-t})
=\widetilde\chi(t)v_*'(t).
\]
Thus \(\bar\psi_0\) vanishes near \(t=0\) and equals \(v_*'\) for all sufficiently large \(t\). 

We treat the radial mode, the harmonics of degree one, and the
higher harmonics separately. The radial equation determines the phase
variable. A spherical harmonic of degree $\ell$ has
eigenvalue $\lambda_\ell=\ell(\ell+n-2)$. The reflected coefficients
have parity $w_\ell(-t)=(-1)^\ell w_\ell(t)$.

\step{Step 1. The radial mode.}
Write
\[
\cL_{\rm rad}:=-\frac{\dd^2}{\dd t^2}+\frac{(n-2)^2}{4}
                  -n(n+2)v_*(t)^{p-1}.
\]

For scalar functions set $\mathcal W(f,g):=fg'-f'g$.
Let $y=v_*'$ and
$v_1=\partial_\eta v_\eta|_{\eta=m_*}$, with the profiles normalized at
their minima. The period is smooth in $\eta$ near $m_*$: smooth
dependence for the ODE and the condition $v_\eta''>0$ at each minimum
allow the implicit function theorem to be applied to the next return.
Both derivatives solve the radial homogeneous equation and satisfy
\[
y(0)=0,\quad v_1(0)=1,\quad v_1'(0)=0,\quad
\mathcal W(v_1,y)=v_*''(0)\ne0.
\]
Let $\mathcal P(\eta)$ denote the period. Differentiating
$v_\eta(t+\mathcal P(\eta))=v_\eta(t)$ with respect to the neck parameter,
we obtain
$v_1(t+\mathcal P(m_*))=v_1(t)-\mathcal P'(m_*)y(t)$. Consequently
\[
v_1(t)=v_{1,\rm per}(t)+\kappa t\,y(t)
\]
for a periodic smooth function $v_{1,\rm per}$ and $\kappa\in\R$.

On $t\ge0$, use the solution $v_1$ at the Neumann endpoint $t=0$ and the
solution $y$ at infinity. For smooth rapidly decaying $f$, applying variation of
constants, we obtain
\[
h(t)=-\frac{y(t)}{\mathcal W_0}\int_0^t v_1(s)f(s)\dd s
      -\frac{v_1(t)}{\mathcal W_0}\int_t^\infty y(s)f(s)\dd s,
\qquad \mathcal W_0:=\mathcal W(v_1,y).
\]
The function $h$ satisfies
\[
-h''(t)+\left\{\frac{(n-2)^2}{4}-n(n+2)v_*(t)^{p-1}\right\}h(t)=f(t),
\qquad h'(0)=0.
\]
Its phase coefficient is
$\tau=-\mathcal W_0^{-1}\int_0^\infty v_1(s)f(s)\dd s$.
Use the decomposition
\[
h=\tau\bar\psi_0+w,\qquad
w=h-\tau\widetilde\chi y=(h-\tau y)+\tau(1-\widetilde\chi)y,
\]
where $\bar\psi_0$ is defined before Step 1 of  this proof.
Since $\widetilde\chi=0$ near $0$, one has $w'(0)=h'(0)=0$; since $\widetilde\chi=1$
on the tail, $w=h-\tau y$ there. Thus $w$ has both the required parity
and the required decay. The even extension of $\bar\psi_0$ to the full line is
$\widetilde\chi(|t|)y(|t|)$, smooth because it vanishes near $t=0$.

Equivalently, one may first determine $\tau$ and then solve
\begin{equation*}
\cL_{\rm rad}w=f-\tau\cL_{\rm rad}\bar\psi_0,
\qquad w'(0)=0.
\end{equation*}
Indeed, $\cL_{\rm rad}\bar\psi_0=-\widetilde\chi''y-2\widetilde\chi'y'$ is supported in the
transition of $\widetilde\chi$, and integration by parts gives
\[
\int_0^\infty v_1\cL_{\rm rad}\bar\psi_0\dd t=-\mathcal W_0,
\qquad
\int_0^\infty v_1(f-\tau\cL_{\rm rad}\bar\psi_0)\dd t=0.
\]
The second identity is the compatibility condition for a decaying
radial solution with the Neumann condition at $0$. For weak sources,
it is interpreted in the weighted duality below.

After subtracting $\tau y(t)$, the integral representing the decaying
part involves only $s\ge t$. On this region, set
\[
K(t,s)=\mathcal W_0^{-1}\{y(t)v_1(s)-v_1(t)y(s)\}.
\]
Substituting the decomposition of $v_1$, we obtain
\[
\mathcal W_0 K(t,s)=y(t)v_{1,\rm per}(s)-v_{1,\rm per}(t)y(s)
             +\kappa(s-t)y(t)y(s).
\]
Consequently $K$, $\partial_tK$, $\partial_sK$, and
$\partial_t\partial_sK$ are bounded by $C(1+s-t)$. The diagonal identities are
$K(t,t)=0$ and $\partial_sK(t,t)=-1$.

For weak data, the radial test space on $[0,\infty)$ is
$H^1_{-\delta}$ with no zero trace condition at $0$. Its dual elements
admit representations
\begin{equation}\label{eq:radial-weak-data}
\ip{f}{\varphi}
 =\int_0^\infty(f_0\varphi-f_1\varphi')\dd t,
 \qquad e^{\delta t}f_0,e^{\delta t}f_1\in L^2(0,\infty).
\end{equation}
The infimum of
$\|e^{\delta t}f_0\|_2+\|e^{\delta t}f_1\|_2$ over such
representations is equivalent to $\|f\|_{H^{-1}_\delta}$.
Indeed, the weighted norm on the test space is equivalent to the norm
of $(e^{-\delta t}\varphi,e^{-\delta t}\varphi')$ in $L^2\oplus L^2$.
Applying Hahn--Banach and the Riesz representation theorem proves
\eqref{eq:radial-weak-data}. The notation $f=f_0+\partial_t f_1$ below
always means this duality, including at the endpoint.

For these data define
\[
\begin{split}
\tau&=-\mathcal W_0^{-1}\int_0^\infty(v_1f_0-v_1'f_1)\dd t,\\
h(t)-\tau y(t)&=\int_t^\infty
       \{K(t,s)f_0(s)-\partial_sK(t,s)f_1(s)\}\dd s.
\end{split}
\]
The formula agrees with the smooth Green formula above, interpreted by
\eqref{eq:radial-weak-data}. In the weak sense,
\[
(h-\tau y)'(t)=\int_t^\infty
 \{\partial_tK(t,s)f_0(s)-\partial_t\partial_sK(t,s)f_1(s)\}\dd s-f_1(t).
\]
The last term is the diagonal contribution from
$\partial_sK(t,t)=-1$. Since
$(1+s)e^{-\delta s}\in L^1(0,\infty)$, we can apply the weighted
convolution estimate on $[0,\infty)$ to obtain
\begin{equation}\label{eq:radial-weak-estimate}
|\tau|+\|h-\tau\bar\psi_0\|_{H^1_\delta(0,\infty)}
 \le C\bigl(\|e^{\delta t}f_0\|_2+
                 \|e^{\delta t}f_1\|_2\bigr).
\end{equation}
Here $v_1,v_1'=O(1+t)$ control $\tau$, and replacing $y$ by the cutoff
$\bar\psi_0$ only adds a compactly supported $H^1$ term.
Taking the infimum over the representations proves the desired bound in the dual norm.

For every $\varphi\in C_c^\infty([0,\infty))$, the function $h$ satisfies
\[
\int_0^\infty
 \left[h'\varphi'+\left\{\frac{(n-2)^2}{4}
              -n(n+2)v_*^{p-1}\right\}h\varphi\right]\dd t
=\int_0^\infty(f_0\varphi-f_1\varphi')\dd t.
\]
This identity follows first for smooth representations and then by
approximation in the two weighted $L^2$ norms. It includes the endpoint
condition in the weak formulation. For a smooth source without an
endpoint distribution, integration by parts gives $h'(0)=0$.
Even reflection gives the radial weak equation with the required endpoint condition.

For uniqueness, a homogeneous solution with $h'(0)=0$ has the form
$h=c v_1$. If $h-\tau\bar\psi_0\in H^1_\delta$, then along a sequence
$t_j\to\infty$,
\[
c\,\mathcal W(v_1,y)
=\mathcal W(h-\tau y,y)(t_j)\to0.
\]
Here $y,y'$ are bounded and $\bar\psi_0=y$ on the tail. Since
$\mathcal W(v_1,y)\ne0$, we obtain $c=0$ and $h=0$.

\step{Step 2. Degree one.}
For the full-line estimates in Steps 2 and 3, choose a smooth even function \(\vartheta\) satisfying \(|\vartheta'|\le1\) and \(\vartheta(t)=|t|\) for \(|t|\ge2\). Define
\[
\|u\|_{H_s^1(\mathbb R)}
:=\|e^{s\vartheta(t)}u\|_{H^1(\mathbb R)},
\qquad s\in\{-\delta,\delta\}.
\]
These norms are equivalent to \(\|e^{s|t|}u\|_{H^1(\mathbb R)}\). We use the same weight with the product Sobolev norm on \(\mathbb R\times S^{n-1}\).

By the negative energy identity \eqref{eq:fixed-neck},
$|v_*'|<\frac{n-2}{2}v_*$. The positive solutions
\[
y_+(t)=e^t(\frac{n-2}{2}v_*+v_*'),\qquad
y_-(t)=e^{-t}(\frac{n-2}{2}v_*-v_*')
\]
of the equation for degree one satisfy
\begin{equation*}
\mathcal W(y_+,y_-)=-2nE_*\ne0.
\end{equation*}
Their periodic factors are bounded above and below by positive constants.
The Green kernel on the full line is
\[
G_1(t,s):=
\frac{y_+(\min\{t,s\})y_-(\max\{t,s\})}
     {-\mathcal W(y_+,y_-)}.
\]
It is continuous, its first derivatives have jumps on the diagonal, and
$G_1$, $\partial_tG_1$, and $\partial_sG_1$ are bounded off the diagonal
by $Ce^{-|t-s|}$. In the sense of distributions,
\[
\partial_t\partial_sG_1(t,s)=K_1(t,s)+\delta_{t=s},
\qquad |K_1(t,s)|\le Ce^{-|t-s|},
\]
where $K_1$ is the ordinary mixed derivative for $t\ne s$ and
$\delta_{t=s}$ is normalized by
$\int_\R\delta_{t=s}f(s)\dd s=f(t)$. The positive sign follows from the
jump $\partial_sG_1(s+,s)-\partial_sG_1(s-,s)=1$ across $t=s$.
For weak data $f=f_0+\partial_s f_1$, interpreted by duality as in
\eqref{eq:radial-weak-data} but on $\R$, define
\begin{align*}
h(t)&=\int_\R G_1(t,s)f_0(s)\dd s
       -\int_\R\partial_sG_1(t,s)f_1(s)\dd s,\\
h'(t)&=\int_\R\partial_tG_1(t,s)f_0(s)\dd s
       -\int_\R K_1(t,s)f_1(s)\dd s-f_1(t).
\end{align*}
Since $|\vartheta(t)-\vartheta(s)|\le|t-s|$ and $0<\delta<1$,
weighted convolution with $e^{-(1-\delta)|t-s|}$ proves
\[
\|h\|_{H^1_\delta(\R)}
\le C\bigl(\|e^{\delta\vartheta}f_0\|_2+
             \|e^{\delta\vartheta}f_1\|_2\bigr).
\]
Taking the infimum over the dual representations proves the estimate in
$H^{-1}_\delta$. Approximation establishes the weak equation. The two
homogeneous solutions grow at opposite ends, and no nonzero combination
belongs to $H^1_\delta$. This proves uniqueness. By reflection symmetry,
the inverse restricts to odd functions, as required by the Dirichlet condition for degree one.

\step{Step 3. Higher harmonics.}
Since $v_*(t)<2^{-(n-2)/2}$, we have
$n(n+2)\,v_*(t)^{p-1}\le\frac{n(n+2)}4$. For $\ell\ge2$,
\[
c_\ell:=\frac{(n-2)^2}{4}+\lambda_\ell-n(n+2)v_*^{p-1}
\ge \frac{n+2}{2}+(\lambda_\ell-2n).
\]

For weak data, conjugate by the weight and put
$Z(t)=e^{\delta\vartheta(t)}w_\ell(t)$.
The resulting bilinear form on $H^1(\R)$ is
\[
\mathcal B_\ell(Z,\varphi)
=\int_\R\bigl\{(Z'-\delta\vartheta'Z)
 (\varphi'+\delta\vartheta'\varphi)+c_\ell Z\varphi\bigr\}\dd t.
\]
On the diagonal this form satisfies
\[
\begin{aligned}
\mathcal B_\ell(Z,Z)
=\int_\R\bigl\{|Z'|^2+(c_\ell-\delta^2(\vartheta')^2)Z^2\bigr\}\dd t
\ge c
\int_\R\bigl\{|Z'|^2+(1+\lambda_\ell)Z^2\bigr\}\dd t,
\end{aligned}
\]
with $c=c(n,\delta)>0$, since $|\vartheta'|\le1$ and $\delta<1$.
Sum the forms over an orthonormal spherical harmonic basis of degrees
at least two. On this closed subspace, Parseval's identity and the
uniform bounds above imply
\[
\begin{split}
\sum_{\ell\ge2}\sum_m\mathcal B_\ell(Z_{\ell m},Z_{\ell m})
&\ge c\|Z\|_{H^1(\R\times S^{n-1})}^2,\\
|\langle e^{\delta\vartheta}f,Z\rangle|
&\le C\|f\|_{H^{-1}_\delta}\|Z\|_{H^1(\R\times S^{n-1})}.
\end{split}
\]
Here $Z_{\ell m}$ are the spherical harmonic coefficients of
$Z=e^{\delta\vartheta}w$, and $m$ indexes the basis in each eigenspace.
Multiplication of the source by the weight is understood by duality.
By the Lax--Milgram theorem,
\[
\|w\|_{H^1_\delta}\le C\|f\|_{H^{-1}_\delta}.
\]
The required parity is preserved because the weight and coefficients
are even. Combining this bound with \eqref{eq:radial-weak-estimate} and
the estimate for degree one, and undoing the normalization, proves the lemma.
\end{proof}

\begin{lemma}\label{lem:ground-state}
Suppose $\gamma$ satisfies the metric conditions stated after
\eqref{eq:conformal-covariance}. Let $V>0$ be smooth on $\overline{B_1}\setminus\{0\}$ and
satisfy \eqref{eq:background} and \eqref{eq:boundary-matching}.
Then $Q_\gamma$ is coercive on $H^1_{\mathrm{ev}}(B_1)$ with the
norm \eqref{eq:unweighted-even-space}.
\end{lemma}

\begin{proof}
The assumptions give $c\Theta\le V\le C\Theta$ and
$|D\log V|\le C/r$. For a smooth $\varphi$ vanishing near $0$ and
having even boundary trace, write $\varphi=V\zeta$ and integrate the
equation for $V$ against $V\zeta^2$. The boundary contribution is
\[
\int_{\partial B_1}\frac{(\partial_r+\frac{n-2}{2})V}{V}\varphi^2\dd S=0,
\]
since the numerator is antipodally odd and the other factors are even.
We obtain

\begin{equation}\label{eq:ground-state-identity}
Q_\gamma(\varphi,\varphi)
=\int_{B_1}V^2|D(\varphi/V)|_\gamma^2\dd x
+n(n-2)\int_{B_1}V^{p-1}\varphi^2\dd x.
\end{equation}
The second term controls $\int r^{-2}\varphi^2$, because
$V^{p-1}\ge c r^{-2}$. Also
\[
\int|D\varphi|^2
\le C\int V^2|D(\varphi/V)|_\gamma^2+C\int r^{-2}\varphi^2.
\]
Density proves coercivity in the norm
\eqref{eq:unweighted-even-space}.
\end{proof}

We next vary the background phase and estimate the nonlinear residual.
Fix a metric $\gamma$ satisfying the conditions stated after
\eqref{eq:conformal-covariance} and a positive solution $V$ satisfying
\eqref{eq:background} and \eqref{eq:boundary-matching}, with phase $T$.
Keep $\chi$ fixed and set
\[
A_\tau(x):=V(x)+\Phi_{T,\tau}(|x|).
\]
Near $0$, where $\chi=1$, this changes the asymptotic phase:
\[
A_\tau(x)=V_{T+\tau}(x)+w_0(x),\qquad
\frac{A_\tau(x)}{\Theta(x)}
=v_*(-\log|x|+T+\tau)+O(|x|^\delta).
\]
On $\{\chi=0\}$, we have $A_\tau=V$.
Throughout $X$, $\partial_\tau A_\tau=\psi_{T+\tau}$.
For sufficiently small $|\tau|$, $A_\tau/\Theta$ has fixed positive
upper and lower bounds. For $w\in H^1_\delta$ and
$\varphi\in H^1_{-\delta}$, define the residual weakly by
\[
\ip{\Fop_\gamma(A_\tau+w)}{\varphi}
=\ip{\Fop_\gamma(A_\tau)}{\varphi}+Q_\gamma(w,\varphi)
-n(n-2)\int_X\bigl((A_\tau+w)_+^p
                         -A_\tau^p\bigr)\varphi\dd x.
\]
The membership of $\Fop_\gamma(A_\tau+w)$  and $\Fop_\gamma(A_\tau)$ in \(H^{-1}_\delta\), and their differentiability with respect to \(\tau\), are established in Step 2 of the proof of Lemma~\ref{lem:nonlinear-map} below.
Formal differentiation gives the following candidate differential: 
for $(\dot\tau,\dot w)\in\cX_\delta$, 
\[
D_{(\tau,w)}[\Fop_\gamma(A_\tau+w)](\dot\tau,\dot w)
=\{-L_\gamma-n(n+2)(A_\tau+w)_+^{p-1}\}
  (\dot\tau\psi_{T+\tau}+\dot w).
\]
In the following, we verify that the phase contribution belongs to \(H^{-1}_\delta\), despite \(\psi_{T+\tau}\notin H^1_\delta\), and establish that this formula is the Fréchet differential.

Define the nonlinear remainder at $(0,0)$ by
\[
\cN(\tau,w):=\widehat\cL_{\gamma,V}(\tau,w)-\Fop_\gamma(A_\tau+w).
\]
Since $\Fop_\gamma(V)=0$, we have $\cN(0,0)=0$. 
For the metric estimate, fix an open set $D$ with compact closure in $X$
such that $\chi=0$ on a neighborhood of $\overline D$.

\begin{lemma}\label{lem:nonlinear-map}
After fixing \((\gamma,V)\), \(\delta\), \(\chi\), and \(D\), there exist \(\varepsilon_0,\eta_0,C>0\)
 such that
$(\tau,w)\mapsto\Fop_\gamma(A_\tau+w)$ is $C^1$ from
$(-\varepsilon_0,\varepsilon_0)\times H^1_\delta$ to $H^{-1}_\delta$.
Its derivative at $(0,0)$ is $\widehat\cL_{\gamma,V}$.
For $|\tau|<\varepsilon_0$ and $w\in H^1_\delta$,
\begin{equation}\label{eq:augmented-comparison}
\bigl\|D_{(\tau,w)}[\Fop_\gamma(A_\tau+w)]
-\widehat\cL_{\gamma,V}\bigr\|_{\cX_\delta\to H^{-1}_\delta}
\le C\bigl(|\tau|+\|w\|_{H^1_\delta}
+\|w\|_{H^1_\delta}^{p-1}\bigr).
\end{equation}
For $0<\varepsilon<\varepsilon_0$ and
$\|(\tau,w)\|_{\cX_\delta},
 \|(\widetilde\tau,\widetilde w)\|_{\cX_\delta}\le\varepsilon$,
\begin{equation}\label{eq:nonlinear-remainder-bound}
\|\cN(\tau,w)-\cN(\widetilde\tau,\widetilde w)\|_{H^{-1}_\delta}
\le C(\varepsilon+\varepsilon^{p-1})
       \|(\tau-\widetilde\tau,w-\widetilde w)\|_{\cX_\delta}.
\end{equation}
For every smooth metric $\gamma_1$ on $B_1$ satisfying
\[
dV_{\gamma_1}=dx,\qquad
\gamma_1=\gamma \quad\text{on }B_1\setminus D,\qquad
\|\gamma_1-\gamma\|_{C^2(\overline D)}\le\eta_0,
\]
we have
\begin{equation}\label{eq:compact-metric-variation}
\|\Fop_{\gamma_1}(V)\|_{H^{-1}_\delta}
+\|L_{\gamma_1}-L_\gamma\|_{H^1_\delta\to H^{-1}_\delta}
\le C\|\gamma_1-\gamma\|_{C^2(\overline D)}.
\end{equation}
The derivative of $(\tau,w)\mapsto\Fop_{\gamma_1}(A_\tau+w)$ differs
from $\widehat\cL_{\gamma,V}$ by at most the right side of
\eqref{eq:augmented-comparison} plus
$C\|\gamma_1-\gamma\|_{C^2(\overline D)}$ in the operator norm from $\cX_\delta$ to $H^{-1}_\delta$.
\end{lemma}

\begin{proof}
\step{Step 1. Weighted Sobolev and multiplication estimates.}
Apply the ordinary Sobolev inequality to $r^{-\delta}w$; its gradient
is controlled by \eqref{eq:weighted-spaces}. By duality,
\[
\|w\|_{L^{p+1}_\delta}\le C\|w\|_{H^1_\delta},
\qquad \|f\|_{H^{-1}_\delta}\le C\|f\|_{L^{2n/(n+2)}_\delta}.
\]

Placing reciprocal weights in the paired factors gives
\begin{equation}\label{eq:weighted-multiplication}
\|\beta w\|_{H^{-1}_\delta}\le C\|\beta\|_{L^{n/2}}\|w\|_{H^1_\delta}
\quad\text{and}\quad
\|\beta w\|_{H^{-1}_\delta}\le C\|r^2\beta\|_{L^\infty}\|w\|_{H^1_\delta}.
\end{equation}
The second estimate follows from the $r^{-2}|w|^2$ term in the norm.
We also use
$\|f\|_{H^{-1}_\delta}\le C\|r^{1-\delta}f\|_2$.

Since $0<p-1<1$, for $A>0$ we have
\begin{equation}\label{eq:scalar-Nemytskii}
|s_+^{p-1}-t_+^{p-1}|\le|s-t|^{p-1},
\qquad |(A+w)_+^{p-1}-A^{p-1}|\le C A^{p-2}|w|.
\end{equation}
The second inequality follows by separating $|w|\le A/2$ and its
complement. The first inequality and \eqref{eq:weighted-multiplication}
control the derivative in $w$ and give the remainder bound
\[
\|(A+w)_+^p-A^p-pA^{p-1}w\|_{H^{-1}_\delta}
\le C\|w\|_{H^1_\delta}^p.
\]
The weights satisfy
$r^{-\delta}|w|^p=r^{\delta(p-1)}|r^{-\delta}w|^p$ and $r\le1$.

\step{Step 2. The phase derivative.}
Near $0$, where \(\gamma=\gamma_0\), \(\chi=1\),
\[
-L_\gamma=-\Delta,\qquad
-\Delta V_T=n(n-2)V_T^p,\qquad
-\Delta V_{T+\tau}=n(n-2)V_{T+\tau}^p.
\]
Subtracting the equation for \(V_T\) from the assumed equation \eqref{eq:background} for \(V=V_T+w_0\) gives
\[
-\Delta w_0
=n(n-2)\bigl((V_T+w_0)^p-V_T^p\bigr)
\]
in that neighborhood. Since \(A_\tau=V_{T+\tau}+w_0\) there, we obtain
\[
\mathcal F_\gamma(A_\tau)
=n(n-2)\left[
(V_T+w_0)^p-V_T^p
+V_{T+\tau}^p-(V_{T+\tau}+w_0)^p
\right].
\]
The profiles \(V_T\), \(V_{T+\tau}\), \(V_T+w_0\) and \(A_\tau\) are uniformly comparable to \(\Theta\) in this neighborhood. The interval \(|\tau|<\varepsilon_0\) has already been chosen so that \(A_\tau\) is positive throughout \(X\). These  give
\[
|\mathcal F_\gamma(A_\tau)|\le Cr^{-2}|w_0|.
\]
Furthermore, in this neighborhood, 
\begin{equation}\label{eq:phase-column-cancellation}
\partial_\tau[\Fop_\gamma(A_\tau)]
=n(n+2)\{V_{T+\tau}^{p-1}-(V_{T+\tau}+w_0)^{p-1}\}
                           \partial_TV_{T+\tau}.
\end{equation}
Differentiating once more gives
\[
\begin{aligned}
\partial_\tau^2\mathcal F_\gamma(A_\tau)
=n(n+2)\Bigl\{
&(p-1)\bigl[V_{T+\tau}^{p-2}
 -(V_{T+\tau}+w_0)^{p-2}\bigr]
 (\partial_TV_{T+\tau})^2\\
&+\bigl[V_{T+\tau}^{p-1}
 -(V_{T+\tau}+w_0)^{p-1}\bigr]
 \partial_T^2V_{T+\tau}
\Bigr\}.
\end{aligned}
\]
The mean value theorem gives, for \(q=p-1\) and \(q=p-2\),
\[
\left|V_{T+\tau}^{q}-(V_{T+\tau}+w_0)^q\right|
\le C\Theta^{q-1}|w_0|.
\]
These, together with \(
|\partial_TV_{T+\tau}|+|\partial_T^2V_{T+\tau}|
\le C\Theta
\), then yield
\[
|\partial_\tau\mathcal F_\gamma(A_\tau)|
+
|\partial_\tau^2\mathcal F_\gamma(A_\tau)|
\le Cr^{-2}|w_0|
\quad\text{near }0.
\]
Away from this neighborhood, \(\mathcal F_\gamma(A_\tau)\) is smooth and depends smoothly on \(\tau\). Since \(A_\tau\) satisfies the outer matching conditions, its weak residual has no outer boundary contribution.  Combining the bounds on this compact region
with the estimate near $0$, and using the weighted source estimate
from Step~1, shows that
$\tau\mapsto\Fop_\gamma(A_\tau)$ is $C^2$ with values in
$H^{-1}_\delta$, with uniformly bounded second derivative
for $|\tau|<\varepsilon_0$. Since
\[
\left.\partial_s[\Fop_\gamma(A_s)]\right|_{s=0}
=\widehat\cL_{\gamma,V}(1,0),
\]
integration of the second derivative gives
\[
\|\partial_\tau [\Fop_\gamma(A_\tau)]
-\widehat\cL_{\gamma,V}(1,0)\|_{H^{-1}_\delta}
\le C|\tau|.
\]

For fixed $w\in H^1_\delta$, the mean value formula and
\eqref{eq:scalar-Nemytskii} bound the phase difference quotients
of $(A_\tau+w)_+^p-A_\tau^p$ by $Cr^{-2}|w|$.
Since $r^{-1-\delta}w\in L^2(X)$, dominated convergence after
multiplication by $r^{1-\delta}$ justifies differentiation in
$H^{-1}_\delta$ and gives
\[
\partial_\tau[\Fop_\gamma(A_\tau+w)]
-\partial_\tau[\Fop_\gamma(A_\tau)]
=
-n(n+2)\bigl\{(A_\tau+w)_+^{p-1}-A_\tau^{p-1}\bigr\}
\psi_{T+\tau}.
\]
By \eqref{eq:scalar-Nemytskii}, the absolute value of the
right-hand side is at most $Cr^{-2}|w|$. Hence
\[
\bigl\|
\partial_\tau[\Fop_\gamma(A_\tau+w)]
-\partial_\tau[\Fop_\gamma(A_\tau)]
\bigr\|_{H^{-1}_\delta}
\le C\|w\|_{H^1_\delta}.
\]

We next prove that the phase derivative
\[
(\tau,w)\longmapsto
\partial_\tau[\Fop_\gamma(A_\tau+w)]
\]
is continuous from
$(-\varepsilon_0,\varepsilon_0)\times H^1_\delta$
to $H^{-1}_\delta$.
Let $(\tau_j,w_j)\to(\tau,w)$ in $\cX_\delta$.
From every subsequence, extract a further subsequence for which $r^{-1-\delta}w_j$ converges
pointwise and is bounded in absolute value by one $L^2$ function.
Apply dominated convergence to the displayed product
$-n(n+2)\{(A_\tau+w)_+^{p-1}-A_\tau^{p-1}\}\psi_{T+\tau}$,
evaluated at $(\tau_j,w_j)$ and multiplied by $r^{1-\delta}$.
This gives convergence in $L^2$. Together with
smooth dependence of $\partial_\tau[\Fop_\gamma(A_\tau)]$ on $\tau$,
this gives
\[
\left.\partial_\tau[\Fop_\gamma(A_\tau+w_j)]\right|_{\tau=\tau_j}
\to\partial_\tau[\Fop_\gamma(A_\tau+w)]
\quad\text{in }H^{-1}_\delta
\]
for the original sequence.

For \(|\tau|<\varepsilon_0\) and \(w\in H^1_\delta\), consider the partial derivative with respect to \(w\). Split the coefficient difference into
$(A_\tau+w)_+^{p-1}-A_\tau^{p-1}$ and $A_\tau^{p-1}-V^{p-1}$.
The scalar Hölder estimate and the Sobolev embedding give
\[
\bigl\|(A_\tau+w)_+^{p-1}-A_\tau^{p-1}
\bigr\|_{L^{n/2}(X)}
\le C\|w\|_{H^1_\delta}^{p-1}.
\]
The uniform comparability of $A_\tau$ and $V$ to $\Theta$ throughout \(X\), together
with $|A_\tau-V|\le C|\tau|\Theta$, gives
\[
\bigl\|r^2(A_\tau^{p-1}-V^{p-1})\bigr\|_{L^\infty(X)}
\le C|\tau|.
\]
Applying the two estimates in
\eqref{eq:weighted-multiplication} to these respective
coefficient differences gives, for every
$\dot w\in H^1_\delta$,
\[
\left\|
D_w[\Fop_\gamma(A_\tau+w)][\dot w]
-\widehat\cL_{\gamma,V}(0,\dot w)
\right\|_{H^{-1}_\delta}
\le
C\left(|\tau|+\|w\|_{H^1_\delta}^{p-1}\right)
\|\dot w\|_{H^1_\delta}.
\]
To compare these derivatives at two points, first change \(w\) and then \(\tau\). The scalar Hölder estimate controls the first coefficient difference in \(L^{n/2}\) by \(C\|w-\widetilde w\|_{H^1_\delta}^{p-1}\). Since
\[
|A_\tau-A_{\widetilde\tau}|
\le C|\tau-\widetilde\tau|\Theta,
\]
the scalar Hölder estimate therefore controls the second coefficient difference after multiplication by \(r^2\) in \(L^\infty\) by \(C|\tau-\widetilde\tau|^{p-1}\). Applying \eqref{eq:weighted-multiplication} gives
\[
\left\|D_w[\Fop_\gamma(A_\tau+w)]
-D_w[\Fop_\gamma(A_{\widetilde\tau}+\widetilde w)]\right\|
 _{H^1_\delta\to H^{-1}_\delta}
\le
C\left(
\|w-\widetilde w\|_{H^1_\delta}^{p-1}
+|\tau-\widetilde\tau|^{p-1}
\right).
\]
Both partial derivatives are jointly continuous, so the map is \(C^1\). Combining the preceding two estimates for the partial derivative with respect to \(\tau\) with the estimate for the partial derivative with respect to \(w\) proves \eqref{eq:augmented-comparison} for every $w\in H^1_\delta$ and
$|\tau|<\varepsilon_0$.  Since
\[
D\cN(\tau,w)=\widehat\cL_{\gamma,V}
-D_{(\tau,w)}[\Fop_\gamma(A_\tau+w)],
\]
integrating along the segment between $(\tau,w)$ and
$(\widetilde\tau,\widetilde w)$ proves
\eqref{eq:nonlinear-remainder-bound}. In particular, $D\cN(0,0)=0$.

\step{Step 3. Compactly supported metric changes.}
Since \(\overline D\) is compact and disjoint from \(0\), the weighted \(H^1_\delta\) and \(H^1_{-\delta}\) norms restricted to \(D\) are equivalent to the ordinary \(H^1(D)\) norm, with constants depending on \(D\) and \(\delta\).
Integration by parts gives
\[
\|(L_{\gamma_1}-L_\gamma)w\|_{H^{-1}_\delta}
\le C\|\gamma_1-\gamma\|_{C^2(\overline D)}\|w\|_{H^1_\delta}.
\]

Also, $\Fop_{\gamma_1}(V)=-(L_{\gamma_1}-L_\gamma)V$ is supported
in $\overline D$. The smooth bounds for $V$ on $\overline D$ imply
$\|\Fop_{\gamma_1}(V)\|_{H^{-1}_\delta}
\le C\|\gamma_1-\gamma\|_{C^2(\overline D)}$, proving
\eqref{eq:compact-metric-variation}.
Since $\chi=0$ on a neighborhood of $\overline D$,
$(L_{\gamma_1}-L_\gamma)\psi_{T+\tau}=0$.
Thus the difference between the derivatives of
$(\tau,w)\mapsto\Fop_{\gamma_1}(A_\tau+w)$ and
$(\tau,w)\mapsto\Fop_\gamma(A_\tau+w)$ is
$(\dot\tau,\dot w)\mapsto-(L_{\gamma_1}-L_\gamma)\dot w$.
The outer boundary form agrees for both metrics, which are Euclidean
on the fixed boundary collar.
\end{proof}

We now define a finite relative energy and compute its first variation.
The outer Robin term enforces \eqref{eq:boundary-matching}.
For small $\varepsilon>0$, define
\[
\begin{split}
E_{\gamma,\varepsilon}(Y):={}&\frac12\int_{B_1\setminus B_\varepsilon}
 (|DY|_\gamma^2+c(n)\Scal_\gamma Y^2)\dd x
-\frac{(n-2)^2}{2}\int_{B_1\setminus B_\varepsilon}(Y_+)^{p+1}\dd x\\
&+\frac{n-2}{4}\int_{\partial B_1}Y^2\dd S.
\end{split}
\]

For the background in \eqref{eq:background}, put
\begin{equation}\label{eq:relative-energy}
\begin{split}
\cJ_{\gamma,V}(\tau,w):=\lim_{\varepsilon\downarrow0}\bigg\{
 &E_{\gamma,\varepsilon}(V+\Phi_{T,\tau}+w)-E_{\gamma,\varepsilon}(V)\\
 &-|S^{n-1}|\int_0^\tau
 \bigl[(v_*')^2+\frac{n-2}{2}v_*v_*'\bigr](-\log\varepsilon+T+\xi)\dd\xi
 \bigg\}.
\end{split}
\end{equation}

Integrating the radial cross term gives
$\frac{n-2}{4}|S^{n-1}|v_*^2(-\log\varepsilon+T)$
at the inner boundary. Its phase derivative accounts for the
$\frac{n-2}{2}v_*v_*'$ term in \eqref{eq:relative-energy}.

\begin{lemma}\label{lem:relative-energy}
Let
\[
\mathcal U_\delta
:=(-\varepsilon_0,\varepsilon_0)\times H^1_\delta
\subset\mathcal X_\delta.
\]
The limit in \eqref{eq:relative-energy} exists for every
$(\tau,w)\in\mathcal U_\delta$ and defines a finite real-valued
functional
\[
\cJ_{\gamma,V}\in C^2(\mathcal U_\delta;\mathbb R).
\]
For every $(\dot\tau,\dot w)\in\mathcal X_\delta$,
\begin{equation}\label{eq:relative-first-variation}
D\cJ_{\gamma,V}(\tau,w)[\dot\tau,\dot w]
=
\dot\tau\,\ip{\Fop_\gamma(A_\tau+w)}{\psi_{T+\tau}}
+
\ip{\Fop_\gamma(A_\tau+w)}{\dot w}.
\end{equation}
Here both brackets denote the duality between
$H^{-1}_\delta=(H^1_{-\delta})^*$ and $H^1_{-\delta}$,
using the continuous inclusion
$H^1_\delta\hookrightarrow H^1_{-\delta}$ for $\dot w$.
\end{lemma}

\begin{proof}
\step{Step 1. The pure profile and the remainder.}
For $V_T=r^{-(n-2)/2}v_*(-\log r+T)$, polar integration with
$t=-\log r$ gives the radial energy density
\[
\frac12(v_*'+\frac{n-2}{2}v_*)^2-\frac{(n-2)^2}{2}v_*^{p+1}
=(v_*')^2+E_*+\frac{n-2}{2}v_*v_*'.
\]

On an interval ending at $t=-\log\varepsilon$, the change in
the integral of this energy density under a phase shift equals
the phase integral subtracted in \eqref{eq:relative-energy},
minus the same phase integral evaluated at the fixed lower
endpoint in $t$. The constant $E_*$ cancels.

On the complement of a fixed neighborhood of $0$, the energy terms,
including those involving derivatives of $\chi$, are finite and
continuous in $(\tau,w)$ by the Sobolev embedding.
Near $0$, every term left after subtracting the pure Fowler contribution
contains $w$ or $w_0$. Weighted Cauchy--Schwarz gives
\[
\int_{B_1}\bigl(r^{-n/2}|Dw|+r^{-(n+2)/2}|w|\bigr)\dd x
\le C\|w\|_{H^1_\delta},
\]
because the complementary radial integral is
$\int_0^1r^{-1+2\delta}\dd r<\infty$.
The quadratic and nonlinear remainder terms are bounded by
\[
C\int_{B_1}\bigl(|Dw|^2+r^{-2}w^2+|w|^{p+1}\bigr)\dd x
\le C\bigl(\|w\|_{H^1_\delta}^2+
                   \|w\|_{H^1_\delta}^{p+1}\bigr).
\]

The outer boundary term is continuous by the trace inequality.
Apply these bounds to $w_0$ after multiplying it by a smooth cutoff
supported near $0$; the complementary terms lie in a fixed compact
subset of $X$. The resulting estimates prove existence of the limit.
Applying the inequalities to differences proves continuity in $(\tau,w)$.

\step{Step 2. First variation for smooth corrections.}
First take $w,\dot w$ smooth, vanishing near $0$, and satisfying
both matching conditions. Put $Y=V+\Phi_{T,\tau}+w$ and
$\dot Y=\dot\tau\psi_{T+\tau}+\dot w$.
At $r=1$ the boundary term from integration
by parts is
$\int_{\partial B_1}(\partial_r+\frac{n-2}{2})Y\,\dot Y\dd S=0$
by antipodal parity. At the inner sphere the outward normal of the
annulus is $-\partial_r$. Since
\[
-\partial_r V_{T+\tau}
=r^{-n/2}(\frac{n-2}{2}v_*+v_*')(-\log r+T+\tau),
\]
its pairing with $\dot\tau\,\partial_TV_{T+\tau}$ is
\[
|S^{n-1}|\bigl[(v_*')^2+\frac{n-2}{2}v_*v_*'\bigr]
        (-\log\varepsilon+T+\tau)\dot\tau
+O(\varepsilon^\delta)|\dot\tau|.
\]

The remainder estimate follows from
$w_0\in\cC^{1,\alpha}_{\Theta;\delta}$.
Differentiating the subtracted term in \eqref{eq:relative-energy}
cancels this boundary contribution. Letting $\varepsilon\downarrow0$
proves \eqref{eq:relative-first-variation} for these variations.

\step{Step 3. Weak variations.}
Choose smooth matching functions $w_j,\dot w_j$ converging to
$w,\dot w$ in $H^1_\delta$, and put
$Y_j(s)=V+\Phi_{T,\tau+s\dot\tau}+w_j+s\dot w_j$.
Let $Y(s)$ denote the same expression with $w,\dot w$.
For $|h|$ small, the estimates in Lemma~\ref{lem:nonlinear-map} give uniform convergence
of the residuals in $H^{-1}_\delta$ and of the variations in
$H^1_{-\delta}$ along the segment from $0$ to $h$. Thus
\[
\begin{aligned}
\cJ_{\gamma,V}(\tau+h\dot\tau,w+h\dot w)-\cJ_{\gamma,V}(\tau,w)
=\lim_{j\to\infty}\int_0^h
\ip{\Fop_\gamma(Y_j(s))}{Y_j'(s)}\dd s
=\int_0^h\ip{\Fop_\gamma(Y(s))}{Y'(s)}\dd s.
\end{aligned}
\]
Dividing by $h$ and letting $h\to0$ proves
\eqref{eq:relative-first-variation}. The residual is $C^1$, and
$\psi_{T+\tau}$ depends smoothly on $\tau$ in $H^1_{-\delta}$.
Hence $\cJ_{\gamma,V}$ is $C^2$.
\end{proof}

The relative energy is also consistent with a change of cutoff.
Displaying the cutoff temporarily as a superscript and using
$\widetilde w$ from \eqref{eq:cutoff-nonlinear-change}, we have
\[
\cJ^{\widetilde\chi}_{\gamma,V}(\tau,\widetilde w)
=\cJ^\chi_{\gamma,V}(\tau,w).
\]
Indeed the represented function, reference background, and asymptotic
phase are identical, so both the truncated energies and the subtracted
term in \eqref{eq:relative-energy} agree.

\section{Metric perturbation and the projected equation}\label{sec:metric-insertion}

In this section, we continue the background under a local metric
perturbation and solve the projected equation after adding a bubble.
Lemmas~\ref{lem:background-continuation} and \ref{lem:projected-inverse}
provide the required inverses, and Proposition~\ref{prop:projected-solution}
constructs the correction and its parameter derivatives.
Fix an admissible background $(\gamma,V)$ with phase $T$ in
\eqref{eq:background}, and a ball $B_{4\rho}(a)$ as in
Proposition~\ref{prop:one-bubble}.

We define the tensor used in the localized metric perturbation and
the Euclidean bubble to be inserted in a conformally flat ball. We use the tensor
identities and metric variations of Brendle--Marques
\cite{BrendleMarques2009}.

Choose a nonzero algebraic Weyl tensor $W$ on $\R^n$ and put
\begin{equation*}
H_{ik}(x):=\sum_{r,s=1}^n W_{irks}x_rx_s,\qquad
f(s):=80-50s+7s^2,\qquad
\cH_{ik}(x):=f(|x|^2)H_{ik}(x).
\end{equation*}
The Weyl tensor at a point of a nonconformally flat
product of two space forms is one such choice of $W$. The Weyl symmetries imply
\begin{equation*}
\sum_{i=1}^nH_{ii}=0,\qquad
\sum_{i=1}^nx_iH_{ik}=0,\qquad
\sum_{i=1}^n\partial_iH_{ik}=0,
\quad1\le k\le n.
\end{equation*}
If $\psi=\psi(|x|^2)$ is any smooth radial scalar, then
\begin{equation}\label{eq:radial-gauge}
\begin{aligned}
\operatorname{tr}(\psi H)&=0,&
\sum_{i=1}^n x_i(\psi H)_{ik}&=0,\quad 1\le k\le n,\\
\sum_{i=1}^n\partial_i(\psi H)_{ik}&=0,\quad 1\le k\le n,\qquad &
\sum_{i,k=1}^n\partial_i\partial_k(\psi H)_{ik}&=0.
\end{aligned}
\end{equation}
The same formulas hold after multiplication by a radial cutoff.

For $(\lambda,b)\in(0,\infty)\times\R^n$, define
\[
U_{\lambda,b}(y)
:=\Bigl(\frac{\lambda}{\lambda^2+|y-b|^2}\Bigr)^{(n-2)/2}.
\]
Then $-\Delta U_{\lambda,b}=n(n-2)U_{\lambda,b}^p$.
The Jacobi fields and their span are
\[
(J_{\lambda,b})_{0}:=\lambda\partial_\lambda U_{\lambda,b},\qquad
(J_{\lambda,b})_{i}:=\partial_{b_i}U_{\lambda,b}\quad 1\le i\le n,
\]
and $\cK_{\lambda,b}:=\operatorname{span}\{(J_{\lambda,b})_{0},\ldots,(J_{\lambda,b})_{n}\}$.
By the Euclidean kernel theorem of Bianchi--Egnell
\cite[Lemma~A1]{BianchiEgnell1991}, $\cK_{\lambda,b}$ is the kernel of
$-\Delta-n(n+2)U_{\lambda,b}^{p-1}$ in
$\dot H^1(\R^n)$, the closure of $C_c^\infty(\R^n)$ under the
Dirichlet norm $\|D\,\cdot\|_{L^2}$.
Orthogonality in $\dot H^1$ always refers to the Dirichlet scalar product
$(\phi,\psi)_{\dot H^1}=\int_{\R^n}D\phi\cdot D\psi\dd y$.
The Sobolev inequality and its dual give
\[
\dot H^1(\R^n)\hookrightarrow L^{p+1}(\R^n),
\qquad
L^{2n/(n+2)}(\R^n)\hookrightarrow\dot H^{-1}(\R^n).
\]

Fix a smooth radial cutoff $\eta:[0,\infty)\to[0,1]$ with
$\eta=1$ on $[0,1/2]$ and $\eta=0$ on $[3/4,\infty)$.
To describe the part of the bubble residual that is linear
in the metric perturbation, set
\[
S_{\lambda,b}:=\sum_{i,k=1}^n\cH_{ik}\partial_{ik}U_{\lambda,b}.
\]
The following identities control the energy of the explicit bubble
and the size of its projected correction.

\begin{lemma}\label{lem:linear-cancellation}
For every $\psi\in C_c^\infty([0,\infty))$ and every
$(\lambda,b)\in(0,\infty)\times\mathbb R^n$,
\begin{equation}\label{eq:linear-energy-zero}
\sum_{i,k=1}^n\int_{\R^n}\psi(|y|^2)H_{ik}
\partial_iU_{\lambda,b}\partial_kU_{\lambda,b}\dd y=0.
\end{equation}
The first variation of scalar curvature in $\psi H$ is zero. Moreover,
\begin{equation}\label{eq:center-source}
S_{\lambda,b}(y)=\frac{n(n-2)\lambda^{(n-2)/2}}{(\lambda^2+|y-b|^2)^{(n+2)/2}}{\sum_{i,k=1}^n\cH_{ik}(y)b_ib_k}.
\end{equation}
On each compact subset of $(0,\infty)\times\R^n$, uniformly in the parameters $(\lambda,b)$,
\begin{equation}\label{eq:center-source-bounds}
\begin{aligned}
|S_{\lambda,b}(y)|&\le C|b|^2(1+|y|)^{4-n},\\
|D_{\lambda,b}S_{\lambda,b}(y)|&\le C|b|(1+|y|)^{4-n},\\
|D_{\lambda,b}^2S_{\lambda,b}(y)|&\le C(1+|y|)^{4-n}.
\end{aligned}
\end{equation}
The same estimates hold after multiplication by $\eta(|y|/R)$,
uniformly for $R\ge1$.
\end{lemma}

\begin{proof}

\step{Step 1. The linear energy term.}
Differentiating the bubble directly, we obtain, for $1\le i,k\le n$,
\[
\partial_iU_{\lambda,b}\partial_kU_{\lambda,b}
-c(n)\partial_{ik}(U_{\lambda,b}^2)
=\frac{\delta_{ik}}n
\bigl(|DU_{\lambda,b}|^2-c(n)\Delta(U_{\lambda,b}^2)\bigr).
\]
Contracting with $\psi H$, using its zero trace, and integrating by parts
twice, we obtain
\[
\sum_{i,k=1}^n\int_{\R^n}(\psi H)_{ik}
\partial_{y_i}U_{\lambda,b}\partial_{y_k}U_{\lambda,b}\dd y
=c(n)\sum_{i,k=1}^n\int_{\R^n}
\partial_{y_i}\partial_{y_k}(\psi H)_{ik}\,U_{\lambda,b}^2\dd y=0.
\]
The last equality follows from \eqref{eq:radial-gauge}. The trace and divergence identities in \eqref{eq:radial-gauge}
also imply that the first variation of scalar curvature
in the direction $\psi H$ vanishes.

\step{Step 2. The residual and parameter derivatives.}
The Hessian of the bubble is
\[
\partial_{ik}U_{\lambda,b}
=-(n-2)\lambda^{(n-2)/2}
\frac{\delta_{ik}}{(\lambda^2+|y-b|^2)^{n/2}}+n(n-2)\lambda^{(n-2)/2}
\frac{(y_i-b_i)(y_k-b_k)}
{(\lambda^2+|y-b|^2)^{(n+2)/2}}.
\]
The trace and radial identities in \eqref{eq:radial-gauge} remove the
first term and every factor $y_i$ or $y_k$ in the second term.
We thus obtain \eqref{eq:center-source}. Differentiating this formula
at most twice and using $|\cH(y)|\le C(1+|y|)^6$, we deduce
\eqref{eq:center-source-bounds}. The cutoff is independent of the parameters $(\lambda,b)$.
\end{proof}

We now perturb the metric in the insertion ball and continue the background solution.
The fixed background satisfies $\gamma=\gamma_0$ on $B_{4\rho}(a)$.
The coordinates $x$ remain centered at the puncture; the bubble
rescaling is $x=a+\mu y$. Choose the phase cutoff
to vanish on a neighborhood of $\overline{B_{4\rho}(a)}$.
For example, take $\chi(r)=\chi_*(r/\varepsilon)$ with
$2\varepsilon<|a|-4\rho$, using the fixed smooth function in
\eqref{eq:stage-cutoff}. Equation~\eqref{eq:cutoff-inverse-bound}
preserves the isomorphism supplied by admissibility. Fix this cutoff and
its resulting inverse bound before varying $\mu$; all estimates in
this insertion use this same cutoff, independent of $(\mu,\lambda,b)$.
With the tensor profile just defined, set
\begin{equation}\label{eq:metric-profile}
\begin{aligned}
(h_\mu)_{ik}(x)&=\mu^{19/4}\eta(|x-a|/\rho)
                    f(\mu^{-2}|x-a|^2)H_{ik}(x-a),\\
h_\mu(a+\mu y)&=\mu^{27/4}\eta(|y|/R)\cH(y),
\qquad R=\rho/\mu.
\end{aligned}
\end{equation}
Define $\gamma_\mu=e^{h_\mu}$ in $B_\rho(a)$ and
$\gamma_\mu=\gamma$ elsewhere. The zero trace condition gives
$\dd V_{\gamma_\mu}=\dd x$. The metric remains Euclidean near $0$
and on the unchanged boundary collar. For every fixed $k\ge0$,
\begin{equation}\label{eq:metric-small}
\|\gamma_\mu-\gamma\|_{C^k(\overline{B_1})}
\le C_{k,a,\rho}\mu^{3/4}\to0.
\end{equation}
In the fixed coordinates $x$, the terms of degrees two,
four, and six in $h_\mu$ have coefficients proportional to
$\mu^{19/4}$, $\mu^{11/4}$, and $\mu^{3/4}$, respectively.
This gives \eqref{eq:metric-small}.
After the rescaling $x=a+\mu y$, the prefactor in
\eqref{eq:metric-profile} is $\mu^{27/4}$.
Lemma~\ref{lem:linear-cancellation} eliminates the linear
metric contribution, so the leading quadratic contribution
to the bubble energy has order $\mu^{27/2}$.
By Lemma~\ref{lem:pure-energy}, the interaction with the
background is $O(\mu^{(n-2)/2})=o(\mu^{27/2})$ for $n\ge30$.
Proposition~\ref{prop:local-well} supplies a nondegenerate
minimum of the leading coefficient in this dimension range.

\begin{lemma}\label{lem:background-continuation}
For the admissible background $(\gamma,V)$ fixed in this section
and the metric $\gamma_\mu$ constructed above,
there are, for all sufficiently small $\mu>0$, a positive solution
$V_\mu$ and a phase $T_\mu$ such that
\[
V_\mu=V+\Phi_{T,T_\mu-T}+w_\mu,
\qquad |T_\mu-T|+\|w_\mu\|_{H^1_\delta}\to 0
\quad\text{as }\mu\downarrow0.
\]
The pair $(\gamma_\mu,V_\mu)$ satisfies \eqref{eq:background}
with phase $T_\mu$, and $V_\mu$ satisfies
\eqref{eq:boundary-matching}. For every fixed integer $k\ge0$,
\begin{equation}\label{eq:background-smooth}
V_\mu\to V\quad\text{in }\cC^k_{\Theta;0}(X)
\quad\text{as }\mu\downarrow0.
\end{equation}
The linearization $\widehat\cL_{\gamma_\mu,V_\mu}$ is an isomorphism
whose inverse norm is bounded independently of small $\mu$.
\end{lemma}
\begin{proof}
\step{Step 1. Continuation and regularity.}
The residual $\Fop_{\gamma_\mu}(V)$ is supported in $B_\rho(a)$, and
\[
\|\Fop_{\gamma_\mu}(V)\|_{H^{-1}_\delta}
\le C\|\gamma_\mu-\gamma\|_{C^2(B_\rho(a))}\to0
\quad\text{as }\mu\downarrow0.
\]

Apply \eqref{eq:compact-metric-variation} with $D=B_\rho(a)$.
By Lemma~\ref{lem:nonlinear-map} and the implicit function theorem,
we obtain $T_\mu$ and $w_\mu$ with the stated convergence.
By Lemma~\ref{lem:ground-state} and the small metric perturbation,
$-L_{\gamma_\mu}$ remains coercive. Also,
\[
\inf_X\frac{V+\Phi_{T,T_\mu-T}}{\Theta}
\ge\inf_X\frac{V}{\Theta}-C|T_\mu-T|
\ge\tfrac12\inf_X\frac{V}{\Theta}>0
\]
for small $\mu$. Lemma~\ref{lem:regularity-end} therefore proves
positivity, smoothness, and a decaying remainder with phase
$T_\mu$. Its convergence assertion, applied to the small
correction and phase shift above, proves \eqref{eq:background-smooth}
for every fixed $k$.

\step{Step 2. Preservation of invertibility.}
Keep the spaces and the cutoff fixed as $\mu\downarrow0$.
The derivative of the
nonlinear map in Lemma~\ref{lem:nonlinear-map} at the continued solution is
\[
(\dot\tau,\dot w)\longmapsto
\cL_{\gamma_\mu,V_\mu}(\dot\tau\psi_{T_\mu}+\dot w).
\]
By Lemma~\ref{lem:nonlinear-map},
\[
\|\widehat\cL_{\gamma_\mu,V_\mu}
       -\widehat\cL_{\gamma,V}\|_{\cX_\delta\to H^{-1}_\delta}
\le C\bigl(\|\gamma_\mu-\gamma\|_{C^2(B_\rho(a))}
  +|T_\mu-T|+\|w_\mu\|_{H^1_\delta}
 +\|w_\mu\|_{H^1_\delta}^{p-1}\bigr)=o(1).
\]

The derivatives with respect to the phase are compared by
\eqref{eq:phase-column-cancellation}.
A Neumann series proves the asserted isomorphism and uniform inverse bound.
\end{proof}

We now estimate the residual and its first two bubble parameter
derivatives after adding a cutoff bubble to the continued background.
Let $s_n$ be the number supplied by
Proposition~\ref{prop:local-well}. Fix a compact neighborhood $K$
of $(\sqrt{s_n},0)$ in $(0,\infty)\times\mathbb R^n$.
We take parameter derivatives on an open neighborhood of $K$
whose closure is compact in this parameter space.
The notation $D_{\lambda,b}$ denotes ordinary derivatives with respect to
$(\lambda,b)$, with the spatial variable, insertion scale, background,
and cutoffs held fixed. We also write $\partial_{(\lambda,b)}^\beta$
for parameter derivatives of multiindex $\beta$.
For a $C^2$ function on this neighborhood, its $C^2(K)$ norm
includes the suprema on $K$ of the function and its derivatives
through order two.
The insertion point, radius, and background remain fixed.
Choose a smooth cutoff $0\le\zeta\le1$ equal to one on $B_\rho(a)$
and zero outside $B_{2\rho}(a)$. Define the bubble on $X$ by
\begin{equation}\label{eq:cutoff-bubble}
\widehat U_{\mu,\lambda,b}(x)
:=\zeta(x)
\left(\frac{\mu\lambda}{(\mu\lambda)^2+|x-a-\mu b|^2}\right)^{(n-2)/2},
\end{equation}
extended by zero. It has Euclidean scale $\mu\lambda$ and center
$a+\mu b$; $\lambda,b$ are the rescaled parameters. Put
\[
\begin{aligned}
\cR_{\mu,\lambda,b}&:=\Fop_{\gamma_\mu}
                    (V_\mu+\widehat U_{\mu,\lambda,b}),\\
\cL_{\mu,\lambda,b}&:=-L_{\gamma_\mu}
        -n(n+2)(V_\mu+\widehat U_{\mu,\lambda,b})^{p-1}.
\end{aligned}
\]

The residual is supported in $\overline{B_{2\rho}(a)}$. Write
\begin{equation}\label{eq:error-scale}
e_\mu:=\mu^{27/2}+\mu^{(n+2)/4}.
\end{equation}
\begin{lemma}\label{lem:residual}
There are $C,\mu_0>0$ such that the following estimates hold
for $0<\mu<\mu_0$ and $(\lambda,b)\in K$:
\begin{align}\label{eq:residual-bounds}
\|\cR_{\mu,\lambda,b}\|_{H^{-1}_\delta}
 &\le C\bigl(\mu^{27/4}|b|^2+e_\mu\bigr),\notag\\
\|D_{\lambda,b}\cR_{\mu,\lambda,b}\|_{H^{-1}_\delta}
 &\le C\bigl(\mu^{27/4}|b|+e_\mu\bigr),\\
\|D^2_{\lambda,b}\cR_{\mu,\lambda,b}\|_{H^{-1}_\delta}
 &\le C\bigl(\mu^{27/4}+e_\mu\bigr).\notag
\end{align}
In particular, $e_\mu=o(\mu^{27/4})$ for $n\ge30$.
\end{lemma}
\begin{proof}
\step{Step 1. Metric and cutoff errors.}
The residual and its bubble parameter derivatives are supported
in $\overline{B_{2\rho}(a)}$. On this fixed set, the weights and their
reciprocals have bounds independent of $\mu$. The weighted dual norm
is therefore controlled by the local $L^{2n/(n+2)}$ norm, with constants
depending on the ball and background. This exponent gives the
critical source scaling under $x=a+\mu y$.

For the linear metric source $\mu^{27/4}\eta(|y|/R)S_{\lambda,b}$,
use the factors $|b|^2,|b|,1$ in \eqref{eq:center-source-bounds}.
The membership $(1+|y|)^{4-n}\in L^{2n/(n+2)}$ for $n>10$ proves
the first term in each estimate. We rescale the metric contribution
$(-L_{\gamma_\mu}+\Delta)\widehat U_{\mu,\lambda,b}$
by setting $x=a+\mu y$ and multiplying by $\mu^{(n+2)/2}$.
Its linear part is $\mu^{27/4}\eta(|y|/R)S_{\lambda,b}$.
After subtracting this linear part, the remainder and its
first two parameter derivatives are bounded by
\[
C\mu^{27/2}(1+|y|)^{12-n}\mathbf 1_{|y|\le R}.
\]
This majorant has bounded $L^{2n/(n+2)}$ norm for $n>26$.
The bubble cutoff error is $O(\mu^{(n-2)/2})$ in the dual norm, also after two
parameter derivatives.

\step{Step 2. The nonlinear interaction.}
For the interaction, in the local Euclidean variables let
\[
A(x):=V_\mu(x),\qquad
B(x):=\widehat U_{\mu,\lambda,b}(x)
=\zeta(x)\left(\frac{\mu\lambda}
 { (\mu\lambda)^2+|x-a-\mu b|^2}\right)^{(n-2)/2}.
\]

The positive function $A$ is independent of $(\lambda,b)$ and has
uniform positive upper and lower bounds on $B_{2\rho}(a)$. The cutoff makes $B\ge0$.
Since $n\ge30$, the exponent $(n+2)/(n-2)$ lies in $(1,2)$. Thus
\[
0\le(A+B)^p-A^p-B^p
\le C\min\{AB^{p-1},A^{p-1}B\}.
\]

Differentiating the explicit bubble, for every multiindex $\beta$ with
$|\beta|\le2$ we have
$|\partial_{(\lambda,b)}^{\beta}B|\le C B$, uniformly on $K$.
Applying the chain rule, we obtain
\[
\left|\partial_{(\lambda,b)}^{\beta}
 \{(A+B)^p-A^p-B^p\}\right|
\le C\min\{AB^{p-1},A^{p-1}B\}.
\]

When $0<B\le A$, divide the mixed expression by
$A^p$ and use the variable $B/A\le1$; its first two differentiated
terms, multiplied by $B$ and $B^2$ from the chain rule, are bounded by
$CA^{p-1}B$. When $B\ge A$, divide by $B^p$ and use $A/B\le1$;
the same products are bounded by $CAB^{p-1}$.
At a zero of $\zeta$, $B$ and all its parameter derivatives vanish,
hence the inequality extends there.

\step{Step 3. Integration of the interaction estimates.}
For the radial integration about the bubble center, write $s=|x-a-\mu b|$. On $s\le\sqrt\mu$ use
$B^{p-1}\le C\mu^2(\mu+s)^{-4}$, and on
$s\ge\sqrt\mu$ use $B\le C\mu^{(n-2)/2}s^{2-n}$. Since $n>6$,
\[
\begin{aligned}
\int_0^{\sqrt\mu}
 [\mu^2(\mu+s)^{-4}]^{2n/(n+2)}s^{n-1}\dd s
 &\le C\mu^{n/2},\\
\int_{\sqrt\mu}^{3\rho}
 (\mu^{(n-2)/2}s^{2-n})^{2n/(n+2)}s^{n-1}\dd s
 &\le C\mu^{n/2}.
\end{aligned}
\]
For the first estimate, split at $s=\mu$ and use
$n-8n/(n+2)>0$; for the second, use $n-(n-2)2n/(n+2)<0$.
Taking the $(n+2)/(2n)$ power, we obtain the bound
$C\mu^{(n+2)/4}$ for the interaction and its first two parameter derivatives. The cutoff error is absorbed
because $\mu^{(n-2)/2}\le\mu^{(n+2)/4}$ for $0<\mu\le1$ and $n\ge30$.
Together with the metric errors, this proves \eqref{eq:residual-bounds}.
By \eqref{eq:error-scale}, \(\mu^{-27/4}e_\mu
=\mu^{27/4}+\mu^{(n-25)/4}\to0\) as \(\mu\to 0\).
\end{proof}

We next impose orthogonality to the bubble Jacobi fields and define
the associated projections.
Let
\[
(J_{\mu,\lambda,b})_{0}:=\lambda\partial_\lambda \widehat U_{\mu,\lambda,b},\qquad
(J_{\mu,\lambda,b})_{i}:=\partial_{b_i}\widehat U_{\mu,\lambda,b}
\qquad 1\le i\le n,
\]
and let $\cK_{\mu,\lambda,b}$ be their span.
We use $\partial_i$ for Euclidean spatial derivatives and
$\partial_\lambda$, $\partial_{b_i}$, and
$\partial_{(\lambda,b)}^\beta$ for parameter derivatives.
Gradients and Hessians in the bubble parameters use ordinary
derivatives in $(\lambda,b)$.
Use the flat Robin scalar product
\begin{equation}\label{eq:star-product}
(\phi,\psi)_*:=Q_{\gamma_0}(\phi,\psi)
=\int_{B_1}D\phi\cdot D\psi\dd x
 +\frac{n-2}{2}\int_{\partial B_1}\phi\psi\dd S.
\end{equation}

When one factor is a bubble Jacobi field, its boundary trace is zero,
and this is the usual Euclidean Dirichlet pairing. In particular the
associated source is $(-\Delta)J_\alpha$.
Pairing with a compactly supported Jacobi field extends continuously to
$H^1_\delta$. Define the Gram matrix by
\[
(G_{\mu,\lambda,b})_{\alpha\beta}
:=((J_{\mu,\lambda,b})_\alpha,(J_{\mu,\lambda,b})_\beta)_*,
\qquad 0\le\alpha,\beta\le n.
\]
Changing variables by $x=a+\mu y$, we obtain
\begin{equation}\label{eq:gram-limit}
(G_{\mu,\lambda,b})_{\alpha\beta}
\to
\int_{\R^n}D (J_{\lambda,b})_{\alpha}\cdot D (J_{\lambda,b})_{\beta}\dd y,
\end{equation}
uniformly on $K$ as $\mu\downarrow0$. This limiting Gram matrix is
positive definite.

Write
\[
\cK_{\mu,\lambda,b}^{\perp_*}
:=\{w\in H^1_\delta:(w,(J_{\mu,\lambda,b})_\alpha)_*=0,\
                         0\le\alpha\le n\}.
\]
Let $\Pi^\perp:H^1_\delta\to\cK_{\mu,\lambda,b}^{\perp_*}$
be the projection along $\cK_{\mu,\lambda,b}$.
Write $G=G_{\mu,\lambda,b}$ and
$J_\alpha=(J_{\mu,\lambda,b})_\alpha$ for $0\le\alpha\le n$.
The projection
\[
\Pi^{\perp,*}:H^{-1}_\delta\to
\{f\in H^{-1}_\delta:\ip{f}{J_\beta}=0,\ 0\le\beta\le n\}
\]
is given by
\begin{equation}\label{eq:dual-projection}
\Pi^{\perp,*}f
=f-\sum_{\alpha,\beta=0}^n
 (G^{-1})_{\alpha\beta}\ip{f}{J_\beta}(-\Delta)J_\alpha.
\end{equation}
Its kernel is $\operatorname{span}\{(-\Delta)J_0,\ldots,(-\Delta)J_n\}$.
The fields $J_\alpha$ and sources $(-\Delta)J_\alpha$ are supported
in $\overline{B_{2\rho}(a)}$. Since $\chi=0$ on a neighborhood of this
closed ball, both $\Phi_{T_\mu,\tau}$ and $\psi_{T_\mu+\tau}$ vanish
there. The locality of $\cL_{\mu,\lambda,b}$ gives
\[
(\psi_{T_\mu},J_\alpha)_*=0,\qquad
\langle\cL_{\mu,\lambda,b}\psi_{T_\mu},J_\alpha\rangle=0,
\qquad 0\le\alpha\le n.
\]
Thus $\cL_{\mu,\lambda,b}\psi_{T_\mu}$ belongs to
$\Pi^{\perp,*}H^{-1}_\delta$.

We next prove a uniform inverse estimate for corrections
$w\perp_*\cK_{\mu,\lambda,b}$.
The constants may depend on the fixed background, weight, insertion
ball, local tensor, cutoffs, and parameter set $K$.

\begin{lemma}\label{lem:projected-inverse}
There are $C,\mu_0>0$ such that, for $0<\mu<\mu_0$,
$(\lambda,b)\in K$, and $(\tau,w)\in\mathbb R\oplus\mathcal K_{\mu,\lambda,b}^{\perp_*}$,
\begin{equation}\label{eq:projected-inverse}
|\tau|+\|w\|_{H^1_\delta}
\le C\left\|
\Pi^{\perp,*}\cL_{\mu,\lambda,b}(\tau\psi_{T_\mu}+w)
\right\|_{H^{-1}_\delta}.
\end{equation}
The projected operator from
$\R\oplus\cK_{\mu,\lambda,b}^{\perp_*}$ to
$\Pi^{\perp,*}H^{-1}_\delta$ is an isomorphism.
\end{lemma}

\begin{proof}
\step{Step 1. The Fredholm index.}
By Lemma~\ref{lem:background-continuation},
$\widehat\cL_{\gamma_\mu,V_\mu}:\cX_\delta\to H^{-1}_\delta$
is an isomorphism. For each fixed $\mu$, the additional potential
\[
\mathcal V_\mu=n(n+2)\{(V_\mu+\widehat U_{\mu,\lambda,b})^{p-1}-V_\mu^{p-1}\}
\]
is smooth, bounded, and compactly supported. Multiplication by
$\mathcal V_\mu$ is compact from $H^1_\delta$ to $H^{-1}_\delta$.
Since $\mathcal V_\mu\psi_{T_\mu}=0$, we have
\[
\cL_{\mu,\lambda,b}(\tau\psi_{T_\mu}+w)
=\widehat\cL_{\gamma_\mu,V_\mu}(\tau,w)-\mathcal V_\mu w.
\]

Thus $(\tau,w)\mapsto\cL_{\mu,\lambda,b}(\tau\psi_{T_\mu}+w)$
is a compact perturbation of an isomorphism from $\cX_\delta$
to $H^{-1}_\delta$ and has Fredholm index zero.
The conditions $(w,J_\alpha)_*=0$, $0\le\alpha\le n$, define a
subspace of codimension $n+1$ in the domain. The projection
$\Pi^{\perp,*}$ onto its range has kernel
$\operatorname{span}\{(-\Delta)J_0,\ldots,(-\Delta)J_n\}$ of
dimension $n+1$. Restricting the domain by these orthogonality
conditions and projecting the image therefore gives Fredholm
index $0-(n+1)+(n+1)=0$.

\step{Step 2. Control of the Jacobi components.}
Suppose the estimate fails. Choose $\mu_j\downarrow0$,
$(\lambda_j,b_j)\to(\lambda_\infty,b_\infty)$ and
pairs $(\tau_j,w_j)$ with
$|\tau_j|+\|w_j\|_{H^1_\delta}=1$ and
$w_j\perp_*\cK_{\mu_j,\lambda_j,b_j}$, whose projected images tend to zero. Write
$\cL_j=\cL_{\mu_j,\lambda_j,b_j}$,
and
$J_j=J_{\mu_j,\lambda_j,b_j}$. With
$Y_j=\tau_j\psi_{T_{\mu_j}}+w_j$,  there exists
$q_j\in\mathbb R^{n+1}$ such that
\[
\cL_jY_j
=f_j+\sum_{\alpha=0}^n (q_j)_\alpha(-\Delta)(J_j)_\alpha,
\qquad
\|f_j\|_{H^{-1}_\delta}\to0.
\]
By the identities
\[
\cL_j(J_j)_0
=\left.\lambda\partial_\lambda\cR_{\mu_j,\lambda,b}
  \right|_{(\lambda,b)=(\lambda_j,b_j)},\qquad 
\cL_j(J_j)_i
=\left.\partial_{b_i}\cR_{\mu_j,\lambda,b}
  \right|_{(\lambda,b)=(\lambda_j,b_j)},\qquad 1\le i\le n,
\]
and Lemma~\ref{lem:residual}, we obtain
$\|\cL_j(J_j)_\beta\|_{H^{-1}_\delta}=o(1)$ for $0\le\beta\le n$.
Testing against the compactly supported $(J_j)_{\beta}$ and using symmetry,
we obtain
\[
\ip{\cL_jY_j}{(J_j)_{\beta}}
=\ip{Y_j}{\cL_j(J_j)_{\beta}}=o(1).
\]
The variations $Y_j$ are bounded in $H^1_{-\delta}$.
By \eqref{eq:gram-limit}, $q_j=o(1)$, and hence
\begin{equation}\label{eq:full-image-small}
\|\cL_jY_j\|_{H^{-1}_\delta}\to0.
\end{equation}
\step{Step 3. Vanishing of the rescaled limit.}
We show that every weak limit at the bubble scale is zero. We then apply
the inverse for the background linearization to control $(\tau_j,w_j)$
on all of $X$.

Choose a fixed cutoff in the original insertion coordinates
$\zeta_0\in C_c^\infty(B_{3\rho}(a))$, equal to one on a neighborhood
of $\overline{B_{2\rho}(a)}$. The phase cutoff vanishes on this larger
ball. In the fixed coordinates define
\[
\widetilde w_j(y)
 =\mu_j^{(n-2)/2}(\zeta_0 w_j)(a+\mu_jy),
\]
and extend it by zero to $\R^n$. Critical scaling and the fixed smooth
cutoff ensure the global bounds
\begin{equation}\label{eq:global-rescaled-energy}
\begin{split}
\|D\widetilde w_j\|_{L^2(\R^n)}
 &=\|D(\zeta_0 w_j)\|_{L^2(\R^n)}
   \le C\|w_j\|_{H^1_\delta}\le C,\\
\|\widetilde w_j\|_{L^{p+1}(\R^n)}&\le C.
\end{split}
\end{equation}
After passing to a subsequence,
\[
\widetilde w_j\rightharpoonup\widetilde w
\quad\text{in }\dot H^1(\R^n)\quad\text{as }j\to\infty.
\]
On every fixed rescaled ball, $\zeta_0(a+\mu_jy)=1$ for large $j$.
By critical rescaling and \eqref{eq:full-image-small}, testing on
that ball and passing to the limit yields
\[
-\Delta\widetilde w-n(n+2)\,U_{\lambda_\infty,b_\infty}^{p-1}
\widetilde w=0\quad\text{in }\R^n.
\]

We next rescale the Euclidean moment conditions. Let
\[
(\mathcal Z_j)_{\alpha}(y)
 :=\mu_j^{(n-2)/2}((J_j)_{\alpha})(a+\mu_jy)
 =\zeta(a+\mu_jy)(J_{\lambda_j,b_j})_{\alpha}(y).
\]
Since $\zeta_0=1$ on a neighborhood of the support of each
$(J_j)_\alpha$, scaling the Dirichlet product \eqref{eq:star-product} yields
\begin{equation}\label{eq:rescaled-moment-identity}
0=(w_j,(J_j)_{\alpha})_*
 =\int_{\R^n}D\widetilde w_j\cdot D(\mathcal Z_j)_{\alpha}\dd y.
\end{equation}
For $L$ sufficiently large and $j$ sufficiently large depending on $L$,
\[
\|D(\mathcal Z_j)_{\alpha}\|_{L^2(\R^n\setminus B_L)}
 \le CL^{-(n-2)/2}.
\]
Indeed, $D_y(J_{\lambda_j,b_j})_\alpha=O(|y|^{1-n})$.
The term obtained by differentiating the cutoff satisfies
\[
\left\|\mu_j(D_x\zeta)(a+\mu_j\,\cdot)
(J_{\lambda_j,b_j})_\alpha\right\|_{L^2(\R^n)}
\le C\mu_j^{(n-2)/2}.
\]
Combining this bound with \eqref{eq:global-rescaled-energy}, we obtain
\[
\limsup_{j\to\infty}
\left|\int_{\R^n\setminus B_L}
 D\widetilde w_j\cdot D(\mathcal Z_j)_{\alpha}\dd y\right|
 \le CL^{-(n-2)/2}.
\]
On $B_L$, the Jacobi factors converge smoothly to
$(J_{\lambda_\infty,b_\infty})_{\alpha}$. First pass to the limit in $j$
in \eqref{eq:rescaled-moment-identity} and then let $L\to\infty$ to obtain
\[
\int_{\R^n}D\widetilde w\cdot
 D (J_{\lambda_\infty,b_\infty})_{\alpha}\dd y=0,
 \qquad 0\le\alpha\le n.
\]

By the Euclidean kernel theorem, $\widetilde w$ belongs to
$\cK_{\lambda_\infty,b_\infty}$. The Jacobi orthogonality conditions
imply $\widetilde w=0$. By Rellich compactness,
$\widetilde w_j\to0$ in $L^{2n/(n+2)}$ on every fixed rescaled ball.

\step{Step 4. The contradiction in the full weighted norm.}
By the concavity of the power $4/(n-2)\in(0,1)$, we have
$0\le \mathcal V_j\le C B_j^{p-1}$.
Here $B_j=\widehat U_{\mu_j,\lambda_j,b_j}$ and $\mathcal V_j$ is the
corresponding additional potential. On a fixed rescaled ball $B_L$,
the rescaled coefficient $\mu_j^2B_j(a+\mu_j y)^{p-1}$ is bounded by
$C_L$. The critical
scaling of the source norm and the tail estimate
\[
\|B_j^{p-1}\|_{L^{n/2}(X\setminus B_{L\mu_j}(a))}\le CL^{-2}
\]
therefore imply
\begin{equation*}
\|\mathcal V_jw_j\|_{H^{-1}_\delta}
\le C_L\|\widetilde w_j\|_{L^{2n/(n+2)}(B_L)}
       +CL^{-2}\|w_j\|_{H^1_\delta}.
\end{equation*}

For the second term, use H\"older's inequality with $L^{n/2}$ and
$L^{p+1}$. The weights and their reciprocals have bounds independent
of $j$ on the fixed bubble support. First fix $L$ and let $j\to\infty$;
the first term tends to zero by the rescaled strong convergence.
Then let $L\to\infty$. We conclude that
$\|\mathcal V_jw_j\|_{H^{-1}_\delta}\to0$.
Since $\mathcal V_j\psi_{T_{\mu_j}}=0$, the identity
\[
\widehat\cL_{\gamma_{\mu_j},V_{\mu_j}}(\tau_j,w_j)
=\cL_jY_j+\mathcal V_jw_j
\]
and \eqref{eq:full-image-small} show that the left side tends to zero.
Its uniformly bounded inverse contradicts
$|\tau_j|+\|w_j\|_{H^1_\delta}=1$. This proves \eqref{eq:projected-inverse} on all of $X$.
Since the projected operator has Fredholm index zero, it is also surjective.
\end{proof}

We next solve the nonlinear projected equation and estimate its parameter derivatives.
We use the nonlinear phase family \eqref{eq:phase-family}. Throughout
this construction the
insertion point, radius, background, and phase cutoff are fixed. The
parameter set $K$ may be reduced so that $|b|\le1$.

For $(\tau,\phi)\in\cX_\delta$ with $|\tau|$ sufficiently small, set
\[
Y_{\mu,\lambda,b}(\tau,\phi)
:=V_\mu+\widehat U_{\mu,\lambda,b}+\Phi_{T_\mu,\tau}+\phi.
\]
In the parameter calculations below, write
$B=\widehat U_{\mu,\lambda,b}$,
$J_\alpha=(J_{\mu,\lambda,b})_\alpha$ for $0\le\alpha\le n$,
$G_\mu=G_{\mu,\lambda,b}$, and $\cR_\mu=\cR_{\mu,\lambda,b}$.

With $\mu,\lambda,b$ fixed, define the linearization at
$V_\mu+B$ and its nonlinear remainder by
\[
\begin{aligned}
\widehat\cL_\mu(\tau,\phi)
&:=\cL_{\mu,\lambda,b}(\tau\psi_{T_\mu}+\phi),\\
\cN_\mu(\tau,\phi)
&:=\cR_\mu+\widehat\cL_\mu(\tau,\phi)
 -\Fop_{\gamma_\mu}(Y_{\mu,\lambda,b}(\tau,\phi)).
\end{aligned}
\]

The reference $Y_{\mu,\lambda,b}(0,0)=V_\mu+B$ has residual $\cR_\mu$,
and $\cN_\mu(0,0)=D\cN_\mu(0,0)=0$.
There is $0<\varepsilon_0\le1$, independent of small $\mu$ and of
$(\lambda,b)\in K$, such that the map
$(\tau,\phi)\mapsto\Fop_{\gamma_\mu}(Y_{\mu,\lambda,b}(\tau,\phi))$
is $C^1$ from $(-\varepsilon_0,\varepsilon_0)\times H^1_\delta$
to $H^{-1}_\delta$, and
\begin{equation}\label{eq:inserted-derivative-comparison}
\bigl\|D_{(\tau,\phi)}
  [\Fop_{\gamma_\mu}(Y_{\mu,\lambda,b}(\tau,\phi))]
  -\widehat\cL_\mu\bigr\|_{\cX_\delta\to H^{-1}_\delta}
\le C\bigl(|\tau|+\|\phi\|_{H^1_\delta}
                         +\|\phi\|_{H^1_\delta}^{p-1}\bigr).
\end{equation}

For $0<\varepsilon<\varepsilon_0$ and
$\|(\tau,\phi)\|_{\cX_\delta},
 \|(\widetilde\tau,\widetilde\phi)\|_{\cX_\delta}\le\varepsilon$,

\begin{equation}\label{eq:inserted-remainder-bound}
\|\cN_\mu(\tau,\phi)-\cN_\mu(\widetilde\tau,\widetilde\phi)\|_{H^{-1}_\delta}
\le C(\varepsilon+\varepsilon^{p-1})
\|(\tau-\widetilde\tau,\phi-\widetilde\phi)\|_{\cX_\delta}.
\end{equation}

The constants are independent of small $\mu$ and uniform on $K$,
with the background, insertion ball, and cutoffs fixed.
By Lemma~\ref{lem:background-continuation}, the continued backgrounds
have uniform weighted bounds. Critical scaling gives
\[
\|B\|_{L^{p+1}(B_{2\rho}(a))}\le C,\qquad
\|(V_\mu+B)^{p-1}\|_{L^{n/2}(B_{2\rho}(a))}\le C.
\]

The phase family vanishes on a neighborhood of $\overline{B_{2\rho}(a)}$.
On this ball, the first scalar estimate in \eqref{eq:scalar-Nemytskii}
and the multiplication estimate \eqref{eq:weighted-multiplication}
control the derivative in $\phi$ uniformly in $\mu,\lambda,b$.
Outside the bubble support, the proof of Lemma~\ref{lem:nonlinear-map}
applies to $V_\mu$ and $T_\mu$.
These bounds prove \eqref{eq:inserted-derivative-comparison}.
Integrating $D\cN_\mu$ along a segment proves
\eqref{eq:inserted-remainder-bound}.

The projected equation is
\begin{equation}\label{eq:projected-equation}
\Pi^{\perp,*}\Fop_{\gamma_\mu}
       (Y_{\mu,\lambda,b}(\tau,\phi))=0,
\qquad \phi\perp_*\cK_{\mu,\lambda,b}.
\end{equation}

For a solution family $(\tau_\mu,\phi_\mu)$ constructed in
Proposition~\ref{prop:projected-solution}, define the multipliers
$c_\mu(\lambda,b)\in\R^{n+1}$ by
\begin{equation}\label{eq:full-projected}
\Fop_{\gamma_\mu}(Y_{\mu,\lambda,b}(\tau_\mu,\phi_\mu))
=\sum_{\alpha=0}^n (c_\mu)_{\alpha}(\lambda,b)(-\Delta) (J_{\mu,\lambda,b})_{\alpha}.
\end{equation}
Define the reduced relative energy by
\begin{equation}\label{eq:reduced-energy}
\widetilde\cJ_\mu(\lambda,b)
:=\cJ_{\gamma_\mu,V_\mu}
       (\tau_\mu,\widehat U_{\mu,\lambda,b}+\phi_\mu).
\end{equation}
\begin{proposition}\label{prop:projected-solution}
For some $C_*>0$, all sufficiently small $\mu$, and all $(\lambda,b)\in K$,
\eqref{eq:projected-equation} has a unique solution
$(\tau_\mu,\phi_\mu)$ in the ball
$|\tau|+\|\phi\|_{H^1_\delta}\le C_*\mu^{27/4}$.
For each fixed $\mu$, the maps
$(\lambda,b)\mapsto(\tau_\mu,\phi_\mu)\in\mathcal X_\delta$
and $(\lambda,b)\mapsto c_\mu$ are $C^1$ and satisfy
\begin{align}\label{eq:correction-bounds}
|\tau_\mu|+\|\phi_\mu\|_{H^1_\delta}+|c_\mu|
 &\le C\bigl(\mu^{27/4}|b|^2+e_\mu\bigr),\\
|D_{\lambda,b}\tau_\mu|
+\|D_{\lambda,b}\phi_\mu\|_{H^1_\delta}
+|D_{\lambda,b}c_\mu|
 &\le C\bigl(\mu^{27/4}|b|+e_\mu\bigr).\label{eq:correction-derivative}
\end{align}
The reduced energy $\widetilde\cJ_\mu$ is $C^2$.
\end{proposition}
\begin{proof}
\step{Step 1. Contraction.}
Use the maps $\widehat\cL_\mu$ and $\cN_\mu$ defined before the
proposition and the uniform estimate \eqref{eq:inserted-remainder-bound}.
The inverse in Lemma~\ref{lem:projected-inverse} defines the map
\[
\mathcal T_\mu(\tau,\phi)
:=(\Pi^{\perp,*}\widehat\cL_\mu)^{-1}
       \Pi^{\perp,*}\bigl(-\cR_\mu+\cN_\mu(\tau,\phi)\bigr)
\]
on $\R\oplus\cK_{\mu,\lambda,b}^{\perp_*}$. On the closed ball centered
at $(0,0)$ of $\cX_\delta$ radius $0<\varepsilon<\varepsilon_0$,
\eqref{eq:residual-bounds} and \eqref{eq:inserted-remainder-bound} imply
\[
\|\mathcal T_\mu(0,0)\|_{\cX_\delta}
 \le C\bigl(\mu^{27/4}|b|^2+e_\mu\bigr)
\quad\text{and}\quad
\operatorname{Lip}(\mathcal T_\mu)\le C(\varepsilon+\varepsilon^{p-1}).
\]
Choose $C_*\ge4C(1+\sup_K|b|^2)$ and set $\varepsilon=C_*\mu^{27/4}$.
Then decrease $\mu$ so that $e_\mu\le\mu^{27/4}$,
$\varepsilon<\varepsilon_0$, and $C(\varepsilon+\varepsilon^{p-1})<1/2$.
Thus $\|\mathcal T_\mu(0,0)\|_{\cX_\delta}\le\varepsilon/2$, and
$\mathcal T_\mu$ is a contraction of the closed ball of radius
$\varepsilon$ into itself. Its unique fixed point solves
\eqref{eq:projected-equation}. Absorbing the nonlinear term proves the
sharper bound in \eqref{eq:correction-bounds}. The residual lies in the
span of the $(-\Delta)J_\alpha$, and the uniformly invertible Gram matrix
controls the multipliers in \eqref{eq:full-projected}.

\step{Step 2. The full differential.}
To differentiate the solution, include the moment constraints and
multipliers as components of the nonlinear system.
Write $Y=Y_{\mu,\lambda,b}(\tau_\mu,\phi_\mu)$ and
$c_\alpha=(c_\mu)_\alpha$ for $0\le\alpha\le n$. The derivative in
$(\tau,\phi,c)$ sends $(\dot\tau,\dot\phi,\dot c)$ to
\[
\Big(
 [-L_{\gamma_\mu}-n(n+2)(Y_+)^{p-1}]
       (\dot\tau\psi_{T_\mu+\tau_\mu}+\dot\phi)
       -\sum_{\alpha=0}^n(\dot c)_\alpha (-\Delta)J_\alpha,
 \bigl((\dot\phi,J_\beta)_*\bigr)_{\beta=0}^n
\Big).
\]

For a source $f$ and prescribed moments $m\in\R^{n+1}$,
the inverse at zero correction is determined by
\[
\begin{aligned}
\dot\phi^\parallel
 &=\sum_{\alpha,\beta=0}^n(G_\mu^{-1})_{\alpha\beta}m_\beta J_\alpha,\\
(\dot\tau,\dot\phi^\perp)
 &=(\Pi^{\perp,*}\widehat\cL_\mu)^{-1}\Pi^{\perp,*}
       \bigl(f-\widehat\cL_\mu(0,\dot\phi^\parallel)\bigr),\\
\dot c
 &=G_\mu^{-1}\bigl(\langle\widehat\cL_\mu(\dot\tau,\dot\phi)-f,
                    J_\beta\rangle\bigr)_{\beta=0}^n,
\qquad \dot\phi=\dot\phi^\parallel+\dot\phi^\perp.
\end{aligned}
\]
By \eqref{eq:weighted-multiplication}, the Gram bound, and
Lemma~\ref{lem:projected-inverse}, we obtain
\[
|\dot\tau|+\|\dot\phi\|_{H^1_\delta}+|\dot c|
\le C\bigl(\|f\|_{H^{-1}_\delta}+|m|\bigr).
\]
By \eqref{eq:inserted-derivative-comparison}, the derivative of
the full system in $(\tau,\phi,c)$ at the constructed solution differs
from its derivative at zero correction by an operator of norm
$O(\varepsilon+\varepsilon^{p-1})$, with $\varepsilon=C_*\mu^{27/4}$.
A Neumann series preserves the inverse. At the constructed solution of the projected equation,
write $\cL_Y:=-L_{\gamma_\mu}-n(n+2)(Y_+)^{p-1}$. The resulting bound is
\begin{equation}\label{eq:full-differential-inverse}
|\dot\tau|+\|\dot\phi\|_{H^1_\delta}+|\dot c|
\le C\Big\|\cL_Y(\dot\tau\psi_{T_\mu+\tau_\mu}+\dot\phi)
-\sum_{\alpha=0}^n(\dot c)_\alpha(-\Delta)J_\alpha\Big\|_{H^{-1}_\delta}
+C\left|\bigl((\dot\phi,J_\beta)_*\bigr)_{\beta=0}^n\right|.
\end{equation}
The implicit function theorem now gives $C^1$ dependence of
$(\tau_\mu,\phi_\mu,c_\mu)$ on $(\lambda,b)$.

We next estimate the parameter derivatives of the solution and
multipliers. Differentiating the full system with respect to
$(\lambda,b)$ gives
\begin{align*}
\cL_Y\bigl((D_{\lambda,b}\tau_\mu)\psi_{T_\mu+\tau_\mu}
                   +D_{\lambda,b}\phi_\mu\bigr)
 -\sum_{\alpha=0}^n(D_{\lambda,b}c_\alpha)(-\Delta)J_\alpha
 &=-\cL_YD_{\lambda,b}B
   +\sum_{\alpha=0}^n c_\alpha(-\Delta)D_{\lambda,b}J_\alpha,\\
(D_{\lambda,b}\phi_\mu,J_j)_*
 &=-(\phi_\mu,D_{\lambda,b}J_j)_*,\qquad 0\le j\le n.
\end{align*}
These identities are read componentwise in the order
$\lambda,b_1,\ldots,b_n$.
Since $\Phi_{T_\mu,\tau_\mu}=0$ on a neighborhood of
$\overline{B_{2\rho}(a)}$, which contains the support of $B$ and its
parameter derivatives, we have
\[
\cL_YD_{\lambda,b}B=D_{\lambda,b}\cR_\mu
 -n(n+2)\{(V_\mu+B+\phi_\mu)_+^{p-1}-(V_\mu+B)^{p-1}\}
            D_{\lambda,b}B.
\]
The scalar bound $|(1+t)_+^{p-1}-1|\le C|t|$ and
$|D_{\lambda,b}B|\le C(V_\mu+B)$ imply
\[
\left|\{(V_\mu+B+\phi_\mu)_+^{p-1}-(V_\mu+B)^{p-1}\}
                  D_{\lambda,b}B\right|
\le C(V_\mu+B)^{p-1}|\phi_\mu|.
\]
The $L^{n/2}$ norm of $(V_\mu+B)^{p-1}$ on this support is uniformly
bounded. 
By H\"older's inequality and the Sobolev embedding, the
$H^{-1}_\delta$ norm of
$\{(V_\mu+B+\phi_\mu)_+^{p-1}-(V_\mu+B)^{p-1}\}D_{\lambda,b}B$
is at most $C\|\phi_\mu\|_{H^1_\delta}$.
The Jacobi fields and their first two parameter derivatives are
uniformly bounded in $H^1_\delta$. Applying
\eqref{eq:full-differential-inverse} and using
\eqref{eq:residual-bounds} and \eqref{eq:correction-bounds}, we obtain
\[
\begin{aligned}
|D_{\lambda,b}\tau_\mu|+\|D_{\lambda,b}\phi_\mu\|_{H^1_\delta}
 +|D_{\lambda,b}c_\mu|
&\le C\bigl(\mu^{27/4}|b|+e_\mu
                  +\|\phi_\mu\|_{H^1_\delta}+|c_\mu|\bigr)\\
&\le C\bigl(\mu^{27/4}|b|+e_\mu\bigr).
\end{aligned}
\]
Here $|b|^2\le|b|$ on the chosen parameter set. This proves
\eqref{eq:correction-derivative}.

\step{Step 3. The reduced gradient.}
Since $\chi=0$ on a neighborhood of $\overline{B_{2\rho}(a)}$,
$\Phi_{T_\mu,\tau}$ and its phase derivatives vanish on the support
of every $(-\Delta)J_\alpha$, $0\le\alpha\le n$.
For $0\le\alpha\le n$ and $1\le i\le n$, define
\[
(M_\mu)_{\alpha0}
 =(J_\alpha,J_0)_*-\lambda(\partial_\lambda J_\alpha,\phi_\mu)_*,\qquad 
(M_\mu)_{\alpha i}
 =(J_\alpha,J_i)_*-(\partial_{b_i}J_\alpha,\phi_\mu)_*.
\]
By \eqref{eq:relative-first-variation} and \eqref{eq:full-projected},
using $\partial_\lambda B=J_0/\lambda$,
$\partial_{b_i}B=J_i$, and the derivatives of
$(\phi_\mu,J_\alpha)_*=0$, we obtain
\begin{equation}\label{eq:reduced-gradient}
\nabla_{(\lambda,b)}\widetilde\cJ_\mu
=\operatorname{diag}(\lambda^{-1},1,\ldots,1)M_\mu^T c_\mu.
\end{equation}
The matrix $M_\mu=G_\mu+O(\mu^{27/4})$ is invertible.
Its entries and the multipliers are $C^1$. Hence the ordinary
gradient is $C^1$, and the reduced energy is $C^2$.
\end{proof}

Since $M_\mu$ is invertible, \eqref{eq:reduced-gradient} shows that
$c_\mu=0$ exactly at the critical points of $\widetilde\cJ_\mu$.
We next find a nondegenerate critical point, which also yields
an invertible linearization at the new solution.

\section{Reduced energy and bubble parameters}\label{sec:projected-energy}

In this section, we choose the bubble parameters at a nondegenerate
minimum of the reduced energy. Proposition~\ref{prop:local-well}
identifies the minimum of the leading coefficient, and
Lemmas~\ref{lem:correction-hessian} and \ref{lem:pure-energy}
yield the Hessian comparison used in Proposition~\ref{prop:reduced-minimum}.

The leading quadratic metric energy is determined by
\begin{equation}\label{eq:F-definition}
F(\lambda,b):=\sum_{i,k,l=1}^n\int_{\R^n}
\Bigl[\frac12\cH_{il}\cH_{kl}\partial_{y_i}U_{\lambda,b}\partial_{y_k}U_{\lambda,b}
-\frac{n-2}{16(n-1)}(\partial_{y_l}\cH_{ik})^2U_{\lambda,b}^2\Bigr]\dd y.
\end{equation}
With the factor $1/2$ in front of the Dirichlet integral in the energy,
the quadratic metric coefficient is $F/2$.
Lemma~\ref{lem:correction-hessian} uses the quadratic vanishing in
\eqref{eq:center-source-bounds} to show that the projected correction
contributes $o(\mu^{27/2})$ to the parameter Hessian at $b=0$.

\begin{proposition}\label{prop:local-well}
For every $n\ge30$, there exists $s_n\in(9/10,11/10)$ such that
$(\sqrt{s_n},0)$ is a nondegenerate strict local minimum of the
function $F$ defined in \eqref{eq:F-definition}.
\end{proposition}
The proof is given in Subsection~\ref{app:coefficient-tables}.

We now estimate the change in reduced energy caused by the correction:
\[
D_\mu(\lambda,b):=\widetilde\cJ_\mu(\lambda,b)
-\cJ_{\gamma_\mu,V_\mu}(0,\widehat U_{\mu,\lambda,b}).
\]

\begin{lemma}\label{lem:correction-hessian}
For all sufficiently small $\mu$, the following estimates
\begin{equation}\label{eq:correction-hessian}
\|D^2_{\lambda,b}D_\mu(\lambda,b)\|
\le C\bigl(\mu^{27/2}|b|^2+\mu^{27/4}e_\mu+e_\mu^2\bigr)
\quad\text{and}\quad |\nabla D_\mu(\lambda,0)|\le Ce_\mu^2,
\end{equation}
hold uniformly for $(\lambda,b)\in K$ ($(\lambda,0)\in K$ for the second estimate).
\end{lemma}

\begin{proof}
\step{Step 1. The gradient identity.}
Write $\phi=\phi_\mu$ and retain the abbreviations $B$, $J_\alpha$,
and $\cR_\mu$ fixed before Proposition~\ref{prop:projected-solution}.
On the support of $B$ and its parameter derivatives, we have
\[
\cN_\mu(0,\phi)
=n(n-2)\bigl\{(V_\mu+B+\phi)_+^p-(V_\mu+B)^p
-p(V_\mu+B)^{p-1}\phi\bigr\}.
\]
Both $\Phi_{T_\mu,\tau_\mu}$ and $\psi_{T_\mu+\tau_\mu}$ vanish
on a neighborhood of $\overline{B_{2\rho}(a)}$.
Let $Y=Y_{\mu,\lambda,b}(\tau_\mu,\phi_\mu)$ and
$c_\alpha=(c_\mu)_\alpha$. Applying
\eqref{eq:relative-first-variation}, with parameter derivatives and
pairings taken componentwise, gives
\[
\begin{aligned}
D_{\lambda,b}D_\mu
={}&\langle\Fop_{\gamma_\mu}(Y)-\Fop_{\gamma_\mu}(V_\mu+B),
                    D_{\lambda,b}B\rangle\\
&+\sum_{\alpha=0}^n c_\alpha(J_\alpha,D_{\lambda,b}\phi)_*.
\end{aligned}
\]
Differentiating $(\phi,J_\alpha)_*=0$ and using the equation on
$B_{2\rho}(a)$, we obtain
\[
(J_\alpha,D_{\lambda,b}\phi)_*=-(D_{\lambda,b}J_\alpha,\phi)_*,
\qquad
\Fop_{\gamma_\mu}(Y)-\Fop_{\gamma_\mu}(V_\mu+B)
 =\cL\phi-\cN_\mu(0,\phi),
\]
where $\cL=\cL_{\mu,\lambda,b}$.
Integration by parts against each component of $D_{\lambda,b}B$,
using $\cL D_{\lambda,b}B=D_{\lambda,b}\cR_\mu$, yields
\begin{equation}\label{eq:correction-gradient-identity}
D_{\lambda,b}D_\mu
=\ip{\phi}{D_{\lambda,b}\cR_\mu}
-\ip{\cN_\mu(0,\phi)}{D_{\lambda,b}B}
-\sum_{\alpha=0}^n c_\alpha(D_{\lambda,b}J_\alpha,\phi)_*.
\end{equation}

\step{Step 2. Bounds for the derivatives.}
For every $t\in\R$, the following scalar estimates hold:
\[
|(1+t)_+^p-1-pt|\le Ct^2,\quad
|(1+t)_+^{p-1}-1|\le C|t|,\quad
|(1+t)_+^{p-1}-1-(p-1)t|\le Ct^2.
\]
The derivative $D_{\lambda,b}\cN_\mu(0,\phi_\mu)$ acts on both
$B$ and $\phi_\mu$, while $V_\mu$ is independent of $(\lambda,b)$.
Thus
\[
\begin{aligned}
\frac{D_{\lambda,b}\cN_\mu(0,\phi)}{n(n+2)}
={}&\bigl[(V_\mu+B+\phi)_+^{p-1}-(V_\mu+B)^{p-1}\bigr]
                                         D_{\lambda,b}\phi\\
&+\bigl[(V_\mu+B+\phi)_+^{p-1}-(V_\mu+B)^{p-1}
       -(p-1)(V_\mu+B)^{p-2}\phi\bigr]D_{\lambda,b}B.
\end{aligned}
\]
The explicit bubble satisfies
$|D_{\lambda,b}B|+|D^2_{\lambda,b}B|\le C(V_\mu+B)$.
Consequently, on its fixed support,
\[
\begin{aligned}
|D_{\lambda,b}\cN_\mu(0,\phi)|\,|D_{\lambda,b}B|
 +|\cN_\mu(0,\phi)|\,|D^2_{\lambda,b}B|
\le C(V_\mu+B)^{p-1}
       \bigl(|\phi|\,|D_{\lambda,b}\phi|+|\phi|^2\bigr).
\end{aligned}
\]
The $L^{n/2}$ norm of $(V_\mu+B)^{p-1}$ on this support is uniformly
bounded. By H\"older's inequality and the critical Sobolev embedding,
we obtain
\[
\begin{aligned}
|\ip{\cN_\mu(0,\phi)}{D_{\lambda,b}B}|
 &\le C\|\phi\|_{H^1_\delta}^2,\\
\bigl\|D_{\lambda,b}\ip{\cN_\mu(0,\phi)}{D_{\lambda,b}B}\bigr\|
 &\le C\bigl(\|\phi\|_{H^1_\delta}^2
       +\|\phi\|_{H^1_\delta}\|D_{\lambda,b}\phi\|_{H^1_\delta}\bigr).
\end{aligned}
\]
Differentiating
\eqref{eq:correction-gradient-identity} involves only
the first parameter derivatives of $\phi_\mu$ and $c_\mu$.
Using the preceding estimates and the bounds for the first
two parameter derivatives of $B$, $\cR_\mu$, and $J_\alpha$,
we obtain
\[
\begin{aligned}
\|D^2_{\lambda,b}D_\mu\|\le C\Bigl\{
 &\|D_{\lambda,b}\phi\|_{H^1_\delta}
       \|D_{\lambda,b}\cR_\mu\|_{H^{-1}_\delta}
 +\|\phi\|_{H^1_\delta}\|D^2_{\lambda,b}\cR_\mu\|_{H^{-1}_\delta}
 +\|\phi\|_{H^1_\delta}^2\\
 &+\|\phi\|_{H^1_\delta}\|D_{\lambda,b}\phi\|_{H^1_\delta}
 +|D_{\lambda,b}c_\mu|\|\phi\|_{H^1_\delta}
 +|c_\mu|\bigl(\|\phi\|_{H^1_\delta}
              +\|D_{\lambda,b}\phi\|_{H^1_\delta}\bigr)\Bigr\}.
\end{aligned}
\]
Norms of parameter derivative collections are maximum norms over
their components. Substituting
\eqref{eq:residual-bounds}--\eqref{eq:correction-derivative}
proves the first estimate in \eqref{eq:correction-hessian}.
At $b=0$, the quantities $\phi$, $D_{\lambda,b}\cR_\mu$, and $c_\mu$
are $O(e_\mu)$. Every term in
\eqref{eq:correction-gradient-identity} is therefore $O(e_\mu^2)$,
which proves the second estimate in \eqref{eq:correction-hessian}.
\end{proof}

It remains to expand the relative energy before adding the projected
correction. Let
\[
\cJ_0(U)=\frac12\int_{\R^n}|DU|^2\dd y
-\frac{(n-2)^2}{2}\int_{\R^n}(U_+)^{p+1}\dd y.
\]
The energy of every $U_{\lambda,b}$ is $\cJ_0(U_{1,0})$.

\begin{lemma}\label{lem:pure-energy}
As $\mu\downarrow0$, with the insertion radius $\rho$ fixed
and $R=\rho/\mu$,
\begin{equation}\label{eq:pure-energy-C2}
\Big\|\cJ_{\gamma_\mu,V_\mu}(0,\widehat U_{\mu,\cdot,\cdot})
-\cJ_0(U_{1,0})-\frac{\mu^{27/2}}2F\Big\|_{C^2(K)}
\le C\bigl(\mu^{(n-2)/2}+\mu^{n-2}+\mu^{27/2}R^{14-n}
+\mu^{81/4}\bigr)=o(\mu^{27/2}).
\end{equation}
\end{lemma}

\begin{proof}
The bubble is compactly supported and the phase is unchanged. Hence the
relative energy in \eqref{eq:pure-energy-C2} is a finite energy increment.
The bubble vanishes near both boundary spheres. The outer Robin
term and phase subtraction cancel in the difference, leaving the bulk
energy for $\gamma_\mu$ on $B_{2\rho}(a)$, with volume element $\dd x$.

\step{Step 1. The background interaction.}
In the same local Euclidean representation as in
Lemma~\ref{lem:residual}, write $A=V_\mu>0$ and
$B=\widehat U_{\mu,\lambda,b}=\zeta U_{\mu\lambda,a+\mu b}\ge0$. The background
$A$ is independent of $(\lambda,b)$ and has uniform positive upper
and lower bounds on $B_{2\rho}(a)$. Testing $-L_{\gamma_\mu}A=n(n-2)A^p$ with $B$ gives
\[
\int_{B_{2\rho}(a)}
\bigl(\langle DA,DB\rangle_{\gamma_\mu}
+c(n)\Scal_{\gamma_\mu}AB\bigr)\dd x
=n(n-2)\int_{B_{2\rho}(a)}A^pB\,\dd x.
\]
Thus the difference between the
full increment and the pure bubble energy is
\[
n(n-2)\int_{B_{2\rho}(a)} A^pB\dd x
 -\frac{(n-2)^2}{2}\int_{B_{2\rho}(a)}
    \bigl((A+B)^{p+1}-A^{p+1}-B^{p+1}\bigr)\dd x.
\]

The volume form in these variables is Euclidean because
$\operatorname{tr}h_\mu=0$. Using
\[
0\le(A+B)^{p+1}-A^{p+1}-B^{p+1}
\le C(A^pB+AB^p)
\]
and $|\partial_{(\lambda,b)}^\beta B|\le CB$ for $|\beta|\le2$, we bound
this expression and its first two parameter derivatives by
$C\int_{B_{2\rho}(a)}(B+B^p)\dd x=O(\mu^{(n-2)/2})$.
To estimate the parameter derivatives, use the ratios $B/A$ on
$\{B\le A\}$ and $A/B$ on $\{B\ge A\}$, as in the proof of
Lemma~\ref{lem:residual}. At a zero of the cutoff, $B$ and its parameter
derivatives vanish.

\step{Step 2. The metric contribution.}
We compute the metric contribution in the coordinates $x=a+\mu y$. On the bubble scale let
$\bar h(y):=h_\mu(a+\mu y)=\mu^{27/4}\cH_R(y)$, where
$\cH_R=\eta(|y|/R)\cH$ and $R=\rho/\mu$.
Since the trace and divergence of $\cH_R$ vanish, the inverse metric
and scalar curvature expansions in
Brendle--Marques~\cite[Proposition~4]{BrendleMarques2009} reduce to
\begin{align*}
(e^{-\bar h})_{ik}
 &=\delta_{ik}-(\bar h)_{ik}
   +\frac12\sum_{l=1}^n(\bar h)_{il}(\bar h)_{lk}+O(|\bar h|^3),
   \qquad 1\le i,k\le n,\\
\Scal_{e^{\bar h}}
 &=-\frac14|D\bar h|^2
   +O(|\bar h|^2|D^2\bar h|+|\bar h||D\bar h|^2).
\end{align*}

In the general curvature formula, the quadratic divergence terms vanish
because $\sum_{i=1}^n\partial_{y_i}(\bar h)_{ik}=0$ for $1\le k\le n$.
On the support of $\bar h$, we have
$\mu^{(n-2)/2}\widehat U_{\mu,\lambda,b}(a+\mu y)=U_{\lambda,b}(y)$.
Equation~\eqref{eq:linear-energy-zero} eliminates the linear metric energy. The quadratic part is
\[
\mu^{27/2}\bigg\{
 \frac14\sum_{i,k,l=1}^n\int_{\R^n}
    (\cH_R)_{il}(\cH_R)_{kl}\partial_{y_i}U_{\lambda,b}\partial_{y_k}U_{\lambda,b}\dd y
 -\frac{c(n)}8\int_{\R^n}|D\cH_R|^2U_{\lambda,b}^2\dd y
\bigg\}.
\]
The corresponding expression with $\cH$ in place of $\cH_R$ is
$\mu^{27/2}F(\lambda,b)/2$.

\step{Step 3. The remainder terms.}
In the $y$ coordinates, put $s=|y|$. For $s\ge1$, the quadratic energy integrands in Step 2 have radial
majorant $Cs^{13-n}$, also after two parameter derivatives.
Derivatives of the radial cutoff have the same order on its transition
annulus. The error from removing this cutoff is bounded by
$C\mu^{27/2}R^{14-n}$ in $C^2(K)$.

The cubic error and its first two parameter derivatives are bounded by
\[
C\mu^{81/4}\left(1+\int_1^R s^{19-n}\dd s\right)
\le C\mu^{81/4},
\]
since $n\ge30$. Here the two majorants before radial integration are
$|\cH_R|^3|DU_{\lambda,b}|^2$ and
$(|\cH_R|^2|D^2\cH_R|+|\cH_R||D\cH_R|^2)U_{\lambda,b}^2$.
Since $\det(e^{\bar h})=1$, the nonlinear energy integral is unchanged.

Finally, replacing $U_{\lambda,b}$ by $\zeta(a+\mu y)U_{\lambda,b}$ in
its flat energy produces an error of $CR^{2-n}=O(\mu^{n-2})$, also after two parameter
derivatives: use $U=O(s^{2-n})$, $DU=O(s^{1-n})$ and the cutoff derivative
bound $C/R$. Adding the background, quadratic cutoff, cubic, and flat
cutoff errors proves \eqref{eq:pure-energy-C2}. Every error divided by
$\mu^{27/2}$ tends to zero as $\mu\downarrow0$, since
$(n-2)/2-27/2=(n-29)/2>0$.
\end{proof}

\begin{proposition}\label{prop:reduced-minimum}
There are $r,c,\mu_0>0$, with
$\overline{B_r(\sqrt{s_n},0)}\subset\operatorname{int}K$, such that, for $0<\mu<\mu_0$,
$\widetilde\cJ_\mu$ has a unique critical point
$(\lambda_\mu,b_\mu)$ in the ball $B_r(\sqrt{s_n},0)$.
It is a nondegenerate strict local minimum, and
\begin{equation}\label{eq:reduced-minimum}
(\lambda_\mu,b_\mu)\to(\sqrt{s_n},0)
\quad\text{as }\mu\downarrow0,\qquad
D^2\widetilde\cJ_\mu\ge c\mu^{27/2}I
\quad\text{on }B_r(\sqrt{s_n},0).
\end{equation}
\end{proposition}
\begin{proof}
\step{Step 1. The Hessian bound.}
Combining Lemmas~\ref{lem:correction-hessian} and \ref{lem:pure-energy}, we obtain
\[
\left\|\mu^{-27/2}D^2\widetilde\cJ_\mu
-\frac12D^2F\right\|\le C|b|^2+o(1).
\]
At the center, the gradient bound in \eqref{eq:correction-hessian} and
\eqref{eq:pure-energy-C2} imply
$\mu^{-27/2}\nabla\widetilde\cJ_\mu(\sqrt{s_n},0)\to0$.
Here $e_\mu=o(\mu^{27/4})$ by Lemma~\ref{lem:residual}.
By Proposition~\ref{prop:local-well} and continuity of $D^2F$, choose
$c>0$ and then $r>0$ so that
\[
\overline{B_r(\sqrt{s_n},0)}\subset\operatorname{int}K,
\qquad \tfrac12D^2F\ge3cI
\quad\text{on }\overline{B_r(\sqrt{s_n},0)},\qquad Cr^2\le c.
\]
Fix this parameter radius for the current background and cutoff.
The bound $Cr^2\le c$ controls $C|b|^2$; we then decrease $\mu$ to
control the remaining $o(1)$ term.
Keep $r$ fixed and take $\mu_0$ sufficiently small. Then, for
$0<\mu<\mu_0$,
\[
D^2\widetilde\cJ_\mu\ge c\mu^{27/2}I
\quad\text{on }B_r(\sqrt{s_n},0),\qquad
|\nabla\widetilde\cJ_\mu(\sqrt{s_n},0)|
\le\tfrac12cr\mu^{27/2}.
\]
\step{Step 2. The critical point.}
For $|(\lambda,b)-(\sqrt{s_n},0)|=r$, integrating the Hessian bound along the line segment,
we obtain
\[
\nabla\widetilde\cJ_\mu(\lambda,b)\cdot(\lambda-\sqrt{s_n},b)
\ge c\mu^{27/2}r^2-r|\nabla\widetilde\cJ_\mu(\sqrt{s_n},0)|
\ge\frac c2\mu^{27/2}r^2>0.
\]
A minimum on the closed ball is therefore attained in its interior.
Strict convexity proves uniqueness, and the positive lower bound for
$D^2\widetilde\cJ_\mu$ proves nondegeneracy. At the critical point, the Hessian bound gives
\[
|(\lambda_\mu,b_\mu)-(\sqrt{s_n},0)|
\le c^{-1}\mu^{-27/2}
|\nabla\widetilde\cJ_\mu(\sqrt{s_n},0)|
\to0
\quad\text{as }\mu\downarrow0.
\]
\end{proof}

\section{Proof of the insertion proposition}\label{sec:insertion-proof}

In this section, we obtain a positive solution from the reduced
critical point and prove that its augmented linearization is invertible.
These properties complete Proposition~\ref{prop:one-bubble} and allow
the induction to continue.

Fix the admissible pair and ball from Proposition~\ref{prop:one-bubble}.
Evaluate $(\tau_\mu,\phi_\mu)$ at the critical point
$(\lambda_\mu,b_\mu)$ from Proposition~\ref{prop:reduced-minimum}, and set
\begin{equation}\label{eq:one-bubble}
Y_\mu:=V_\mu+\widehat U_{\mu,\lambda_\mu,b_\mu}
               +\Phi_{T_\mu,\tau_\mu}+\phi_\mu, \quad \widehat T_\mu:=T_\mu+\tau_\mu.
\end{equation}
\begin{proof}[Proof of Proposition~\ref{prop:one-bubble}]
\step{Step 1. The metric, phase, and exact equation.}
The construction \eqref{eq:metric-profile} and estimate
\eqref{eq:metric-small} give the metric support, volume, and convergence.
Lemma~\ref{lem:background-continuation} gives $T_\mu\to T$, while
\eqref{eq:correction-bounds} gives $\tau_\mu\to0$.
Hence $\widehat T_\mu\to T$ as $\mu\downarrow0$.
At the critical point, \eqref{eq:reduced-gradient} and invertibility
of $M_\mu$ give $c_\mu=0$. Hence $\Fop_{\gamma_\mu}(Y_\mu)=0$ in $H^{-1}_\delta$.
The equation also holds against smooth test functions vanishing
near $0$ and having antipodally even boundary traces.

\step{Step 2. Positivity, regularity, and exterior convergence.}
Lemma~\ref{lem:ground-state}, applied to the continued positive
solution $V_\mu$, makes $Q_{\gamma_\mu}$ coercive.
For each fixed small $\mu$,
$A_\mu=V_\mu+\widehat U_{\mu,\lambda_\mu,b_\mu}+\Phi_{T_\mu,\tau_\mu}$
satisfies $c\Theta\le A_\mu\le C\Theta$ and
$|DA_\mu|\le C\Theta/r$, with constants allowed to depend on $\mu$.
The bubble is nonnegative. The phase perturbation satisfies
$|\Phi_{T_\mu,\tau_\mu}|\le C|\tau_\mu|\Theta$ and vanishes on
$\overline{B_{2\rho}(a)}$. All terms in $A_\mu$ satisfy boundary matching.
Since $Y_\mu=A_\mu+\phi_\mu$ with $\phi_\mu\in H^1_\delta$,
Lemma~\ref{lem:regularity-end} gives positivity and smoothness.
To identify the final phase, let $w_\mu$ be the correction in
Lemma~\ref{lem:background-continuation} and $w_0$ the local difference
in \eqref{eq:background}. Near $0$, the bubble vanishes and $\chi=1$,
so the two construction steps give
\[
\begin{aligned}
V_\mu&=V_{T_\mu}+(w_0+w_\mu),\\
Y_\mu&=V_{\widehat T_\mu}+(w_0+w_\mu+\phi_\mu).
\end{aligned}
\]
Thus $Y_\mu-V_{\widehat T_\mu}\in H^1_\delta$ near $0$.
By Lemma~\ref{lem:regularity-end},
$Y_\mu-V_{\widehat T_\mu}\in\cC^{k,\alpha}_{\Theta;\delta}$
there for every fixed $k\ge0$ and $0<\alpha<1$.
Hence $Y_\mu$ has Fowler phase $\widehat T_\mu$.

On each compact subset of $\overline{B_1}\setminus\{0,a\}$,
the bubble tends smoothly to zero and $Y_\mu\to V$ in $L^{p+1}$.
The local convergence argument in Step 4 of the proof of
Lemma~\ref{lem:regularity-end} gives smooth convergence on these
compact subsets, including across the boundary identification.
On a common punctured ball around $0$, the final phase satisfies
$\widehat T_\mu\to T$, and the correction $w_\mu+\phi_\mu$ tends to
zero in $H^1_\delta$. Applying Step 6 of the proof of Lemma~\ref{lem:regularity-end}
on this common punctured ball gives convergence in
$\cC^k_{\Theta;0}$ there. Combining it with the compact
convergence away from $0$ and $a$ proves
\eqref{eq:outside-convergence}.

\step{Step 3. The peak.}
Set $\widetilde Y_\mu(y)=\mu^{(n-2)/2}Y_\mu(a+\mu y)$.
On $B_1$ in the $y$ coordinates, $\chi(|a+\mu y|)=0$ for small
$\mu$, and $\mu^{(n-2)/2}V_\mu(a+\mu\,\cdot)\to0$ uniformly.
By scaling and \eqref{eq:correction-bounds},
\mbox{$\|\mu^{(n-2)/2}\phi_\mu(a+\mu\,\cdot)\|_{L^{p+1}(B_1)}\to0$}.
Thus
\[
\widetilde Y_\mu\to U_{\sqrt{s_n},0}
\quad\text{in }L^{p+1}(B_1).
\]
The limiting norm is positive; hence
$\max_{\overline{B_\mu(a)}}Y_\mu\ge c\mu^{-(n-2)/2}$.
Choose a maximum point there to obtain \eqref{eq:peak}.

\step{Step 4. Preservation of nondegeneracy.}
All parameter dependent quantities below are evaluated at
$(\lambda_\mu,b_\mu)$, with parameter derivatives taken before
evaluation. Write $\cK_\mu=\cK_{\mu,\lambda_\mu,b_\mu}$ and
$\cL^{\rm new}:=-L_{\gamma_\mu}-n(n+2)Y_\mu^{p-1}$. Set
\[
\mathcal D=\{\dot\tau\psi_{\widehat T_\mu}+w:
                     (\dot\tau,w)\in\R\oplus H^1_\delta\},\qquad
\mathcal D_0=\{\dot\tau\psi_{\widehat T_\mu}+w\in\mathcal D:
                          w\perp_*\cK_\mu\}.
\]
Give $\mathcal D$ the norm $|\dot\tau|+\|w\|_{H^1_\delta}$.
The inverse constructed in the proof of
Proposition~\ref{prop:projected-solution}, with the moment data set to
zero, gives
\begin{equation}\label{eq:new-compressed-inverse}
\Pi^{\perp,*}\cL^{\rm new}:\mathcal D_0\to
                                  \Pi^{\perp,*}H^{-1}_\delta
\quad\hbox{is an isomorphism}.
\end{equation}
This projected operator uses 
$\psi_{\widehat T_\mu}$ at the constructed solution.

Define the tangents to the solved family by
\[
\begin{aligned}
Y_0&=J_0+\lambda_\mu(\partial_\lambda\tau_\mu)
                 \psi_{\widehat T_\mu}+\lambda_\mu\partial_\lambda\phi_\mu,\\
Y_i&=J_i+(\partial_{b_i}\tau_\mu)\psi_{\widehat T_\mu}
                 +\partial_{b_i}\phi_\mu,\qquad 1\le i\le n.
\end{aligned}
\]
All these functions belong to $\mathcal D$.
Differentiate \eqref{eq:full-projected} with respect to $\lambda$ and
$b_i$, and then evaluate at $(\lambda_\mu,b_\mu)$.
Since $c_\mu(\lambda_\mu,b_\mu)=0$, we obtain
\begin{equation}\label{eq:parameter-linearization}
\begin{aligned}
\cL^{\rm new}Y_0
 &=\sum_{\alpha=0}^n\lambda_\mu(\partial_\lambda c_\mu)_\alpha
                                  (-\Delta)J_\alpha,\\
\cL^{\rm new}Y_i
 &=\sum_{\alpha=0}^n(\partial_{b_i}c_\mu)_\alpha(-\Delta)J_\alpha,
                                  \qquad 1\le i\le n.
\end{aligned}
\end{equation}
For $0\le\alpha\le n$ and $1\le i\le n$, differentiating the
moment conditions gives
\begin{equation}\label{eq:parameter-moments}
\begin{aligned}
\ip{(-\Delta)J_\alpha}{Y_0}
 &=(J_\alpha,J_0)_*-\lambda_\mu(\partial_\lambda J_\alpha,\phi_\mu)_*
   =(M_\mu)_{\alpha0},\\
\ip{(-\Delta)J_\alpha}{Y_i}
 &=(J_\alpha,J_i)_*-(\partial_{b_i}J_\alpha,\phi_\mu)_*
   =(M_\mu)_{\alpha i}.
\end{aligned}
\end{equation}
Here $\psi_{\widehat T_\mu}$ vanishes on a neighborhood of the support
of each $J_\alpha$. Since $M_\mu$ is invertible,
\eqref{eq:parameter-moments} establishes the direct sum
\[
\mathcal D=\mathcal D_0\oplus\operatorname{span}\{Y_0,\ldots,Y_n\}.
\]
The space $H^{-1}_\delta$ is the direct sum of
$\Pi^{\perp,*}H^{-1}_\delta$ and
$\operatorname{span}\{(-\Delta)J_0,\ldots,(-\Delta)J_n\}$.
At the critical point, differentiating \eqref{eq:reduced-gradient}
and using $c_\mu=0$ yields
\begin{equation}\label{eq:Hessian-multiplier}
D^2_{\lambda,b}\widetilde\cJ_\mu
=\operatorname{diag}(\lambda_\mu^{-1},1,\ldots,1)
 M_\mu^T D_{\lambda,b}c_\mu.
\end{equation}
Here $D_{\lambda,b}c_\mu$ is the Jacobian with columns
$\partial_\lambda c_\mu,\partial_{b_1}c_\mu,\ldots,\partial_{b_n}c_\mu$.
By Proposition~\ref{prop:reduced-minimum}, the ordinary Hessian is
positive definite. Since $M_\mu$ and the diagonal matrix are invertible,
$D_{\lambda,b}c_\mu$ is invertible.
Choose the ordered bases $Y_0,\ldots,Y_n$ in the finite dimensional
domain and $(-\Delta)J_0,\ldots,(-\Delta)J_n$ in the finite dimensional
source space. With respect to these bases and the preceding direct sums,
\eqref{eq:parameter-linearization} gives the block matrix
\[
\begin{pmatrix}
\left.\Pi^{\perp,*}\cL^{\rm new}\right|_{\mathcal D_0}&0\\
*&D_{\lambda,b}c_\mu\,\operatorname{diag}(\lambda_\mu,1,\ldots,1)
\end{pmatrix}.
\]
By \eqref{eq:new-compressed-inverse} and \eqref{eq:Hessian-multiplier},
both diagonal blocks are invertible. The bounded lower triangular
operator $\cL^{\rm new}:\mathcal D\to H^{-1}_\delta$ is therefore an
isomorphism, with a finite inverse norm for each sufficiently small
fixed $\mu$. This proves \eqref{eq:insertion-inverse}.
\end{proof}

\appendix
\section{Auxiliary estimates and local singular solutions}\label{app:supplementary}
\subsection{Nondegeneracy of the bubble energy minimum}\label{app:coefficient-tables}
\begin{proof}[Proof of Proposition~\ref{prop:local-well}]
We compute the scale and center derivatives of $F$ in
\eqref{eq:F-definition}. Its defining integrals and their first two
parameter derivatives converge locally uniformly for $n\ge30$.
Indeed, writing $r=|y|$, on each compact parameter set we have
\[
|\cH|=O(r^6),\quad |D\cH|=O(r^5),\quad
U_{\lambda,b}=O(r^{2-n}),\quad DU_{\lambda,b}=O(r^{1-n}),
\]
also after two parameter derivatives of the bubble. Each integrand
is therefore $O(r^{14-2n})$, and its radial majorant is $Cr^{13-n}$.
The center calculation uses the angular identities of
Brendle--Marques \cite[Propositions~19--20]{BrendleMarques2009}.
\step{Step 1. The scale derivative.}
To compute the scale dependence, define the coefficients $A_k(n)$ by
\begin{equation*}
\sum_{k=0}^4A_k(n)s^k
:=(n+2)f(s)^2+4sf(s)f'(s)+2s^2f'(s)^2.
\end{equation*}
Set $r_0(n)=1$ and, for $1\le k\le4$,
\begin{equation}
r_k(n):=\prod_{\ell=0}^{k-1}
\frac{n+2+2\ell}{n-8-2\ell},
\label{eq:r-k}
\end{equation}
Define $I_n(s):=\sum_{k=0}^4A_k(n)r_k(n)s^{k+2}$.
Applying the radial integral formula of
Brendle--Marques~\cite[Propositions~17--18]{BrendleMarques2009}
to $f(s)=80-50s+7s^2$, we obtain
\begin{equation}
F(\lambda,0)=-C_{n,W} I_n(\lambda^2).
\label{eq:F-scale}
\end{equation}
Here $C_{n,W}>0$ depends only on $n$ and the fixed Weyl tensor $W$.

Since $U_{\lambda,0}$ is radial, the first term in
\eqref{eq:F-definition} vanishes by \eqref{eq:radial-gauge}.
Each $H_{ik}$ is homogeneous and harmonic of degree two.
Differentiating $\cH(y)=f(|y|^2)H(y)$ and integrating over the unit
sphere, we obtain
\begin{equation*}
\int_{S^{n-1}}|D\cH(r\theta)|^2\dd\theta
=2r^2\left(\int_{S^{n-1}}|H(\theta)|^2\dd\theta\right)
\sum_{k=0}^4A_k(n)r^{2k}.
\end{equation*}
For $0\le k\le4$, the substitution $r=\lambda t$ and the beta integral
formula yield
\begin{equation*}
\lambda^{n-2}\int_0^\infty
\frac{r^{n+1+2k}}{(\lambda^2+r^2)^{n-2}}\dd r
=\frac{\lambda^{2k+4}}2
B\!\left(\frac{n+2}{2}+k,\frac{n-6}{2}-k\right).
\end{equation*}
For $k\ge1$, $\alpha>0$, and $\beta>k$, the recurrence is
\begin{equation}
\frac{B(\alpha+k,\beta-k)}{B(\alpha,\beta)}
=\prod_{\ell=0}^{k-1}\frac{\alpha+\ell}{\beta-1-\ell}.
\label{eq:beta-recurrence}
\end{equation}
With $\alpha=(n+2)/2$ and $\beta=(n-6)/2$, the beta ratio equals $r_k(n)$
in \eqref{eq:r-k}. We obtain \eqref{eq:F-scale} with
\begin{equation*}
C_{n,W}=\frac{n-2}{16(n-1)}
\left(\int_{S^{n-1}}|H(\theta)|^2\dd\theta\right)
B\!\left(\frac{n+2}{2},\frac{n-6}{2}\right)>0.
\end{equation*}
Every denominator in \eqref{eq:r-k} is positive for $n\ge30$.

We seek a nondegenerate minimum of $\lambda\mapsto F(\lambda,0)$.
By \eqref{eq:F-scale}, it suffices to prove $I_n'(9/10)>0$,
$I_n'(11/10)<0$, and $I_n''<0$ on $[9/10,11/10]$:
these conditions locate a unique nondegenerate maximum of $I_n$
in this interval.
To prove the sign of $I_n''$ throughout the interval, we expand
$-I_n''$ in the Bernstein basis, whose functions are nonnegative
and sum to one; hence positive Bernstein coefficients imply $I_n''<0$.
We clear the positive denominators and expand both the endpoint
expressions and the Bernstein coefficients in powers of $n-30$
to verify the required signs for every $n\ge30$.

Set \(D_n:=(n-8)(n-10)(n-12)(n-14)>0\). Reduction to the common denominator $D_n$ yields
\begin{equation}
\begin{aligned}
50000D_nI_n'(9/10)
={}&24500345757696+4693179601152(n-30)\\
&+602130393360(n-30)^2+31533264180(n-30)^3\\
&+1338334164(n-30)^4+25658703(n-30)^5>0.
\end{aligned}
\label{eq:I-left-certificate}
\end{equation}
At the other endpoint,
\begin{equation}
\begin{aligned}
-50000D_nI_n'(11/10)
={}&9936905527296+12448646276352(n-30)\\
&+1144904769360(n-30)^2+82313202180(n-30)^3\\
&+2365410564(n-30)^4+16899003(n-30)^5>0.
\end{aligned}
\label{eq:I-right-certificate}
\end{equation}
To prove $I_n''<0$ throughout $[9/10,11/10]$, write
$s=9/10+t/5$ with $0\le t\le1$ and use the Bernstein polynomials
$B_{r,d}(t):=\binom drt^r(1-t)^{d-r}$ for a nonnegative integer
$d$ and $r=0,\ldots,d$. We convert powers of $t$ to this basis
by the identity
\begin{equation*}
t^j=\sum_{r=j}^d\frac{\binom rj}{\binom dj}B_{r,d}(t),
\qquad 0\le j\le d.
\end{equation*}
Applying this identity with $d=4$ to the polynomial
$-1000D_nI_n''(9/10+t/5)$, we obtain
\begin{equation*}
-1000D_nI_n''(s)=\sum_{r=0}^4P_r(n-30)B_{r,4}(t),
\end{equation*}
where \(P_r(n-30)=\sum_{k=0}^5a_{rk}(n-30)^k\) has positive
coefficients as follows:
\begin{equation*}
\begin{gathered}
\begin{array}{c|rrr}
r&a_{r0}&a_{r1}&a_{r2}\\ \hline
0&6407278534656&2133921686272&230888118960\\
1&5113817063424&1912452268288&203914653840\\
2&3658424426496&1691729783552&174393005360\\
3&2021939011584&1498016766208&145427943440\\
4&17166606336&1334792434432&118893859760
\end{array}\\
\begin{array}{c|rrr}
r&a_{r3}&a_{r4}&a_{r5}\\ \hline
0&14182097980&419289804&4455133\\
1&12987814420&402050916&4432207\\
2&11526073180&376392364&4324053\\
3&9929519220&344804356&4148087\\
4&8297728380&309334924&3919373
\end{array}
\end{gathered}
\end{equation*}
Thus $P_r(n-30)>0$ for every $n\ge30$. Since $D_n>0$,
$B_{r,4}(t)\ge0$, and $\sum_{r=0}^4B_{r,4}(t)=1$ on $[0,1]$,
we conclude that
\begin{equation*}
I_n''(s)<0,\qquad 9/10\le s\le11/10.
\end{equation*}
By the endpoint signs in \eqref{eq:I-left-certificate} and
\eqref{eq:I-right-certificate}, the strictly decreasing function $I_n'$
has a unique zero:
\begin{equation}
s_n\in(9/10,11/10),\qquad I_n'(s_n)=0,
\qquad\text{and}\qquad I_n''(s_n)<0.
\label{eq:s-star}
\end{equation}
\step{Step 2. The center derivatives.}
We prove $D_b^2F(\sqrt{s_n},0)>0$. The polynomial $f$ satisfies
\begin{equation*}
2f(s)f'(s)+sf'(s)^2
=-8000+9740s-3500s^2+392s^3.
\end{equation*}
Define
\begin{align}
t_0(n)&:=1,\qquad
t_k(n):=\prod_{\ell=0}^{k-1}
\frac{n+6+2\ell}{n-8-2\ell},\qquad 1\le k\le3,\notag\\
\shortintertext{and}
K_n(s)&:=-8000+9740t_1(n)s-3500t_2(n)s^2+392t_3(n)s^3.
\label{eq:K-n}
\end{align}
Applying the beta recurrence \eqref{eq:beta-recurrence} with
$\alpha=(n+6)/2$ and $\beta=(n-6)/2$ produces the factors $t_k(n)$ in
\eqref{eq:K-n}.  Let \(E_n:=(n-8)(n-10)(n-12)>0\).  For $s=9/10+t/5$, conversion to the Bernstein basis yields
\begin{equation*}
375E_nK_n(s)=-\sum_{r=0}^3Q_r(n-30)B_{r,3}(t),
\end{equation*}
where \(Q_r(n-30)=\sum_{k=0}^3b_{rk}(n-30)^k\) has positive
coefficients as follows:
\begin{equation*}
\begin{array}{c|rrrr}
r&b_{r0}&b_{r1}&b_{r2}&b_{r3}\\ \hline
0&1471670640&314268036&22334418&668712\\
1& 891088560&215683044&16090622&558898\\
2& 451848240&141444276&10853538&461292\\
3&  89598960& 86462004& 6489102&374718
\end{array}
\end{equation*}
Therefore, for
\(9/10\le s\le11/10\) and \(n\ge30\),
\begin{equation}
K_n(s)<0.
\label{eq:K-negative}
\end{equation}
Define
\begin{align*}
I_1(\lambda)
&:=\int_0^\infty \lambda^{n-2}(\lambda^2+r^2)^{-n}r^{n+5}
\{2f(r^2)f'(r^2)+r^2f'(r^2)^2\}\dd r,
\\
I_2(\lambda)
&:=\int_0^\infty \lambda^{n-2}(\lambda^2+r^2)^{1-n}r^{n+5}
f'(r^2)^2\dd r>0.
\end{align*}
By the substitution $r=\lambda t$ and the beta recurrence
\eqref{eq:beta-recurrence},
\begin{equation*}
I_1(\lambda)=\lambda^4K_n(\lambda^2)
\int_0^\infty(1+r^2)^{-n}r^{n+5}\dd r.
\end{equation*}
For $1\le\alpha,\beta\le n$, differentiating \eqref{eq:F-definition}
with respect to the center parameters $b_\alpha$ and $b_\beta$, we obtain
\[
\begin{aligned}
\partial_{b_\alpha}\partial_{b_\beta}F(\lambda,0)
={}&(n-2)^2\lambda^{n-2}
\int_{\R^n}(\lambda^2+|y|^2)^{-n}
\sum_{l=1}^n\cH_{\alpha l}\cH_{\beta l}\dd y\\
&-\frac{(n-2)^2\lambda^{n-2}}4
\int_{\R^n}(\lambda^2+|y|^2)^{-n}
\sum_{i,k,l=1}^n(\partial_{y_l}\cH_{ik})^2y_\alpha y_\beta\dd y\\
&+\frac{(n-2)^2\lambda^{n-2}\delta_{\alpha\beta}}{8(n-1)}
\int_{\R^n}(\lambda^2+|y|^2)^{1-n}
\sum_{i,k,l=1}^n(\partial_{y_l}\cH_{ik})^2\dd y.
\end{aligned}
\]

The integral expression for $D_b^2F(\lambda,0)$ agrees with the formula in
Brendle--Marques \cite[Proposition~19]{BrendleMarques2009}.
Evaluating these angular integrals as in
Brendle--Marques \cite[Proposition~20]{BrendleMarques2009}, we obtain
\begin{equation*}
\partial_{b_\alpha}\partial_{b_\beta}F(\lambda,0)
=\frac{(n-2)^2|S^{n-1}|}{4n(n+2)}
\left[\frac{I_2(\lambda)}{n-1}S_W\delta_{\alpha\beta}
-\frac{2I_1(\lambda)}{n+4}
\bigl(4M_{\alpha\beta}+S_W\delta_{\alpha\beta}\bigr)\right],
\end{equation*}
where
\begin{equation*}
M_{\alpha\beta}
:=\sum_{i,k,l=1}^n(W_{i\alpha kl}+W_{ilk\alpha})
(W_{i\beta kl}+W_{il k\beta}).
\end{equation*}
The matrix \((M_{\alpha\beta})_{1\le\alpha,\beta\le n}\) is positive
semidefinite because it is a Gram matrix, and by the first Bianchi identity,
\begin{equation*}
S_W:=\sum_{i,j,k,l=1}^n(W_{ijkl}+W_{ilkj})^2
=3|W|^2>0.
\end{equation*}
By \eqref{eq:K-negative}, $I_1(\sqrt{s_n})<0$, whereas
$I_2(\sqrt{s_n})>0$ and $S_W>0$. In the formula for
$\partial_{b_\alpha}\partial_{b_\beta}F(\lambda,0)$, the multiple of
$S_W I$ is strictly positive and the remaining
contribution of the Gram matrix is nonnegative. Hence
\begin{equation*}
D_b^2F(\sqrt{s_n},0)>0.
\end{equation*}

\step{Step 3. The full Hessian.}
The bubble satisfies
$U_{\lambda,-b}(-y)=U_{\lambda,b}(y)$.  Since $\cH(-y)=\cH(y)$, a change of variables in
\eqref{eq:F-definition} proves $F(\lambda,-b)=F(\lambda,b)$, and
\begin{equation}
\partial_{b_\alpha}F(\lambda,0)=0
\qquad\text{and}\qquad
\partial_\lambda\partial_{b_\alpha}F(\lambda,0)=0.
\label{eq:mixed-zero}
\end{equation}
Finally, from \eqref{eq:F-scale} and \eqref{eq:s-star},
\begin{equation*}
\partial_{\lambda\lambda}F(\sqrt{s_n},0)
=-4C_{n,W}s_nI_n''(s_n)>0.
\end{equation*}
By \eqref{eq:F-scale}, \eqref{eq:s-star}, and \eqref{eq:mixed-zero},
\[
\nabla F(\sqrt{s_n},0)=0,\qquad
D^2F(\sqrt{s_n},0)=
\begin{pmatrix}
\partial_{\lambda\lambda}F(\sqrt{s_n},0)&0\\
0&D_b^2F(\sqrt{s_n},0)
\end{pmatrix}>0.
\]
This proves Proposition~\ref{prop:local-well}.
\end{proof}

\subsection{Local solutions with cylindrical asymptotics}\label{app:local-cylindrical}
We construct local singular solutions on arbitrary smooth manifolds
in every dimension $n\ge3$, using logarithmic radius near the singularity.
The constant $A_n$ fixed in Section~\ref{sec:inductive} satisfies
$n(n-2)A_n^{p-1}=(n-2)^2/4$.
For the flat equation, prescribed Fowler asymptotics were considered by
Han--Li~\cite{HanLi2021}. The proof below treats the constant profile on
an arbitrary smooth background.

\begin{theorem}\label{thm:local-cylindrical}
Let $(M^n,g)$ be a smooth Riemannian manifold, $n\ge3$, and let $x_0$ be
an interior point. There are constants $\rho_0,C>0$, depending only on
$n$ and $g$ in a fixed normal neighborhood of $x_0$. For every $0<\rho\le\rho_0$, there is a positive solution
$u_\rho\in C^\infty(\overline{B_\rho^g(x_0)}\setminus\{x_0\})$
of
\[
-L_g u_\rho=n(n-2)u_\rho^p
\quad\text{in }B_\rho^g(x_0)\setminus\{x_0\},
\]
 satisfying
\begin{equation}\label{eq:cyl-geometric-bound}
\sum_{j=0}^2 r^j
\left|\nabla_g^j\bigl(r^{(n-2)/2} u_\rho-A_n\bigr)\right|_g
\le Cr^2,\qquad r=d_g(x,x_0),\quad 0<r\le\rho.
\end{equation}
In particular,
\[
u_\rho(x)=A_n r^{-(n-2)/2}(1+O(r^2))
\]
uniformly in $\rho$, and $u_\rho^{4/(n-2)}g$ is complete toward $x_0$.
\end{theorem}

We first construct an inverse on
$\mathcal C=[0,\infty)\times S^{n-1}$ with product metric $\dd t^2+\sigma$.
Here $\sigma$ is the unit round metric. Fix $0<\alpha<1$ and define
\[
\|w\|_{C^{k,\alpha}_\nu(\mathcal C)}
:=\sup_{s\ge0}e^{\nu s}\|w\|_{C^{k,\alpha}(Q_s)},
\qquad Q_s=[s,s+1]\times S^{n-1},
\]
where the strip norms use the product metric. Interior and boundary
estimates use the enlarged strips
$Q_s^*=[\max\{0,s-1\},s+2]\times S^{n-1}$.

Let $\Pi_\ell$ be the round $L^2$ projection onto spherical harmonics of
degree $\ell$, and let $\Pi_{\ge2}=I-\Pi_0-\Pi_1$.
Set $P:=\partial_t^2+\Delta_\sigma+n-2$.

\begin{lemma}\label{lem:cyl-inverse}
For $1<\nu<\sqrt{n+2}$, $P$ is an isomorphism from
$\{w\in C^{2,\alpha}_\nu(\mathcal C):\Pi_{\ge2}w(0)=0\}$ to
$C^{0,\alpha}_\nu(\mathcal C)$. Its inverse $G_\nu$ satisfies, with $C=C(n,\alpha,\nu)>0$,
\begin{equation}\label{eq:cyl-inverse}
PG_\nu f=f,\qquad
\Pi_{\ge2}(G_\nu f)(0)=0,\qquad
\|G_\nu f\|_{C^{2,\alpha}_\nu}\le C\|f\|_{C^{0,\alpha}_\nu}.
\end{equation}
\end{lemma}
\begin{proof}
\step{Step 1. Degrees zero and one.}
On degree $\ell$, $P$ acts as
$\partial_t^2+n-2-\ell(\ell+n-2)$. In degrees zero and one, respectively,
its homogeneous solutions are trigonometric functions of frequency
$\sqrt{n-2}$ and the functions $e^{\pm t}$. The unique solutions with
the required decay are
\[
\begin{aligned}
w_0(t)&=\frac1{\sqrt{n-2}}
 \int_t^\infty\sin\bigl(\sqrt{n-2}(s-t)\bigr)\Pi_0f(s)\dd s,\\
w_1(t)&=\int_t^\infty\sinh(s-t)\Pi_1f(s)\dd s.
\end{aligned}
\]
Differentiating under the integral twice proves their equations.
The integrals converge because $\nu>1$; in degree one the estimates use
\[
\int_0^\infty e^{-\nu s}\sinh s\dd s=\frac1{\nu^2-1}
\qquad\text{and}\qquad
\int_0^\infty e^{-\nu s}\cosh s\dd s=\frac\nu{\nu^2-1}.
\]
The two integral bounds control $w_0,w_1$ and their first derivatives.
The equations control their second derivatives and H\"older seminorms. Thus
$\|w_0\|_{C^{2,\alpha}_\nu}+\|w_1\|_{C^{2,\alpha}_\nu}
\le C\|f\|_{C^{0,\alpha}_\nu}$.

\step{Step 2. Higher harmonics on finite cylinders.}
For the high harmonics, first solve $Pw_L=\Pi_{\ge2}f$ on
$[0,L]\times S^{n-1}$ with zero boundary values at both ends. On the
subspace of harmonics of degrees at least two in
$H^1_0((0,L)\times S^{n-1})$, with the usual zero traces at both ends,
the bilinear form of $-P$ is
\[
B_L(w,w)=\int_0^L\int_{S^{n-1}}
\bigl(|w_t|^2+|\nabla_\sigma w|_\sigma^2-(n-2)w^2\bigr)\dd\theta\dd t.
\]
On this harmonic subspace,
$\|\nabla_\sigma w\|_2^2\ge2n\|w\|_2^2$, and hence
\[
B_L(w,w)\ge\|w_t\|_2^2+
\frac{n+2}{2n}\|\nabla_\sigma w\|_2^2.
\]
These norms are taken on $[0,L]\times S^{n-1}$ with product volume;
the same Poincar\'e inequality also controls the $L^2$ norm.
The Lax--Milgram theorem gives a unique weak solution in this subspace.
Since the harmonic projections commute with $P$, the equation holds
against arbitrary test functions. Interior and
Dirichlet regularity in Gilbarg--Trudinger
\cite[Chapters~6, 8, and 9]{GilbargTrudinger2001} imply
$w_L\in C^{2,\alpha}$.

\step{Step 3. Uniform weighted estimates.}
To obtain bounds independent of $L$, set
$N_\varepsilon(t)=(\|w_L(t)\|_2^2+\varepsilon^2)^{1/2}$.
Differentiating and using the equation and Cauchy--Schwarz, we obtain
\[
\begin{split}
N_\varepsilon''
&=\frac{\|(w_L)_t\|_2^2+\langle w_L,(w_L)_{tt}\rangle}
           {N_\varepsilon}
-\frac{\langle w_L,(w_L)_t\rangle^2}{N_\varepsilon^3}\\
&\ge(n+2)N_\varepsilon-(n+2)\varepsilon
                         -\|\Pi_{\ge2}f(t)\|_2.
\end{split}
\]
If $M_f=\sup_{t\ge0}e^{\nu t}\|\Pi_{\ge2}f(t)\|_2$, comparison with
\[
\frac{M_f}{n+2-\nu^2}e^{-\nu t}+\varepsilon
\]
for $-\partial_t^2+n+2$, followed by $\varepsilon\downarrow0$, yields
\begin{equation*}
\|w_L(t)\|_2\le
\frac{M_f}{n+2-\nu^2}e^{-\nu t}.
\end{equation*}
The comparison holds at $t=0,L$ because $w_L$ vanishes there.

Applying local elliptic estimates on nested strips, we obtain
\begin{equation*}
\|w_L\|_{C^{2,\alpha}(Q_s)}
\le C\bigl(\|w_L\|_{L^2(Q_s^*)}
                  +\|\Pi_{\ge2}f\|_{C^{0,\alpha}(Q_s^*)}\bigr)
\le Ce^{-\nu s}\|f\|_{C^{0,\alpha}_\nu}
\end{equation*}
when $L>s+3$. For the first inequality, use interior or Dirichlet
$W^{2,m}$ estimates, with zero boundary data in the latter case.
Starting from $m=2$, apply these estimates and Sobolev embeddings
on nested strips. If $m=n/2$, use $W^{2,n/2}\hookrightarrow L^n$
on a smaller strip and then apply the $W^{2,n}$ estimate.
Thus the iteration reaches an exponent $m>n/2$ and gives a supremum
bound. Schauder estimates then prove the $C^{2,\alpha}(Q_s)$ estimate.
The coefficients and strip widths are fixed; hence the constants
are independent of $s,L$.

\step{Step 4. The limit and uniqueness.}
By the uniform estimates, a sequence with $L\to\infty$ converges on
each finite cylinder to a solution in the harmonics of degree at least
two. This solution satisfies \eqref{eq:cyl-inverse}. The $C^{2,\alpha}$ bounds pass to the limit,
including the H\"older seminorms. Add $w_0+w_1$ to obtain $G_\nu f$.
For a homogeneous solution in $C^{2,\alpha}_\nu$, the condition
$\nu>1$ gives $w_0=w_1=0$. For $\ell\ge2$, decay and the boundary condition imply
\[
w_\ell(t)=e^{-\sqrt{\ell(\ell+n-2)-(n-2)}t}w_\ell(0)=0.
\]
This proves uniqueness.
\end{proof}

\begin{proof}[Proof of Theorem~\ref{thm:local-cylindrical}]
\step{Step 1. Reduction to the half cylinder.}
Work in a fixed relatively compact normal neighborhood of $x_0$.
In polar coordinates $x=\exp_{x_0}(r\theta)$,
\[
g=\dd r^2+r^2h(r,\theta),\qquad h(0,\theta)=\sigma.
\]
In normal Cartesian coordinates, Taylor's formula reads
\[
g_{ij}(r\theta)=\delta_{ij}+r^2\sum_{k,l=1}^n\theta_k\theta_l
 \int_0^1(1-\tau)\partial_k\partial_lg_{ij}(\tau r\theta)\dd\tau.
\]
By Gauss' lemma its tangential restriction is $h$. In particular,
$r^{-2}(h-\sigma)$ extends smoothly to $r=0$ and
\begin{equation}\label{eq:cyl-h-estimates}
\|(r\partial_r)^j(h(r)-\sigma)\|_{C^k(S^{n-1})}
\le C_{j,k}r^2.
\end{equation}
Write
\[
\partial_r\log\sqrt{\det h(r,\theta)/\det\sigma(\theta)}=O(r).
\]
This estimate also holds after every fixed number of cylindrical derivatives. Thus
\[
\Delta_g=\partial_r^2+\left(\frac{n-1}{r}+\partial_r\log\sqrt{\det h/\det\sigma}\right)\partial_r
                                      +r^{-2}\Delta_{h(r)}.
\]

For this local construction, put
$V(t,\theta)=r^{(n-2)/2}u(\exp_{x_0}(r\theta))$, where $r=\rho e^{-t}$.
Direct differentiation yields
\[
u_r=r^{-n/2}(-\frac{n-2}{2}V-V_t),\qquad
u_{rr}=r^{-(n+2)/2}\bigl(\frac{n(n-2)}4V+(n-1)V_t+V_{tt}\bigr).
\]
Thus the Yamabe equation is
\begin{equation}\label{eq:cyl-cylinder}
V_{tt}+\Delta_\sigma V-\frac{(n-2)^2}{4}V+n(n-2)V^p+\mathcal E_\rho V=0,
\end{equation}
where
\begin{equation}\label{eq:cyl-E}
\mathcal E_\rho V=(\Delta_{h(r)}-\Delta_\sigma)V
-r\partial_r\log\sqrt{\det h/\det\sigma}\left(V_t+\frac{n-2}{2}V\right)-c(n)r^2\Scal_g V.
\end{equation}
By \eqref{eq:cyl-h-estimates} and the smoothness of $\Scal_g$, every
coefficient of $\mathcal E_\rho$ has $C^{k,\alpha}(Q_s)$ norm at most
$C_k\rho^2e^{-2s}$. The logarithmic volume derivative in
\eqref{eq:cyl-E} has the same bound after multiplication by $r$.
One further derivative controls each H\"older seminorm.

\step{Step 2. A contraction at the constant profile.}
The constant $A_n$ solves the flat equation in \eqref{eq:cyl-cylinder},
and its linearization is $P$. For $V=A_n+w$ the equation becomes
\begin{equation}\label{eq:cyl-w-equation}
\begin{gathered}
Pw+\mathcal E_\rho A_n+\mathcal E_\rho w+\mathcal Q(w)=0,\\
\mathcal Q(w)
:=n(n-2)\bigl\{(A_n+w)^p
 -A_n^p
 -pA_n^{p-1}w\bigr\}.
\end{gathered}
\end{equation}

By the coefficient estimates,
\begin{equation}\label{eq:cyl-metric-estimates}
\|\mathcal E_\rho A_n\|_{C^{0,\alpha}_2}\le C\rho^2
\qquad\text{and}\qquad
\|\mathcal E_\rho w\|_{C^{0,\alpha}_2}\le C\rho^2\|w\|_{C^{2,\alpha}_2}.
\end{equation}
For $|w|\le A_n/2$, Taylor's formula is
\begin{equation*}
\mathcal Q(w)=\frac{4n(n+2)}{n-2}w^2
\int_0^1(1-\tau)(A_n+\tau w)^{p-2}\dd\tau.
\end{equation*}

The integrand and all its $w$ derivatives are bounded on this interval.
Applying H\"older product estimates, we obtain
\begin{align*}
\|\mathcal Q(w)\|_{C^{0,\alpha}_2}&\le C\|w\|_{C^{2,\alpha}_2}^2,\\
\|\mathcal Q(w_1)-\mathcal Q(w_2)\|_{C^{0,\alpha}_2}
&\le C(\|w_1\|_{C^{2,\alpha}_2}+\|w_2\|_{C^{2,\alpha}_2})
                                      \|w_1-w_2\|_{C^{2,\alpha}_2}.
\end{align*}
These estimates hold in all dimensions $n\ge3$, since the arguments
of the power stay uniformly positive.

Since $1<2<\sqrt{n+2}$, Lemma~\ref{lem:cyl-inverse} applies with
$\nu=2$. Combining its inverse with \eqref{eq:cyl-metric-estimates}, the map
\[
\mathcal T_\rho(w)=-G_2(\mathcal E_\rho A_n+
                                  \mathcal E_\rho w+\mathcal Q(w))
\]
satisfies, with constants independent of $\rho$,
\begin{align*}
\|\mathcal T_\rho(w)\|_{C^{2,\alpha}_2}
&\le C_0\rho^2+C_1\rho^2\|w\|_{C^{2,\alpha}_2}
                           +C_2\|w\|_{C^{2,\alpha}_2}^2,\\
\|\mathcal T_\rho(w_1)-\mathcal T_\rho(w_2)\|_{C^{2,\alpha}_2}
&\le\bigl(C_1\rho^2+C_2\|w_1\|_{C^{2,\alpha}_2}
                     +C_2\|w_2\|_{C^{2,\alpha}_2}\bigr)
                                      \|w_1-w_2\|_{C^{2,\alpha}_2}.
\end{align*}
Choose $C_*\ge\max\{1,2C_0\}$ and then $\rho_0>0$ so small that
\[
C_*\rho_0^2\le\min\{1,A_n/2\},\qquad
(C_1+2C_2C_*)\rho_0^2\le\frac12.
\]
The map is a contraction of the closed ball
\[
\{w:\Pi_{\ge2}w(0)=0,\ \|w\|_{C^{2,\alpha}_2}\le C_*\rho^2\}
\]
into itself. It has a unique fixed point $w_\rho$ in that ball, and
$A_n+w_\rho\ge A_n/2$. Consequently
\[
u_\rho(\exp_{x_0}(r\theta))
=r^{-(n-2)/2}\bigl(A_n+w_\rho(\log(\rho/r),\theta)\bigr)
\]
is a positive solution, with
\begin{equation}\label{eq:cyl-base-decay}
\|w_\rho\|_{C^{2,\alpha}(Q_s)}\le C_*\rho^2e^{-2s}.
\end{equation}
Fixing $\alpha=1/2$ in the contraction shows that $\rho_0$ depends only
on $n$ and the fixed smooth geometry.

\step{Step 3. All derivatives and the geometric estimate.}
The equation
\[
(P+\mathcal E_\rho)w_\rho
=-\mathcal E_\rho A_n-\mathcal Q(w_\rho)
\]
is uniformly elliptic. Since $\Pi_{\ge2}w_\rho(0)=0$, the boundary trace lies in
$\operatorname{span}\{1,\theta_1,\ldots,\theta_n\}$. The estimate
\eqref{eq:cyl-base-decay} therefore bounds every tangential derivative
of the trace by $C_k\rho^2$. Interior and Dirichlet estimates first establish
smoothness up to $t=0$. If the derivatives through order $k$ obey
$C_k\rho^2e^{-2s}$, the source terms obey
\[
\|\mathcal E_\rho A_n\|_{C^{k,\alpha}(Q_s)}
 \le C_k\rho^2e^{-2s},\qquad
\|\mathcal Q(w_\rho)\|_{C^{k,\alpha}(Q_s)}
 \le C_k\rho^4e^{-4s}.
\]
Applying Schauder estimates on enlarged strips, we obtain
\begin{equation*}
\begin{split}
\|w_\rho\|_{C^{k+2,\alpha}(Q_s)}
&\le C_k\bigl(\|w_\rho\|_{C^0(Q_s^*)}
 +\|\mathcal E_\rho A_n\|_{C^{k,\alpha}(Q_s^*)}
 +\|\mathcal Q(w_\rho)\|_{C^{k,\alpha}(Q_s^*)}\bigr)
 +C_k\rho^2\mathbf1_{\{s\le1\}}.
\end{split}
\end{equation*}
For $s\le1$ the last term is bounded by $C_ke^2\rho^2e^{-2s}$.
Starting from \eqref{eq:cyl-base-decay}, induction therefore proves
\begin{equation}\label{eq:cyl-all-derivatives}
\|w_\rho\|_{C^{k,\alpha}(Q_s)}\le C_k\rho^2e^{-2s}
\qquad k\ge0.
\end{equation}
The same $\rho_0$ applies to every fixed derivative order.
Bounds for any fixed H\"older exponent follow from one further derivative.

The metric $r^{-2}g=\dd t^2+h(r)$ is uniformly comparable to the
product metric, with uniformly bounded connection difference.
Using the conformal connection formula with $\dd\log r=-\dd t$, we obtain
\[
r|\nabla_gw_\rho|_g+r^2|\nabla_g^2w_\rho|_g
\le C\bigl(|\nabla_{\dd t^2+\sigma}w_\rho|
                  +|\nabla^2_{\dd t^2+\sigma}w_\rho|\bigr)
\le Cr^2.
\]
Together with $|w_\rho|\le Cr^2$, this is
\eqref{eq:cyl-geometric-bound}.
\step{Step 4. Completeness.}
We have $u_\rho\ge(A_n/2)r^{-(n-2)/2}$, and
\[
u_\rho^{4/(n-2)}g
=(A_n+w_\rho)^{4/(n-2)}(\dd t^2+h(r))
\]
is bounded below by a positive multiple of the product metric.
Equivalently, every curve approaching $x_0$ has length at least
$c\int|\dd r|/r=\infty$. This proves completeness toward the puncture
and nonremovability.
\end{proof}

The boundary condition $\Pi_{\ge2}w_\rho(0)=0$ gives
\begin{equation}\label{eq:cyl-trace}
\begin{split}
u_\rho(\exp_{x_0}(\rho\theta))
 &=\rho^{-(n-2)/2}\left(A_n+b_0(\rho)+\sum_{j=1}^n b_j(\rho)\theta_j\right),\\
|b_0(\rho)|+\sum_{j=1}^n|b_j(\rho)|&\le C\rho^2.
\end{split}
\end{equation}
The coefficients $b_0(\rho),\ldots,b_n(\rho)$ are determined by the
inverse $G_2$. For fixed $\rho$, the outer
sphere is at finite distance in the conformal metric.

Applying Theorem~\ref{thm:local-cylindrical} to $g_\infty$ at $o$ gives a
local solution $u_\rho$ that is complete toward $o$ and satisfies
\[
d_{g_\infty}(x,o)^{(n-2)/2} u_\rho(x)\to A_n\quad\text{as }x\to o,
\qquad
\limsup_{x\to o}\frac{u_\infty(x)}{u_\rho(x)}=\infty.
\]
The local solution has the boundary trace in \eqref{eq:cyl-trace}.

\subsection{Second order asymptotics of the local solutions}\label{app:quadratic-curvature}
We compute the curvature term in the local solution and retain the
homogeneous correction imposed by its boundary condition.
We use the curvature convention in which the unit sphere has positive
sectional curvature. Write $\operatorname{Ric}_g^\circ=\operatorname{Ric}_g-\Scal_g g/n$ and define
\begin{equation*}
B(\theta):=A_n\left(\frac16\operatorname{Ric}_g(x_0)(\theta,\theta)
 -\frac{\Scal_g(x_0)}{12(n-1)}\right),
\qquad B_2(\theta):=\frac{A_n}{6}\operatorname{Ric}^{\circ}_g(x_0)(\theta,\theta).
\end{equation*}

\begin{proposition}\label{prop:cyl-curvature}
For the correction $w_\rho$ constructed in the proof of
Theorem~\ref{thm:local-cylindrical},
\begin{equation}\label{eq:cyl-refined}
w_\rho(t,\theta)=\rho^2e^{-2t}B(\theta)
-\rho^2e^{-\sqrt{n+2}\,t}B_2(\theta)+\mathcal R_\rho(t,\theta).
\end{equation}
For every $2<\nu<\min\{3,\sqrt{n+2}\}$ and every fixed integer $k\ge0$,
\begin{equation}\label{eq:cyl-remainder}
\|\mathcal R_\rho\|_{C^{k,\alpha}(Q_s)}
\le C_{k,\nu}\rho^3e^{-\nu s},\qquad
\Pi_{\ge2}\mathcal R_\rho(0)=0.
\end{equation}
The constants are independent of $s\ge0$ and $0<\rho\le\rho_0$.
The coefficients in \eqref{eq:cyl-trace} satisfy
\begin{equation}\label{eq:cyl-trace-refined}
b_0(\rho)=\frac{A_n(n-2)\Scal_g(x_0)}{12n(n-1)}\rho^2+O(\rho^3),
\qquad \sum_{j=1}^n|b_j(\rho)|=O(\rho^3).
\end{equation}
\end{proposition}
\begin{proof}
\step{Step 1. The curvature coefficient.}
The polar volume expansion is
\[
\sqrt{\det h(r,\theta)/\det\sigma}
=1-\frac16\operatorname{Ric}_g(x_0)(\theta,\theta)r^2+O(r^3).
\]
Indeed, for an orthonormal basis $e_j$ of $\theta^\perp$, the normal
Jacobi fields in a parallel frame satisfy
$J_{e_j}(r)=re_j-r^3\operatorname{Rm}(e_j,\theta)\theta/6+O(r^4)$;
taking the square root of the determinant of their normalized Gram
matrix proves the volume formula.
Smoothness preserves the indicated orders under cylindrical derivatives.
It follows that
\[
r\partial_r\log\sqrt{\det h/\det\sigma}
=-\tfrac13\operatorname{Ric}_g(x_0)(\theta,\theta)r^2+O(r^3)
\quad\text{and}\quad r^2\Scal_g=r^2\Scal_g(x_0)+O(r^3).
\]
Substitution in \eqref{eq:cyl-E} yields
\[
\mathcal E_\rho A_n=r^2F_2(\theta)+O(r^3),\qquad
F_2=A_n\left(\frac{n-2}{6}\operatorname{Ric}_g(x_0)(\theta,\theta)
                   -c(n)\Scal_g(x_0)\right).
\]
Write $B=B_0+B_2$, where
$B_0=A_n(n-2)\Scal_g(x_0)/(12n(n-1))$.
The trace free Ricci term has spherical degree two, and $B_2=\Pi_2B$. Directly,
\[
\Pi_0F_2=-(n+2)B_0\qquad\text{and}\qquad \Pi_2F_2=(n-2)B_2.
\]
Thus
\begin{equation*}
(\Delta_\sigma+n+2)B=-F_2,\qquad P(r^2B)=-r^2F_2.
\end{equation*}
\step{Step 2. The remainder estimate.}
Set $\psi=w_\rho-r^2B$. By \eqref{eq:cyl-w-equation},
\[
P\psi=f_*:=-(\mathcal E_\rho A_n-r^2F_2)
                -\mathcal E_\rho w_\rho-\mathcal Q(w_\rho).
\]
By the smooth polar expansion, the first source term is bounded by
$C_k\rho^3e^{-3s}$ in $C^{k,\alpha}(Q_s)$.
By \eqref{eq:cyl-all-derivatives} and the coefficient bounds for
$\mathcal E_\rho$,
\[
\|\mathcal E_\rho w_\rho\|_{C^{k,\alpha}(Q_s)}
 +\|\mathcal Q(w_\rho)\|_{C^{k,\alpha}(Q_s)}
\le C_k\rho^4e^{-4s}.
\]
Taking $\rho_0\le1$ therefore yields
\[
\|f_*\|_{C^{k,\alpha}(Q_s)}\le C_k\rho^3e^{-3s},
\qquad k\ge0,\quad s\ge0.
\]
The homogeneous correction
$h_\rho=-\rho^2e^{-\sqrt{n+2}\,t}B_2$ satisfies
\[
Ph_\rho=0,\qquad \Pi_{\ge2}h_\rho(0)
=-\rho^2B_2=\Pi_{\ge2}\psi(0).
\]
For the stated $\nu$, Lemma~\ref{lem:cyl-inverse} implies
$\|G_\nu f_*\|_{C^{2,\alpha}_\nu}\le C_\nu\rho^3$.
The function $\psi-h_\rho-G_\nu f_*$ lies in the domain with weight $2$ and
has zero image under $P$. Uniqueness at weight two therefore implies
\[
\psi=h_\rho+G_\nu f_*,\qquad \mathcal R_\rho=G_\nu f_*.
\]

The inverse bound proves \eqref{eq:cyl-remainder} through order two.
The boundary trace of $\mathcal R_\rho$ has only degrees zero and one; therefore all derivatives
of that trace are bounded by $C_{k,\nu}\rho^3$.
Applying the strip Schauder estimate to $P\mathcal R_\rho=f_*$, we obtain
\[
\begin{split}
\|\mathcal R_\rho\|_{C^{k+2,\alpha}(Q_s)}
\le C_k\bigl(\|\mathcal R_\rho\|_{C^0(Q_s^*)}
 +\|f_*\|_{C^{k,\alpha}(Q_s^*)}\bigr)
 +C_{k,\nu}\rho^3\mathbf1_{\{s\le1\}}
\le C_{k,\nu}\rho^3e^{-\nu s}.
\end{split}
\]
Here $\nu<3$ controls the source decay and the boundary term is confined
to the first two strips. This proves the estimate in every order,
with the same $\rho_0$.

\step{Step 3. The asymptotic expansion and boundary values.}
Returning to $r=\rho e^{-t}$ yields
\[
w_\rho=r^2B-\rho^{2-\sqrt{n+2}}r^{\sqrt{n+2}}B_2
                   +O(\rho^{3-\nu}r^\nu).
\]
Since $\nu<\sqrt{n+2}$, the homogeneous correction is
$O_\rho(r^\nu)$ for fixed $\rho$.
Thus
\[
r^{(n-2)/2}u_\rho(\exp_{x_0}(r\theta))
=A_n+r^2B(\theta)+O_\rho(r^\nu)\quad\text{as }r\downarrow0,
\]
with the same order after every fixed number of cylindrical derivatives.
At $t=0$, the homogeneous term cancels the degree two component of $\rho^2B$; the remaining
trace is $\rho^2B_0+\mathcal R_\rho(0)$. This proves
\eqref{eq:cyl-trace-refined}.
\end{proof}

At $t=0$, the homogeneous correction in \eqref{eq:cyl-refined}
has order $\rho^2$, while $\mathcal R_\rho(0)=O(\rho^3)$.
For constant sectional curvature $K$, $B_2=0$ and
$B/A_n=(n-2)K/12$. In that case, substituting
$s_K(r)=r-Kr^3/6+O(r^5)$ in the explicit solution
$u=A_n s_K(r)^{-(n-2)/2}$ for $g=\dd r^2+s_K(r)^2\sigma$, we obtain
\[
r^{(n-2)/2}u=A_n\left(1+\frac{n-2}{12}Kr^2+O(r^4)\right).
\]
Its conformal metric is the product cylinder after
$\dd t=\dd r/s_K(r)$, in agreement with \eqref{eq:cyl-refined}.

\subsection{Regularity and weighted convergence}\label{app:regularity}

All spaces in this subsection are the Euclidean weighted spaces
\eqref{eq:weighted-spaces}--\eqref{eq:weighted-holder}.
Let $\gamma$ be smooth on $\overline{B_1}$, Euclidean near $0$
and on the fixed boundary collar, with $\dd V_\gamma=\dd x$.
Assume that $Q_\gamma$ is coercive on
$H^1_{\mathrm{ev}}(B_1)$.
Weak solutions are understood using the form $Q_\gamma$
in \eqref{eq:weak-robin-form}, with smooth test functions
that vanish near $0$ and have antipodally even boundary traces.
For the first assertion below, let $A$ be smooth on
$\overline{B_1}\setminus\{0\}$, satisfy
\eqref{eq:boundary-matching}, and obey
$c\Theta\le A\le C\Theta$ and $|DA|\le C\Theta/r$.

\begin{lemma}\label{lem:regularity-end}
Under the preceding assumptions, the following statements hold.
\begin{enumerate}
\renewcommand{\theenumi}{\roman{enumi}}
\renewcommand{\labelenumi}{\textup{(\theenumi)}}
\item\label{item:regularity-end-positive}
Every weak solution $Y=A+w$, $w\in H^1_\delta$, of
$\Fop_\gamma(Y)=0$ is smooth and strictly positive on
$\overline{B_1}\setminus\{0\}$ and satisfies
\eqref{eq:boundary-matching}.
\item\label{item:regularity-end-weighted}
Let $Y>0$ be a smooth solution of $\Fop_\gamma(Y)=0$
on a punctured neighborhood of $0$.
If $Y-V_T\in H^1_\delta$ on that neighborhood, then
$Y-V_T\in\cC^{k,\alpha}_{\Theta;\delta}$ near $0$
for every fixed integer $k\ge0$ and $0<\alpha<1$.
\item\label{item:regularity-end-local-convergence}
\textup{Smooth convergence away from the puncture.}
Let $\gamma_j$ be smooth metrics on $\overline{B_1}$,
with $\dd V_{\gamma_j}=\dd x$, all equal to $\gamma_0$
on the fixed boundary collar.
Let $Y_j>0$ solve $\Fop_{\gamma_j}(Y_j)=0$
on $B_1\setminus\{0\}$ and satisfy
\eqref{eq:boundary-matching}.
Suppose that $\gamma_j\to\gamma$ smoothly and
$Y_j\to Y$ in $L^{p+1}$ on compact subsets of
$\overline{B_1}\setminus\{0\}$, where $Y>0$ is smooth.
Then $Y_j\to Y$ smoothly on the same compact subsets.
\item\label{item:regularity-end-end-convergence}
Under the assumptions of
\textup{(\ref{item:regularity-end-local-convergence})},
suppose that $\gamma$ and every $\gamma_j$ are Euclidean
on a fixed punctured ball $0<|x|<r_0$ and that
$Y-V_T\in H^1_\delta(0<|x|<r_0)$.
If, on this punctured ball,
\begin{equation}\label{eq:end-convergence-data}
Y_j=Y+V_{T_j}-V_T+w_j,\qquad
|T_j-T|+\|w_j\|_{H^1_\delta(0<|x|<r_0)}\to0,
\end{equation}
then $Y_j\to Y$ in $\cC^k_{\Theta;0}(0<|x|<r_0)$
for every fixed integer $k\ge0$.
\end{enumerate}
\end{lemma}

\begin{proof}
\step{Step 1. Positivity and the boundary identification.}
On $E=\{Y<0\}$, one has $|w|\ge A\ge c\Theta$. Thus
$Y_-\le|w|$ and
\[
|DY_-|\le\mathbf1_E(|DA|+|Dw|)
\le C r^{-1}|w|+|Dw|.
\]

These bounds give $Y_-\in H^1_\delta$ with even boundary trace.
Testing with radial cutoffs of $Y_-$ at radii $\varepsilon$ and
$2\varepsilon$, the cutoff errors are bounded by
\[
C\int_{B_{2\varepsilon}\setminus B_\varepsilon}
\bigl(|Dw|^2+r^{-2}|w|^2\bigr)\dd x\to0.
\]

Hence $Q_\gamma(Y_-,Y_-)=0$, and coercivity gives $Y_-=0$.
Also $Y\not\equiv0$: otherwise $w=-A$ would violate
$\int r^{-2}w^2<\infty$, since $A\ge c r^{-(n-2)/2}$.
In the boundary collar, the weak form \eqref{eq:weak-robin-form}
is precisely the weak matching condition for the Kelvin extension.
The Kelvin extension therefore satisfies the equation weakly
across the sphere. Applying the following interior estimates in the
extended collar charts yields smoothness and strict positivity there.

\step{Step 2. Local smoothness.}
On a compact coordinate ball, $n(n-2)Y^{p-1}\in L^{n/2}$.
For $M>0$, put $Y_M=\min\{Y,M\}$ and test with
$\eta^2Y Y_M^{2m-2}$, where $m\ge1$ and $\eta$ is a smooth cutoff supported in that ball.
Applying the local Sobolev inequality and Young's inequality, we obtain
\[
\|\eta Y Y_M^{m-1}\|_{L^{p+1}(X,\gamma)}^2
\le Cm^2\int_X\bigl[|\nabla\eta|_\gamma^2+
(1+n(n-2)Y^{p-1})\eta^2\bigr]Y^2Y_M^{2m-2}\dd V_\gamma.
\]

Split the potential into a bounded part and an $L^{n/2}$ tail with norm
less than $(2Cm^2)^{-1}$. H\"older's inequality absorbs the tail.
Letting $M\to\infty$ shows that local $L^{2m}$ integrability implies
local $L^{2mn/(n-2)}$ integrability on a smaller ball. Starting with
$Y\in L^{p+1}_{\rm loc}$, iteration proves local integrability to every finite exponent.
Interior $W^{2,m}$ estimates imply $C^{1,\alpha}$ regularity.
The equation then implies $C^{2,\alpha}$ regularity, and the strong
maximum principle establishes $Y>0$. The power nonlinearity is now smooth on each compact set,
and repeated Schauder estimates prove smoothness.

\step{Step 3. Local bounds from small $L^{p+1}$ integrals.}
For a family of charts with uniform ellipticity and $C^2$ bounds there
are $r_*,\varepsilon_0,C>0$ such that positive solutions satisfy
\begin{equation}\label{eq:epsilon-regularity}
\int_{B_{2r}(x)}Y^{p+1}\dd V_\gamma<\varepsilon_0
\quad\Longrightarrow\quad
\sup_{B_r(x)}Y\le Cr^{-(n-2)/2},\qquad r\le r_*.
\end{equation}
After scaling to $B_2$, suppose a sequence $Y_j$ with metrics $\gamma_j$
violates a uniform bound on $B_1$. Choose $x_j$ maximizing
$(3/2-|x|)^{(n-2)/2}Y_j(x)$ on $B_{3/2}$, and let $\sigma_j=3/2-|x_j|$, $\mu_j=Y_j(x_j)^{-2/(n-2)}$.
Then $\sigma_j/\mu_j\to\infty$ and
\[
\widetilde Y_j(y)=\mu_j^{(n-2)/2}Y_j(x_j+\mu_jy)
\]
satisfies
\[
\widetilde Y_j(0)=1,\qquad
0<\widetilde Y_j\le2^{(n-2)/2}
\quad\text{on }B_{\sigma_j/(2\mu_j)}.
\]

The rescaled coefficients are uniformly elliptic, the right sides are
bounded, and the rescaled scalar curvature coefficient is
$c(n)\mu_j^2\Scal_{\gamma_j}(x_j+\mu_jy)$.
By interior estimates, there is a fixed $\rho>0$ such that
$\widetilde Y_j\ge1/2$ on $B_\rho$. Critical scaling yields
\[
\int_{B_2}Y_j^{p+1}\dd V_{\gamma_j}
\ge\int_{B_{\rho\mu_j}(x_j)}Y_j^{p+1}\dd V_{\gamma_j}
\ge c\int_{B_\rho}\widetilde Y_j^{p+1}\dd y\ge c'>0.
\]
Choosing $\varepsilon_0<c'$ proves \eqref{eq:epsilon-regularity}.

\step{Step 4. Local convergence.}
Let $Y_j\to Y>0$ in $L^{p+1}_{\rm loc}$ and fix a compact set $K_0\subset\overline{B_1}\setminus\{0\}$.
Choose finitely many coordinate balls $B_{2r}(x)$ whose concentric balls
$B_r(x)$ cover $K_0$ and satisfy
$\int_{B_{2r}(x)}Y^{p+1}\dd V_\gamma<\varepsilon_0/2$.
For large $j$, we have
$\int_{B_{2r}(x)}Y_j^{p+1}\dd V_{\gamma_j}<\varepsilon_0$
on each chosen ball. By \eqref{eq:epsilon-regularity} and elliptic compactness,
\[
\|Y_j-Y\|_{C^0(K_0)}\to0
\qquad\text{and}\qquad
\inf_{K_0}Y_j\ge\tfrac12\inf_{K_0}Y>0.
\]
Passing to the limit in the weak equations gives
$-L_\gamma Y=n(n-2)Y^p$. The coefficients converge smoothly. Applying interior estimates to
\[
-L_{\gamma_j}(Y_j-Y)
=n(n-2)(Y_j^p-Y^p)+(L_{\gamma_j}-L_\gamma)Y,
\]
and using the positive lower bound, we obtain
$\|Y_j-Y\|_{C^k(K_0)}\to0$ for every fixed $k$.

\step{Step 5. Weighted estimates near the puncture.}
To prove the local regularity assertion, let $Y>0$ be a smooth
solution of $\Fop_\gamma(Y)=0$ on a punctured neighborhood of $0$,
with $w_0:=Y-V_T\in H^1_\delta$ there.
Choose $r_0>0$ so that $\gamma=\gamma_0$ and this membership holds
on $0<r<r_0$.
For $0<\rho<r_0/4$, on
$\mathcal A^*=B_4\setminus\overline{B_{1/4}}$ set
\[
Y_\rho(y)=\rho^{(n-2)/2}Y(\rho y),\quad
W_\rho(y)=\rho^{(n-2)/2}w_0(\rho y),\quad
P_\rho(y)=|y|^{-(n-2)/2}v_*(-\log\rho-\log|y|+T).
\]

The profiles $P_\rho$ have positive lower bounds and uniform bounds
for all derivatives on $\mathcal A^*$. Scaling the weighted norm gives
\[
\|W_\rho\|_{H^1(\mathcal A^*)}
+\|W_\rho\|_{L^{p+1}(\mathcal A^*)}
\le C\rho^\delta\|w_0\|_{H^1_\delta(0<r<r_0)}.
\]

Cover a slightly smaller annulus by finitely many balls whose
doubled balls lie in $\mathcal A^*$. Choose their radii so that the
integral of $P_\rho^{p+1}$ on each doubled ball is less than
$\varepsilon_0/2^{p+2}$, uniformly in $\rho$.
The displayed norm estimate gives
$\int_{\mathcal A^*}|W_\rho|^{p+1}\dd y<\varepsilon_0/2^{p+2}$
for small $\rho$. Since $Y_\rho=P_\rho+W_\rho$ and
$Y_\rho^{p+1}\le2^p(P_\rho^{p+1}+|W_\rho|^{p+1})$,
\eqref{eq:epsilon-regularity} bounds $Y_\rho$ uniformly on the
original balls. Since both $Y_\rho$ and $P_\rho$ solve the flat equation,
\begin{equation}\label{eq:end-secant}
\left[-\Delta_y-n(n+2)\int_0^1
       (P_\rho+\xi W_\rho)^{p-1}\dd\xi\right]W_\rho=0.
\end{equation}

The coefficient is bounded on a larger annulus than
$\mathcal A=B_2\setminus\overline{B_{1/2}}$. The local boundedness
estimate for this linear equation gives
\[
\|W_\rho\|_{C^0(\mathcal A)}
\le C\|W_\rho\|_{L^2(\mathcal A^*)}
\le C\rho^\delta\|w_0\|_{H^1_\delta}.
\]
Consequently $Y_\rho$ is uniformly positive for small $\rho$.
On nested annuli, the equation for $Y_\rho$ gives uniform smooth
bounds, and \eqref{eq:end-secant} then gives
$\|W_\rho\|_{C^{k,\alpha}(\mathcal A)}\le C_k\rho^\delta$.
By the definition of the weighted norm,
$w_0\in\cC^{k,\alpha}_{\Theta;\delta}$.

\step{Step 6. Uniform weighted convergence.}
For the limit $Y$ in \eqref{eq:end-convergence-data}, set
$w_0:=Y-V_T$ near $0$.
By Step 5,
$w_0\in\cC^{k,\alpha}_{\Theta;\delta}$ near $0$
for every fixed $k\ge0$ and $0<\alpha<1$.
Put
$A_j:=Y+V_{T_j}-V_T=V_{T_j}+w_0$.
On a sufficiently small fixed punctured ball, $A_j/\Theta$
has uniform positive upper and lower bounds, and all its
scaled derivatives are uniformly bounded for large $j$.
Scaling the weighted norm gives
\[
\|\rho^{(n-2)/2}w_j(\rho\,\cdot)\|_{L^{p+1}(\mathcal A^*)}
\le C\rho^\delta\|w_j\|_{H^1_\delta(0<r<r_0)}.
\]

The normalized functions $\rho^{(n-2)/2}A_j(\rho\,\cdot)$ are
uniformly bounded on $\mathcal A^*$. Applying the argument in
Step 5 to $Y_j=A_j+w_j$, we cover a neighborhood of
$\overline{\mathcal A}$ by finitely many balls whose doubled balls
lie in $\mathcal A^*$. The integral of
$|\rho^{(n-2)/2}Y_j(\rho y)|^{p+1}$ over each doubled ball is less
than $\varepsilon_0$ for large $j$ and small $\rho$.
Equation~\eqref{eq:epsilon-regularity} therefore
bounds $\rho^{(n-2)/2}Y_j(\rho\,\cdot)$ uniformly for large $j$.
The difference $w_j=Y_j-A_j$ satisfies
\[
\left[-\Delta-n(n+2)\int_0^1(A_j+\xi w_j)^{p-1}\dd\xi\right]w_j
=-\Fop_{\gamma_0}(A_j).
\]

Since $Y,V_{T_j},V_T$ solve their equations,
\[
\Fop_{\gamma_0}(A_j)=-n(n-2)
\bigl[(V_{T_j}+w_0)^p-V_{T_j}^p-(V_T+w_0)^p+V_T^p\bigr].
\]
After dividing the four powers by $\Theta^p$, apply the mean
value theorem first to the increment $w_0/\Theta$ and then
to the Fowler phase between $T$ and $T_j$.
This gives the factor $|T_j-T|$, while every spatial derivative
still contains $w_0$ or one of its derivatives.  Since $\Theta^{p-1}=r^{-2}$,
\[
\|r^2\Fop_{\gamma_0}(A_j)\|_{\cC^{k,\alpha}_{\Theta;\delta}}
\le C_k|T_j-T|\to0.
\]
The factor $r^2$ accounts for the order of the differential operator.
The inhomogeneous interior estimate on each rescaled annulus yields
\[
\|w_j\|_{\cC^0_{\Theta;\delta}}
\le C\left(\|w_j\|_{H^1_\delta}
       +\|r^2\Fop_{\gamma_0}(A_j)\|_{\cC^0_{\Theta;\delta}}\right)\to0
\]
on a smaller fixed punctured ball. Uniform positivity and repeated
interior estimates give $\|w_j\|_{\cC^k_{\Theta;\delta}}\to0$ for each fixed $k$.
Since $\|V_{T_j}-V_T\|_{\cC^k_{\Theta;0}}\le C_k|T_j-T|$,
$\|Y_j-Y\|_{\cC^k_{\Theta;0}}\to0$ there. The compact convergence in Step 4
covers the remaining annulus and proves the assertion for $0<r<r_0$.
\end{proof}

\subsection*{AI assistance and interactions}

This paper grew out of a project that began in mid July 2026 with an investigation of whether the main conclusion of \cite{HanXiongZhang2023} could fail without the assumption that the metric has flatness order at least \((n-2)/2\). We made progress gradually, first in dimension seven, then in dimensions up to 10,  then  in dimensions up to 24; and found by early August that the threshold dimension is 30:  the upper bound holds in dimensions \(7\le n \le 29\) and there are counter-examples to the upper bound in dimensions \(n\ge 30\). By August 9, we had prepared initial drafts of the proofs of the classification and upper bound theorems in dimensions \(7\le n \le 29\), whose final statements and proofs appear in \cite{YamabePartI},  and an exploratory draft of a possible counterexample in dimension \(30\). We then began an extensive verification, correction, addition, and revision process. 
In particular, the initial draft counterexample defined an incomplete conformal metric. During this process, we found new counterexamples in all dimensions \(\ge 30\) with a complete conformal metric, as presented in this paper. 

In this process, the coauthors used ChatGPT  and Codex as interactive research-assistance tools for literature reading, exploratory proof searches, calculations, checking arguments, and manuscript preparation. The systems responded to specific questions and directions from the coauthors. The coauthors determined the research questions and proof strategies, supplied and verified the citations and attributions, and selected, checked, and substantially revised every adopted argument. This process involved repeated exchanges among the coauthors and the systems, dozens of manuscript versions, and extensive human verification and rewriting. The coauthors are responsible for all mathematical claims, citations, and final content. 

\subsection*{Conflict of interest} 

The authors declare that they have no conflicts of interest.

\bigskip
\begingroup
\small
\Needspace{7\baselineskip}

\noindent\textsc{Zheng-Chao Han}\\
Department of Mathematics, Rutgers University,\\
Hill Center--Busch Campus, 110 Frelinghuysen Road,\\
Piscataway, NJ 08854, USA\\
\textit{Email address:}
\href{mailto:zchan@math.rutgers.edu}{\nolinkurl{zchan@math.rutgers.edu}}

\medskip
\Needspace{6\baselineskip}

\noindent\textsc{Qinfeng Jiang}\\
School of Mathematical Sciences, Beijing Normal University,\\
Beijing 100875, China\\
\textit{Email address:}
\href{mailto:202531130031@mail.bnu.edu.cn}%
{\nolinkurl{202531130031@mail.bnu.edu.cn}}

\medskip
\Needspace{7\baselineskip}

\noindent\textsc{Hua-Yang Wang}\\
School of Mathematical Sciences, Laboratory of Mathematics and Complex
Systems, MOE,\\
Beijing Normal University, Beijing 100875, China\\
\textit{Email address:}
\href{mailto:wanghuayang@amss.ac.cn}{\nolinkurl{wanghuayang@amss.ac.cn}}

\medskip
\Needspace{10\baselineskip}

\noindent\textsc{Jingang Xiong}\\
School of Mathematical Sciences, Laboratory of Mathematics and Complex
Systems, MOE,\\
Beijing Normal University, Beijing 100875, China;\\
Center for Basic Mathematics, Institute for Advanced Study,\\
Beijing Normal University, Beijing 100875, China\\
\textit{Email address:}
\href{mailto:jx@bnu.edu.cn}{\nolinkurl{jx@bnu.edu.cn}}

\medskip
\Needspace{7\baselineskip}

\noindent\textsc{Lei Zhang}\\
Department of Mathematics, University of Florida,\\
358 Little Hall, P.O. Box 118105,\\
Gainesville, FL 32611-8105, USA\\
\textit{Email address:}
\href{mailto:leizhang@ufl.edu}{\nolinkurl{leizhang@ufl.edu}}
\endgroup

\end{document}